\documentclass[11pt,a4paper,reqno]{amsart}
\fontsize{6.5pt}{6.5pt}\selectfont 
\makeatletter
\def\@evenhead{{\fontsize{6.5pt}{6.5pt}\selectfont \hfil \leftmark\hfil\thepage}}
\def\@oddhead{{\fontsize{6.5pt}{6.5pt}\selectfont\hfil\rightmark \hfil\thepage}}
\makeatother
\usepackage{amsfonts}
\usepackage{booktabs}
\usepackage{amsmath,amssymb,amsthm,amsxtra}
\usepackage{float}
\usepackage{amssymb}
\usepackage{mathabx}
\usepackage{bbm}
\usepackage[colorlinks, linkcolor=blue,anchorcolor=Periwinkle, citecolor=Red,  urlcolor=Emerald]{hyperref}
\usepackage[usenames,dvipsnames]{xcolor}
\usepackage{enumitem}
\usepackage{geometry,array} 
\usepackage{graphicx}
\usepackage{subfigure}
\usepackage{bookmark}
\usepackage{tikz}
\usetikzlibrary{matrix,intersections, calc, arrows.meta}
\usepackage{url}
\usepackage{dsfont}
\usepackage[colorinlistoftodos]{todonotes}
\usepackage{tablists} \restorelistitem
\usepackage{bm}

\usetikzlibrary{decorations.markings}
\tikzset{->-/.style={decoration={  markings,  mark=at position #1 with
    {\arrow{>}}},postaction={decorate}}}
\tikzset{-<-/.style={decoration={  markings,  mark=at position #1 with
    {\arrow{<}}},postaction={decorate}}}
\usepackage{extarrows}
\usepackage{setspace}
\usepackage{thmtools}
\usepackage{thm-restate}
\usepackage{hyperref}
\usepackage{cleveref}
\usepackage{multirow}
\usepackage{ifthen}
\usepackage{caption}
\usepackage{tikz-3dplot}
\usepackage{comment}
\usepackage{mathrsfs}
\usepackage{subfiles}
\usepackage{placeins}
\usepackage{aligned-overset}

\usepackage{pgfplots}
\pgfplotsset{compat=1.18}

\newcommand{\tildelim}{\mathop{\widetilde{\lim}}}

\renewcommand{\dim}{\operatorname{dim}}

\newcommand{\mfu}{\mathbf{u}}

\newcommand{\mfw}{\mathbf{w}}

\newcommand{\mcC}{\mathcal{C}}

\newcommand{\mcT}{\mathcal{T}}

\newcommand{\mbR}{\mathbb{R}}

\newcommand{\mbZ}{\mathbb{Z}}

\newcommand{\zeroregion}[1]{\mathcal{H}_D\left(#1\right)}
\newcommand{\positiveregion}[1]{\mathcal{H}^{+}_D\left(#1\right)}
\newcommand{\negativeregion}[1]{\mathcal{H}^{-}_D\left(#1\right)}
\newcommand{\positiveclosure}[1]{\overline{\mathcal{H}}_D^{+}\left(#1\right)}
\newcommand{\negativeclosure}[1]{\overline{\mathcal{H}}_D^{-}\left(#1\right)}

\newcommand{\Ezeroregion}[1]{\mathcal{H}\left(#1\right)}
\newcommand{\Epositiveregion}[1]{\mathcal{H}^{+}\left(#1\right)}
\newcommand{\Enegativeregion}[1]{\mathcal{H}^{-}\left(#1\right)}

\theoremstyle{plain}
\newtheorem{theorem}{Theorem}[section]

\newtheorem{lemma}[theorem]{Lemma}
\newtheorem{corollary}[theorem]{Corollary}
\newtheorem{proposition}[theorem]{Proposition}
\newtheorem{conjecture}[theorem]{Conjecture}
\theoremstyle{definition}
\newtheorem{definition}[theorem]{Definition}

\newtheorem{example}[theorem]{Example}

\newtheorem{remark}[theorem]{Remark}

\numberwithin{equation}{section}
\newtheorem{definition-proposition}[theorem]{Definition-Proposition}

\begin{document}

\title{Geometric structures of $G$-fans associated with rank $3$ cluster-cyclic exchange matrices}

\author{Ryota Akagi}
\address{Graduate School of Mathematics\\ nagoya University\\Chikusa-ku\\ Nagoya\\464-8602\\ Japan.}
\email{ryota.akagi.e6@math.nagoya-u.ac.jp}

\author{Zhichao Chen}
\address{School of Mathematical Sciences\\ University of Science and Technology of China \\ Hefei, Anhui 230026, P. R. China.}
\email{czc98@mail.ustc.edu.cn}

\maketitle

\begin{abstract} In this paper, we investigate the geometric structures of $G$-fans associated with rank $3$ real cluster-cyclic exchange matrices. In this class, a simple recursion for tropical signs was found, which enables us to study the detailed properties of $c$-, $g$-vectors. We introduce two kinds of upper bounds of the $G$-fans. The first one is the global upper bound, which comes from a hyperbolic surface containing all $g$-vectors after an initial mutation. The second one is the local upper bound, which reflects the internal separateness structure. As applications, we prove that there is no periodicity among $g$-vectors, and we completely determine the sign of $g$-vectors. We also prove the monotonicity of $g$-vectors under the minimum assumption. Moreover, we show that the three global upper bounds can be simplified to a single uniform upper bound.
\\\\
Keywords: $G$-fan, global upper bound, local upper bound, periodicity, monotonicity.
\\
2020 Mathematics Subject Classification: 13F60, 05E10, 05E45.
\end{abstract}

\tableofcontents
\begin{figure}[htbp]
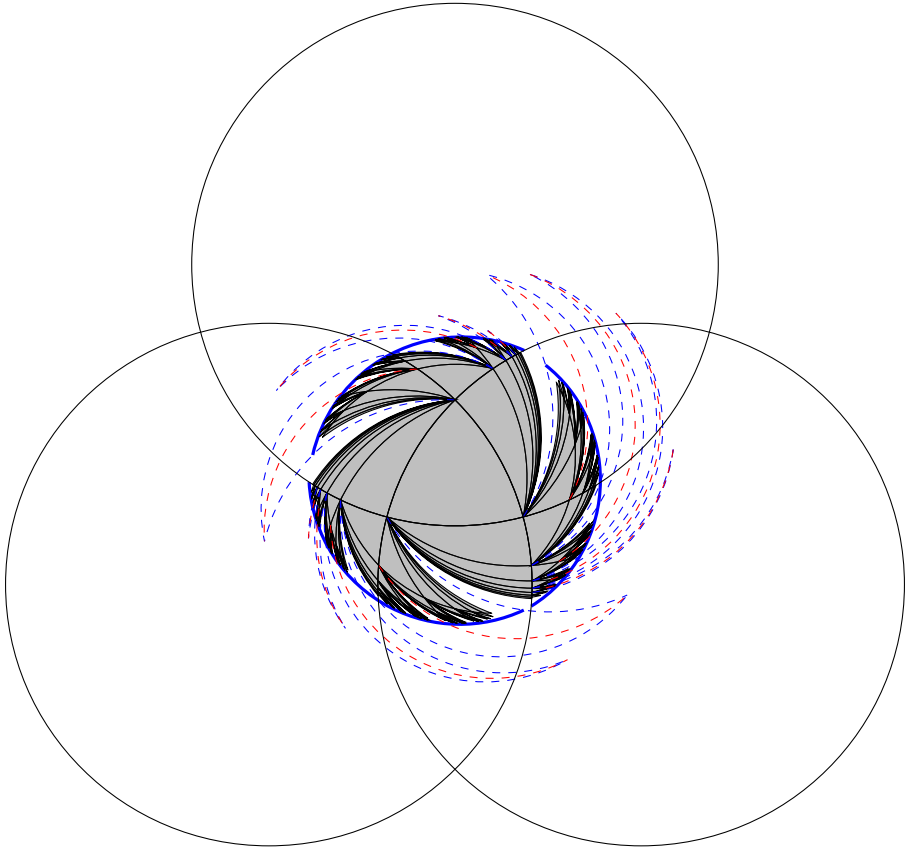

\centering
\subfile{figures/intro_G_fan_hyperbolic_case/intro_G_fan_hyperbolic_case}
\caption{$G$-fan corresponding to a cluster-cyclic exchange matrix.}
\label{fig: main graph}
\end{figure}

\section{Introduction}\label{intro}
\subsection{Background}
Cluster algebras are a class of commutative rings equipped with cluster variables, which are rational functions, and mutations, which are their transformation rules. This concept was introduced by \cite{FZ02} and has been developed by many researchers. Nowadays, related structures have been discovered in a wide range of areas in mathematics.
\par
It is known that cluster algebras possess many interesting combinatorial structures. One of the most important objects is the {\em $g$-vectors}. Roughly speaking, these are vectors defined by a certain ``degree" of cluster variables. Although $g$-vectors are defined using the specific information from cluster variables, it is a surprising fact that they inherit information about the periodicity of cluster variables \cite{CKLP13, CL20, Nak23}. Thus, they provide a natural parametrization of cluster variables.
\par
It is also known that the $g$-vectors form an important combinatorial structure called the {\em $G$-fans} (or the {\em $g$-vector fans}), which are simplicial fans in a Euclidean space. This is a geometric realization of the cluster complex in \cite{FZ07}. 
The existence of this concept was conjectured by \cite{FZ07}, which was eventually proved by \cite{GHKK18}, and it has been studied by many researchers (e.g. \cite{Rea14, Nak23, Nak26, Yur25}). This concept is strongly related to the {\em pointed bases}, which are a class of bases parametrized by integer vectors. According to \cite{Qin24}, such bases must contain all cluster monomials, and each pair of pointed bases can be transformed via certain linear combinations. Since all cluster monomials are parametrized by the points in the $G$-fan, the problem of constructing a pointed basis can be reduced to the problem of adding elements parametrized by points in the compliment of the $G$-fan. Recently, from this point of view, the {\em dominance region} was defined by using the $G$-fan, and studied in \cite{RS23, RRS25}.
\par
Through the {\em cluster scattering diagrams} and the {\em theta bases} in \cite{GHKK18}, the $G$-fans play essential roles in the construction of a pointed basis. The cluster scattering diagrams are wall-crossing structures that naturally contain the $G$-fans. The theta basis is a basis of a certain algebra that contains cluster algebras. The construction of the theta basis strongly depends on the geometry of the cluster scattering diagrams. Hence, it is a central topic in cluster algebra theory to understand the $G$-fans.
\par
Although the $G$-fans have important information of cluster algebras, it seems that they appear to be highly complex and varied. Due to this complexity, there has been limited progress in the explicit understanding. In some specific cases, there are several nice results. For example, for the rank $2$ cluster algebras, all the $G$-fans are explicitly understood in \cite{Rea14, AC25a}, see \Cref{fig: rank 2 G fans}. Also, the relationships between the $G$-fans (or scattering diagrams) and the roots of Kac-Moody Lie algebras were studied in \cite{RS16, RS18, Rea20a, Rea20b, RS22, Yur23} for cases with acyclic and affine initial exchange matrices, and in \cite{Yur20} for the case arising from the surface.
Based on the observations from these results, it seems quite difficult to connect the known theory with the $G$-fans if the initial exchange matrix has an oriented cycle. Recently, in \cite{Nak26}, several prototypical examples of rank 3 $G$-fans were presented, along with a conjecture for the classification of $G$-fans.
\begin{figure}[htbp]
\centering
\begin{minipage}{0.3\textwidth}
\centering
\begin{tikzpicture}
\draw (0,0)--(1,0);
\draw (0,0)--(0,1);
\draw (0,0)--(-1,0);
\draw (0,0)--(0,-1);
\draw (0,0)--(1,-1);
\end{tikzpicture}
\end{minipage}
\begin{minipage}{0.3\textwidth}
\centering
\begin{tikzpicture}
\draw (0,0)--(1,0);
\draw (0,0)--(0,1);
\draw (0,0)--(-1,0);
\draw (0,0)--(0,-1);
\foreach \x in {1,2,3,...,20}
    {
    \draw (0,0)--(1, {-\x/(\x+1)});
    \draw (0,0)--({\x/(\x+1)},-1);
    }
\draw[dashed] (0,0)--(1,-1); 
\end{tikzpicture}
\end{minipage}
\begin{minipage}{0.3\textwidth}
\centering
\begin{tikzpicture}
\draw (0,0)--(1,0);
\draw (0,0)--(0,1);
\draw (0,0)--(-1,0);
\draw (0,0)--(0,-1);
\foreach \x in {2,3,...,10}
    {
    \draw (0,0)--(1, {-\x/(\x+1)+0.3});
    \draw (0,0)--({\x/(\x+1)-0.3},-1);
    }
\fill[gray, opacity=0.5] (0,0)--(10/11-0.3,-1)--(1,-1)--(1,-10/11+0.3)--cycle;
\end{tikzpicture}
\end{minipage}
\caption{Rank $2$ $G$-fans. The corresponding initial exchange matrices are $\left(\begin{smallmatrix}
0 & -1\\
1 & 0
\end{smallmatrix}\right)$, $\left(\begin{smallmatrix}
0 & -2\\
2 & 0
\end{smallmatrix}\right)$, and $\left(\begin{smallmatrix}
0 & -p\\
p & 0
\end{smallmatrix}\right)$ ($p > 2$), respectively.}
\label{fig: rank 2 G fans}
\end{figure}

\subsection{Purpose and related works}
In this paper, we investigate the structure of the $G$-fan associated with the {\em cluster-cyclic} exchange matrices of rank $3$. This is an exchange matrix whose mutation equivalent exchange matrices all have an oriented cycle. As stated above, this class is rather difficult to study via existing theories. On the other hand, this class has an important example called the {\em Markov quiver}, see \eqref{eq: Markov matrix}. Hence, it is worthwhile to  explore a new approach to this class.
\par
Although it is hard to establish the connections between this class and another theory, several unique and interesting phenomena have already been found. Firstly, it is known that the classification of these exchange matrices can be done via the {\em Markov constant} \cite{BBH11, Aka24}, see \Cref{prop: ordinary classification of cluster-cyclicity}. In \cite{Sev12}, another classification was provided by the distribution of eigenvalues of the corresponding pseudo Cartan companion, and we reveal its detailed property in this paper, see \Cref{lem: eigenvalues lemma for cluster-cyclic matrix}. Secondly, in \cite{LL24}, it was shown that all $g$-vectors are contained in a quadratic surface arising from the pseudo Cartan companion. We refine this theorem in \Cref{thm: global upper bound theorem}. Lastly, and most importantly, in \cite{AC25b}, the sign-coherence of $c$-vectors was shown even when {\em real} entries are allowed, and a recursion for their signs was obtained, see \Cref{thm: recursion for tropical signs}. 
\par
The example of the $G$-fan or the scattering diagram associated with the Markov quiver is presented in \cite{Cha12, FG16, Rea23}.
In \cite[Section $6$]{Nak26}, it was conjectured that the class of cluster-cyclic exchange matrices provides one class for the classification of the $G$-fans. Although we have not got the explicit definition how to classify the $G$-fans, our results provide an affirmative answer to this conjecture.

\subsection{Conventions}\label{subsec: conv}
A sequence $\mathbf{w}=[k_1,\dots,k_r]$ of $1,2,3$ is said to be {\em reduced} if $k_i \neq k_{i+1}$ for any $i=1,2,\dots,r-1$. By convention, the empty sequence $\emptyset = [\ ]$ is also reduced. We write the set of all reduced sequences by $\mathcal{T}$. Then, for each $j=1,2,3$ and $\mathbf{w} \in \mathcal{T}$, we can assign $c$-vectors $\mathbf{c}_{j}^{\mathbf{w}}$ and $g$-vectors $\mathbf{g}_j^{\mathbf{w}}$ following the standard definition, see \Cref{def: C G matrices}.
\par
For any $\mathbf{w}=[k_1,\dots,k_r] \in \mathcal{T}$ and $l=1,2,3$, define $\mathbf{w}[l] \in \mathcal{T}$ as follows: $\mathbf{w}[l]=[k_1,\dots,k_r,l]$ if $k_r \neq l$, and $\mathbf{w}[k]=[k_1,\dots,k_{r-1}]$ if $k_r = l$. Repeating this process, we define the product of two reduced sequences $\mathbf{w}[l_1,\dots,l_m]=\mathbf{w}[l_1][l_2]\dots[l_m]$.
\par
We define the partial order $\leq$ on $\mathcal{T}$ as follows:
\begin{equation}
[k_1,\dots,k_r] \leq [l_1,\dots,l_{r'}] \Longleftrightarrow\ \textup{$r \leq r'$ and $k_i=l_i$ for any $i=1,\dots,r$}.
\end{equation}
For any $\mathbf{w} \in \mathcal{T}$, define the subset $\mathcal{T}^{\geq \mathbf{w}}$ of $\mcT$ by 
\begin{equation}\label{eq: after subset}
\mathcal{T}^{\geq \mathbf{w}}=\{\mathbf{u} \in \mathcal{T} \mid \mathbf{u} \geq \mathbf{w}\}.
\end{equation}
\par
We specify the following points, which slightly differ from conventions used in previous works:
\begin{itemize}
\item {\em Generalization to real entries}: Since the sign-coherence has been established for this class allowing for real entries, we handle the $G$-fan in this broader setting. This generalization was introduced in \cite{AC25a} with several conjectures, though those conjectures have been solved for the class of exchange matrices studied in this paper.
\item {\em Modified $c$-, $g$-vectors}: Instead of usual $c$-, $g$-vectors, we employ the {\em modified $c$-vectors} $\tilde{\mathbf{c}}_j^{\mathbf{w}}$ and the {\em modified $g$-vectors} $\tilde{\mathbf{g}}_{j}^{\mathbf{w}}$ introduced by \cite{AC25a}, see \Cref{def: modified c g vectors}.
They are defined by multiplying the original vectors by certain positive scalars; this modification does not affect the structure of the $G$-fan. However, to recover the results for the original $c$-, $g$-vectors, certain adjustments are required.
\item {\em $K$, $S$, $T$-labeling}: To indicate the mutation directions and indices, we use the symbol $K$, $S$, $T$ instead of $1$, $2$, $3$, see \Cref{def: K S T labeling}. This labeling is compatible with the recursion for the signs of $c$-vectors obtained in \cite{AC25b}.
\end{itemize}
To explain the main theorem, we introduce some notations. Let $\mathcal{M}$ be the free monoid generated by $S$ and $T$. We introduce the right monoid action of $\mathcal{M}$ on $\mathcal{T}\setminus\{\emptyset\}$ by \eqref{eq: monoid action definition}. Then, every nonempty reduced sequence $\mathbf{w} \in \mathcal{T}\setminus\{\emptyset\}$ is uniquely expressed as $\mathbf{w}=[i]X$ for some $i=1,2,3$ and $X \in \mathcal{M}$. See \Cref{def: K S T labeling} for the precise definition. In this expression, we have $\mathcal{T}^{\geq [i]}=\{[i]X \mid X \in \mathcal{M}\}$. We decompose this set into the following two parts:
\begin{itemize}
\item {\em Trunk}: $\mathcal{T}^{<[i]S^{\infty}}=\{[i]S^n \mid n \in \mathbb{Z}_{\geq 0}\}$.
\item {\em Maximal branches}: $\mathcal{T}^{\geq [i]S^nT}=\{[i]S^nTX \mid X \in \mathcal{M}\}$ with $n \in \mathbb{Z}_{\geq 0}$.
\end{itemize}
See \Cref{fig: tropical signs} as an example. For any $\mathbf{w}$ in a maximal branch, the set $\mathcal{T}^{\geq \mathbf{w}}$ is called a {\em branch}.

\subsection{Main results}
Let $B$ be a cluster-cyclic initial exchange matrix of rank $3$, and let $\Delta(B)$ be the $G$-fan associated with $B$. For any reduced sequence $\mathbf{w}$, we define the sub $G$-fan $\Delta^{\geq \mathbf{w}}(B)$, which consists of the cones spanned by $\tilde{\mathbf{g}}_1^{\mathbf{u}}$, $\tilde{\mathbf{g}}_2^{\mathbf{u}}$, and $\tilde{\mathbf{g}}_3^{\mathbf{u}}$ with $\mathbf{u} \in \mathcal{T}^{\geq {\mathbf{w}}}$. When $\mathcal{T}^{\geq \mathbf{w}}$ is a branch, the sub $G$-fan $\Delta^{\geq \mathbf{w}}(B)$ is also called a branch.
\par
The first main theorem introduces the {\em global upper bound}, which is an upper bound depending only on the initial mutation direction $i=1,2,3$. See the thick blue lines in \Cref{fig: main graph}. By \cite{LL24}, it is known that all modified $g$-vectors with the initial mutation direction $i$ are on a quadratic surface $H_i$ (\Cref{lem: surface for modified g-vectors}). Note that this quadratic surface is defined by the mutated exchange matrix $B^{[i]}=\mu_i(B)$, not the initial one $B$. We prove that this quadratic surface is a hyperboloid of two sheets, or a degeneration thereof, which consists of two connected components separated by a plane orthogonal to a sign-coherent vector (\Cref{lem: eigenvalues lemma for cluster-cyclic matrix}). We decompose $H_i$ into the ``positive" part $H^+_i$ and the ``negative" part $H^{-}_i$. See \Cref{fig: two plane} and \Cref{fig: hyperboloid} as examples. Let $Q^+_i$ be the minimum convex cone including $H^+_i$ and $\mathbf{0}$. See \Cref{def: for global upper bounds} for the precise definition. Then, the first main theorem is that $Q_i^{+}$ is an upper bound for $\Delta^{\geq [i]}(B)$.
\begin{theorem}[\Cref{thm: global upper bound theorem}]\label{thm: intro global upper bound theorem}
The following inclusion holds:
\begin{equation}
|\Delta^{\geq [i]}(B)| \subset Q_i^+.
\end{equation}
In particular, we have $|\Delta(B)| \subset Q_1^+ \cup Q_2^+ \cup Q_3^+$.
\end{theorem}
\par
The second main theorem introduces the {\em local upper bound} $\mathscr{V}^{\mathbf{w}}$, which serves as an upper bound for each branch $\Delta^{\geq \mathbf{w}}(B)$, see \Cref{def: local upper bound}. This set can be computed solely from local information at $\mathbf{w}$, such as the modified $g$-vectors $\tilde{\mathbf{g}}_1^{\mathbf{w}}$, $\tilde{\mathbf{g}}_{2}^{\mathbf{w}}$, $\tilde{\mathbf{g}}_3^{\mathbf{w}}$ and the exchange matrix $B^{\mathbf{w}}$. Let $\mathscr{V}_{\circ}^{\mathbf{w}}$ be the interior of $\mathscr{V}^{\mathbf{w}}$. Then, the following statements hold.
\begin{theorem}[\Cref{thm: local upper bound theorem}]
For any branch $\Delta^{\geq {\bf w}}(B)$, the following inclusion holds.
\begin{equation}
|\Delta^{\geq {\bf w}}(B)| \subset \mathscr{V}^{\bf w}.
\end{equation}
Moreover, about the next local upper bounds $\mathscr{V}^{\mathbf{w}S}$, $\mathscr{V}^{\mathbf{w}T}$, the following properties hold.
\\
\textup{($a$)} Their interiors $\mathscr{V}_{\circ}^{\mathbf{w}S}$ and $\mathscr{V}_{\circ}^{\mathbf{w}T}$ are separated by the plane $\zeroregion{\bar{\mathbf{c}}^{\mathbf{w}}}$.
\\
\textup{($b$)} The local upper bounds are monotonically decreasing $\mathscr{V}^{\mathbf{w}} \supset \mathscr{V}^{\mathbf{w}S}, \mathscr{V}^{\mathbf{w}T}$.
\end{theorem}
In \Cref{fig: main graph}, the blue dashed lines represent the local upper bounds for the roots of the maximal branches $\mathbf{w}=[i]S^nT$, while the red lines correspond to the planes $\zeroregion{\bar{\mathbf{c}}^{\mathbf{w}}}$ in this theorem. This theorem allows us to approximate the sub $G$-fan $\Delta^{\geq \mathbf{w}}(B)$, which consists of infinitely many cones, by only one simplicial cone $\mathscr{V}^{\mathbf{w}}$ (excluding certain boundaries). The claim ($b$) implies that the accuracy of this approximation improves with mutations. This situation is illustrated in \Cref{fig: stereo of Vw}.
\par
Based on these two theorems and the preparatory result in \Cref{prop: separateness of upper bounds}, we derive several properties of $g$-vectors. We say that an equality $\mathbf{g}_{l}^{\mathbf{w}}=\mathbf{g}_{m}^{\mathbf{u}}$ ($\mathbf{w},\mathbf{u} \in \mathcal{T}$, $l,m=1,2,3$) of $g$-vectors is {\em trivial} if $l=m$ and ${\bf u}={\bf w}[k_1,\dots,k_r]$ with $k_i \neq l$. Roughly speaking, {\em a nontrivial equality means periodicity among $g$-vectors with respect to the mutations}. The following theorem implies that there is no periodicity among $g$-vectors associated with a cluster-cyclic exchange matrix of rank $3$. 
\begin{theorem}[\Cref{thm: non periodicity}]
All equalities ${\bf g}^{\bf w}_{l}={\bf g}^{\bf u}_{m}$ of $g$-vectors are trivial.
\end{theorem}
Since the cluster variables and the $g$-vectors share the same periodicity, this theorem can be restated as the one for cluster variables, see \Cref{cor: no periodicity among cluster variables}.
\par
Next application is for the signs of $g$-vectors. In the ordinary cluster algebra theory, it is known that every $G$-cone is a subset of one orthant \cite{GHKK18}. This property is called the {\em sign-coherence of $G$-matrices}, which implies that every row vector of each $G$-matrix has the same sign, and for the cluster-cyclic exchange matrix of rank $3$, it still holds even when we consider the generalization to real entries \cite{AC25b}. The sign of the $j$th row of $G^{\mathbf{w}}$ is denoted by $\tau_j^{\mathbf{w}} \in \{\pm 1\}$. In fact, as expected from \Cref{fig: main graph}, the explicit structure of these signs can be completely expressed.
\begin{theorem}[\Cref{thm: sign of G matrix}]
The signs of $G$-matrices are given in \Cref{tab: list of G-signs}.
\end{theorem}
Later, we obtain some specific phenomena under the {\em minimum assumption}, which means that the initial exchange matrix is the minimum element in its mutation equivalence class. It is known that there is the minimum element in the mutation equivalence class of each {\em integer} cluster-cyclic exchange matrix. See \Cref{lem: the condition for minimum assumption} for the precise statement. Hence, it is worth considering the structure under this assumption. Firstly, we establish the monotonicity of $g$-vectors.
\begin{theorem}[\Cref{thm: monotonicity of g-vectors}]
Under the minimum assumption, for any $j=1,2,3$ and $\mathbf{u}, \mathbf{w} \in \mathcal{T}$, if $\mathbf{w} \leq \mathbf{u}$, it holds that
\begin{equation}
\mathrm{abs}({\bf g}_{j}^{\bf w}) \leq \mathrm{abs}({\bf g}_{j}^{\bf u}),
\end{equation}
where $\mathrm{abs}(\mathbf{g})=(|g_1|,|g_2|,|g_3|)^{\top}$ for $\mathbf{g}=(g_1,g_2,g_3)^{\top} \in \mathbb{R}^3$.
\end{theorem}
Recall the global upper bounds $Q_i$ in \Cref{thm: intro global upper bound theorem}. As stated before \Cref{thm: intro global upper bound theorem}, this upper bound depends on the initial mutation direction $i=1,2,3$. On the other hand, we can also define $Q_{\mathrm{initial}}^{+}$ as the same way with respect to the initial exchange matrix. In general, this is not an upper bound of the $G$-fan. However, under the minimum assumption, we can obtain the following inclusion.
\begin{theorem}[\Cref{thm: simplified global upper bound}]
Under the minimum assumption, the following inclusion holds.
\begin{equation}
|\Delta(B)| \subset Q^{+}_{\mathrm{initial}}.
\end{equation}
\end{theorem}

\subsection{Structure of this paper}
This paper is organized as follows.
\par
In \Cref{pre}, we recall the basic notations for the $G$-fans. In \Cref{sec: cluster-cyclic framework}, we recall the recursion for the signs of $c$-vectors (\Cref{thm: recursion for tropical signs}), and simplify some notations by using the specific structure of the rank $3$ cluster-cyclic exchange matrices. In \Cref{sec: global}, we introduce the global upper bounds (\Cref{thm: global upper bound theorem}). In \Cref{sec: S-mutations}, we focus on the properties of $S$-mutations. In \Cref{sec: trunks}, we give the explicit expressions in trunks. In \Cref{sec: support of cluster-cyclic G-fan}, we introduce the local upper bounds (\Cref{thm: local upper bound theorem}). In \Cref{sec: separateness}, we prove the separateness among local upper bounds. In \Cref{sec: application}, we give some applications about exhibiting the non-periodicity (\Cref{thm: non periodicity}) and the signs of $G$-matrices (\Cref{thm: sign of G matrix}). In \Cref{sec: under the minimum assumption}, we show some properties under the minimum assumption, including the monotonicity of $g$-vectors (\Cref{thm: monotonicity of g-vectors}) and the simplification of the global upper bound (\Cref{thm: simplified global upper bound}).

\section{Preliminaries}\label{pre}
In this paper, we always consider the matrices associated with rank $3$ cluster algebras unless stated otherwise, whose entries are usually integers but are allowed to be real numbers in more general cases. Following \cite{Nak23}, we also introduce the basic notations for the $G$-fans.
\subsection{$B$-, $C$-, $G$-matrices}
A matrix $B \in \mathrm{M}_{3}(\mathbb{R})$ is said to be {\em skew-symmetrizable} if there exists a positive diagonal matrix $D=\mathrm{diag}(d_1,d_2,d_3)$ ($d_1,d_2,d_3>0$) such that $DB$ is skew-symmetric. For a $3 \times 3$ matrix $B$, it is known that this condition is equivalent to the {\em sign-skew-symmetry} ($\mathrm{sign}(b_{ij})=-\mathrm{sign}(b_{ji})$ for any $i,j$) and the following equality \cite[Lem.~7.4]{FZ03}:
\begin{equation}\label{eq: rank 3 skew-symmetrizable condition}
|b_{12}b_{23}b_{31}|=|b_{21}b_{32}b_{13}|.
\end{equation}
In the cluster algebra theory, a skew-symmetrizable matrix is called an {\em exchange matrix}, and we also use this terminology depending on the context.
\par
For any real number $x \in \mathbb{R}$, define $[x]_{+}=\max(x,0)$. Recall that $\mathcal{T}$ denotes the set of all reduced sequences of $1$, $2$, $3$ (Subsection~\ref{subsec: conv}). Then, the mutation and the $B$-pattern are defined as follows.
\begin{definition}\label{mutation rules}
	Let $B=(b_{ij}) \in \mathrm{M}_3(\mathbb{R})$ be a skew-symmetrizable matrix. For any $k=1,2,3$, the {\em $k$-direction mutation} $\mu_{k}(B)=(b^{\prime}_{ij})$ of $B$ is defined as:
	\begin{align}\label{eq: mutation of b}
		b_{ij}^{\prime}=\left\{
		\begin{array}{ll}
			-b_{ij} &  \text{if}\ i=k \ \mbox{or}\ j=k, \\
			b_{ij}+\mathrm{sign}(b_{ik})[b_{ik}b_{kj}]_{+}&  \text{if}\ i\neq k \ 
			\mbox{and}\ j\neq k. 
		\end{array} \right. 
	\end{align}
    Note that $\mu_k$ is an involution. For any ${\bf w}=[k_1,\dots,k_r] \in \mathcal{T}$, we define $B^{\bf w}=\mu_{k_r}\cdots\mu_{k_1}(B)$. The collection of all such matrices ${\bf B}(B)=\{B^{\bf w}\}_{{\bf w} \in \mathcal{T}}$ is called the {\em $B$-pattern} or the {\em mutation equivalence class}. For this pattern, the given $B$ is called the \emph{initial exchange matrix} of $\mathbf{B}(B)$. If $B$ and $B'$ form the same mutation equivalence class $\mathbf{B}(B)=\mathbf{B}(B')$ as a set, these two matrices are said to be {\em mutation equivalent}.
\end{definition}

\begin{definition}\label{def: C G matrices}
Let $B \in \mathrm{M}_3(\mathbb{R})$ be an initial exchange matrix.
We define the {\em $C$-matrices} $C^{\bf w}=(c_{ij}^{\mathbf{w}}) \in \mathrm{M}_3(\mathbb{R})$ and the {\em $G$-matrices} $G^{\bf w}=(g_{ij}^{\mathbf{w}}) \in \mathrm{M}_3(\mathbb{R})$ by the initial conditions $C^{\emptyset}=G^{\emptyset}=I$ (the identity matrix) and the following recursions:
\begin{equation}\label{eq: recursion of C}
\begin{aligned} 
c_{ij}^{\mathbf{w}[k]}&=\begin{cases}
-c_{ik}^{\mathbf{w}} & j=k,\\
c_{ij}^{\mathbf{w}} + c_{ik}^{\mathbf{w}}[b_{kj}^{\mathbf{w}}]_{+}+[-c_{ik}^{\mathbf{w}}]_{+}b_{kj}^{\mathbf{w}} & j \neq k,
\end{cases}
\\
g_{ij}^{\mathbf{w}[k]}&=\begin{cases}
-g_{ik}^{\mathbf{w}}+\sum_{l=1}^{3}g_{il}^{\mathbf{w}}[b_{lk}^{\mathbf{w}}]_{+}-\sum_{l=1}^3b_{il}^{\emptyset}[c_{lk}^{\mathbf{w}}]_{+} & j=k,
\\
g_{ij}^{\mathbf{w}} & j \neq k.
\end{cases}
\end{aligned}
\end{equation}
The collections of all $C$-matrices and $G$-matrices are called the {\em $C$-pattern} and the {\em $G$-pattern}, and we denote them by ${\bf C}(B)=\{C^{\bf w}\}_{{\bf w} \in \mathcal{T}}$ and ${\bf G}(B)=\{G^{\bf w}\}_{{\bf w} \in \mathcal{T}}$, respectively.
Each column vector of $C^{\bf w}$ (resp. $G^{\bf w}$) is called a {\em $c$-vector} (resp. {\em $g$-vector}), and we denote them by $C^{\bf w}=({\bf c}^{\bf w}_1, {\bf c}^{\bf w}_2,{\bf c}^{\bf w}_{3})$ and $G^{\bf w}=({\bf g}^{\bf w}_1, {\bf g}^{\bf w}_2,{\bf g}^{\bf w}_3)$.
\end{definition}
\begin{remark}\label{rmk: C,G}
The $c$-, $g$-vectors were introduced in \cite{FZ07} to record the mutation dynamics of coefficients through the tropical semifield and describe the $\mbZ^n$-grading on cluster variables respectively. They play a central role in understanding the sign-coherence, separation formulas, and the parameterization of cluster variables.
\end{remark}

We introduce a partial order $\leq$ on $\mathbb{R}^3$ as follows: $(u_1,u_2,u_3)^{\top} \leq (v_1,v_2,v_3)^{\top}$ if and only if $u_i \leq v_i\ \textup{for any $i=1,2,3$}$. The following is a key notion for $C$-matrices and $G$-matrices.
\begin{definition}
Let $B \in \mathrm{M}_3(\mathbb{R})$ be an initial exchange matrix. Then, a $C$-matrix $C^{\bf w}$ is said to be {\em sign-coherent} if either ${\bf c}^{\bf w}_{i} \geq {\bf 0}$ or ${\bf c}^{\bf w}_i \leq {\bf 0}$ holds for any $i=1,2,3$. The $C$-pattern ${\bf C}(B)$ is called {\em sign-coherent} if every $C^{\bf w}$ is sign-coherent. When $C^{\bf w}$ is sign-coherent, the {\em tropical sign} $\varepsilon_{i}^{\bf w} \in \{0,\pm 1\}$ of ${\bf c}^{\bf w}_{i}$ is defined as follows: $\varepsilon_{i}^{\bf w}=0$ if ${\bf c}_i^{\bf w} = {\bf 0}$, otherwise $\varepsilon_{i}^{\bf w} \in \{\pm 1\}$ such that $\varepsilon_{i}^{\bf w}{\bf c}_{i}^{\bf w} \geq {\bf 0}$.
\end{definition}

\begin{remark}\label{rmk sign coherent}
In \cite{GHKK18}, it was shown that every ${C}$-pattern corresponding to an {\em integer} skew-symmetrizable matrix is sign-coherent. However, when real entries are allowed, it is known that there exists a sign-incoherent $C$-pattern \cite[Ex.~5.4]{AC25a}, and it is an open problem when the sign-coherence holds.
\end{remark}

Under this assumption, we can obtain the recursion for $c$-, $g$-vectors.
\begin{proposition}[\cite{FZ07}]
Let $B \in \mathrm{M}_3(\mathbb{R})$ be an initial exchange matrix.
If ${\bf C}(B)$ is sign-coherent, we obtain a recursion for $c$-, $g$-vectors as
\begin{align}
{\bf c}^{{\bf w}[k]}_{i}&=\begin{cases}
-{\bf c}^{\bf w}_{k} & i=k,\\
{\bf c}^{\bf w}_{i}+[\varepsilon_{k}^{\bf w}b^{\bf w}_{ki}]_{+}{\bf c}^{\bf w}_{k} & i \neq k,
\end{cases}\label{eq: mutation of c-vectors}\\
{\bf g}^{{\bf w}[k]}_i&=\begin{cases}
-{\bf g}^{\bf w}_{k} + \sum_{j=1}^{3}[-\varepsilon_k^{\bf w}b_{jk}^{\bf w}]_{+}{\bf g}^{\bf w}_{j} & i=k,\\
{\bf g}^{\bf w}_{i} & i \neq k.
\end{cases}\label{eq: mutation of g-vectors}
\end{align}
\end{proposition}

As a consequence of this recursion, we obtain the following important properties.
\begin{proposition}[\cite{FZ07}]\label{prop: fundamental properties under sign-coherency}
Suppose that ${\bf C}(B)$ is sign-coherent. For any ${\bf w} \in \mathcal{T}$, we have $|C^{\bf w}|=|G^{\bf w}|=(-1)^{|{\bf w}|}$. In particular, the sets of $c$-vectors $\{{\bf c}_{1}^{\bf w},{\bf c}_{2}^{\bf w},{\bf c}_{3}^{\bf w}\}$ and $g$-vectors $\{{\bf g}_{1}^{\bf w}, {\bf g}_{2}^{\bf w}, {\bf g}_{3}^{\bf w}\}$ are bases of $\mathbb{R}^3$.
\end{proposition}
When we discuss real $C$-, $G$-matrices, the following conjecture is important but not shown in general. However, for our focusing matrices which are said to be cluster-cyclic of rank $3$, it has already been proved in \cite{AC25b}, see also \Cref{prop: conjecture is true}.
\begin{conjecture}[{\cite[Conj. 6.1 \& Conj. 6.3]{AC25a}}]\label{conj: two conjectures}
Let $B \in \mathrm{M}_{3}(\mathbb{R})$ be an initial exchange matrix. Suppose that its $C$-pattern ${\bf C}(B)$ is sign-coherent.
\\
\textup{($a$)} For any ${\bf w} \in \mathcal{T}$, all $C$-patterns ${\bf C}(B^{\bf w})$ and ${\bf C}((B^{\bf w})^{\top})$ are sign-coherent.
\\
\textup{($b$)} For any ${\bf w} \in \mathcal{T}$, let $B'=B^{\bf w}$. Consider its $C$-pattern ${\bf C}(B')=\{C_{B'}^{\bf u}\}_{{\bf u} \in \mathcal{T}}$. If its $c$-vector ${\bf c}^{\bf u}_{i;B'}$ is expressed as ${\bf c}^{\bf u}_{i;B'}=\alpha{\bf e}_{j}$ for some $\alpha \in \mathbb{R}$ and $j\in \{1,\dots,n\}$, then we have $\alpha=\pm\sqrt{d_id_j^{-1}}$.
\end{conjecture}

\subsection{$G$-fan structure}
In this section, we fix an initial exchange matrix $B \in \mathrm{M}_3(\mathbb{R})$ and its skew-symmetrizer $D=\mathrm{diag}(d_1,d_2,d_3)$. We suppose that ${\bf C}(B)$ is sign-coherent. In this case, there are some good geometric structures in $c$-, $g$-vectors.
\par
We introduce an inner product $\langle\, ,\, \rangle_{D}$ in $\mathbb{R}^3$ as
\begin{equation}\label{eq: inner product}
\langle {\bf a}, {\bf b} \rangle_{D}={\bf a}^{\top}D{\bf b}.
\end{equation}
Then, the following duality between $c$-, $g$-vectors holds.
\begin{proposition}[{\cite[Prop.~2.16]{Nak23}}]\label{prop: orthogonality of c g vectors}
Suppose that ${\bf C}(B)$ is sign-coherent. Then, for any ${\bf w} \in \mathcal{T}$ and $i,j=1,2,3$, we have
\begin{equation}\label{eq: orthogonal relation for usual vectors}
\langle {\bf g}_{i}^{\bf w}, {\bf c}_{j}^{\bf w}\rangle_{D}=\begin{cases}
d_i & \textup{if $i=j$},\\
0 & \textup{if $i \neq j$}.
\end{cases}
\end{equation}
\end{proposition}
Motivated by this equality, we always fix one skew-symmetrizer $D$ and the corresponding inner product $\langle\,,\,\rangle_{D}$ in $\mathbb{R}^{3}$.
\par
Then, we introduce the following notions.
\begin{definition}[Convex cone]
A nonempty set $\mathcal{C} \subset \mathbb{R}^{3}$ is called a {\em convex cone} if $\lambda{\bf a}+\lambda'{\bf b} \in \mathcal{C}$ holds for any ${\bf a},{\bf b} \in \mathcal{C}$ and $\lambda,\lambda' \in \mathbb{R}_{> 0}$. For any convex cone $\mathcal{C} \subset \mathbb{R}^3$, we write its {\em relative interior} by $\mathcal{C}^{\circ}$, which is the interior of $\mathcal{C}$ in the subspace $\langle \mathcal{C} \rangle_{\mathrm{vec}} \subset \mathbb{R}^3$. For any convex cone $\mathcal{C} \subset \mathbb{R}^3$, its {\em dimension} $\dim(\mathcal{C})$ is defined by the dimension $\dim\,\langle \mathcal{C} \rangle_{\mathrm{vec}}$ of the vector subspace spanned by $\mathcal{C}$.
\par
For any finite set of vectors ${\bf a}_1,\dots,{\bf a}_r \in \mathbb{R}^3$, we define the following two convex cones:
\begin{equation}
\mathcal{C}({\bf a}_1,\dots,{\bf a}_{r})=\left\{\left.\sum_{i=1}^{r}\lambda_{i}{\bf a}_{i}\  \right|\ \lambda_{i} \geq 0\right\},
\quad
\mathcal{C}^{\circ}({\bf a}_1,\dots,{\bf a}_{r})=\left\{\left.\sum_{i=1}^{r}\lambda_{i}{\bf a}_{i}\  \right|\ \lambda_{i} > 0\right\}.
\end{equation}
Note that $\mathcal{C}^{\circ}({\bf a}_1,\dots,{\bf a}_r)$ is the relative interior of $\mathcal{C}({\bf a}_1,\dots,{\bf a}_r)$. In convention, we write $\mathcal{C}(\emptyset)=\{{\bf 0}\}$. We say that a convex cone of the form $\mathcal{C}({\bf a}_1, \dots, {\bf a}_r)$ is a {\em polyhedral cone}, or simply a {\em cone}.
If ${\bf a}_1,\dots,{\bf a}_r$ are linearly independent, then the cone $\mathcal{C}({\bf a}_1,\dots,{\bf a}_r)$ said to be {\em simplicial}. In this case, $r \leq 3$ holds.
\end{definition}
For any vector ${\bf v} \in \mathbb{R}^{3}$, set
\begin{equation}
\begin{aligned}
\zeroregion{\bf v}&=\{{\bf x} \in \mathbb{R}^3 \mid \langle {\bf x}, {\bf v}\rangle_{D} = 0\},\\
\positiveregion{\bf v}&=\{{\bf x} \in \mathbb{R}^3 \mid \langle {\bf x}, {\bf v}\rangle_{D} > 0\},\\
\negativeregion{\bf v}&=\{{\bf x} \in \mathbb{R}^3 \mid \langle {\bf x}, {\bf v}\rangle_{D} < 0\}.
\end{aligned}
\end{equation}
Note that $\zeroregion{\bf 0}=\mathbb{R}^3$. Also, we set $\positiveclosure{\bf v}=\positiveregion{\bf v}\cup\zeroregion{\bf v}$ and $\negativeclosure{\bf v}=\negativeregion{\bf v}\cup\zeroregion{\bf v}$. When $\mathbf{v} \neq \mathbf{0}$, they represent the corresponding closures.
\par
When $d_1=d_2=d_3=1$, that is, when the inner product coincides with the Euclidean inner product, we omit $D$ and write $\Ezeroregion{\mathbf{v}}$, $\Epositiveregion{\mathbf{v}}$, and so on.
\begin{definition}[Face]
Let $\mathcal{C} \subset \mathbb{R}^3$ be a convex cone. Let ${\bf v} \in \mathbb{R}^{3}$ and suppose that $\mathcal{C} \subset \positiveclosure{\mathbf{v}}$. Then, the following set $\mathrm{face}_{\bf v}(\mathcal{C}) \subset \mathcal{C}$ is called a {\em face} of $\mathcal{C}$ associated with a {\em normal vector} ${\bf v}$.
\begin{equation}
\mathrm{face}_{\bf v}(\mathcal{C})=\mathcal{C} \cap \zeroregion{\mathbf{v}}.
\end{equation}
\end{definition}
Note that $\mathcal{C}=\mathrm{face}_{\bf 0}(\mathcal{C})$ holds. Hence, each cone is a face of itself.
\par
Let $\mathcal{C}=\mathcal{C}({\bf a}_1,\dots,{\bf a}_{r})$ be a simplicial cone generated by linearly independent vectors ${\bf a}_1,\dots,{\bf a}_{r}$, where $r\leq 3$. For any $J \subset \{1,\dots,r\}$, we write
\begin{equation}
\mathcal{C}_{J}({\bf a}_1,\dots,{\bf a}_r)=\left.\left\{ \sum_{j \in J} \lambda_j{\bf a}_{j} \ \right|\ \lambda_j \geq 0\right\}.
\end{equation}
Then, every face of a simplicial cone $\mathcal{C}=\mathcal{C}({\bf a}_{1},\dots,{\bf a}_{r})$ may be expressed as the above form. In fact, for any ${\bf v} \in \mathbb{R}^{3}$ satisfying $\mathcal{C} \subset \positiveclosure{\mathbf{v}}$, by setting $J=\{j \in \{1,2,\dots,r\}\mid  {\bf a}_{j} \in \zeroregion{\mathbf{v}}\}$, we have $\mathrm{face}_{\bf v}(\mathcal{C})=\mathcal{C}_{J}({\bf a}_1,\dots,{\bf a}_{r})$.

\begin{definition}[Fan]
A nonempty set $\Delta$ of cones in $\mathbb{R}^3$ is called a {\em fan} if it satisfies the following two conditions:
\begin{itemize}
\item For any $\mathcal{C} \in \Delta$, every face of $\mathcal{C}$ is also contained in $\Delta$.
\item For any $\mathcal{C}_1,\mathcal{C}_2 \in \Delta$, then $\mathcal{C}_1 \cap \mathcal{C}_2$ is a face of both $\mathcal{C}_1$ and $\mathcal{C}_2$.
\end{itemize}
For any fan $\Delta$, its {\em support} $|\Delta|$ is defined by $\bigcup_{\mathcal{C} \in \Delta}\mathcal{C}$. Moreover, a fan $\Delta$ is said to be {\em simplicial} if every cone $\mcC \in \Delta$ is simplicial.
\end{definition}
In the cluster algebra theory, it is known that $g$-vectors form a fan structure.
\begin{definition}[$G$-cone, $G$-fan]\label{def: G-cone, G-fan}
Let $B \in \mathrm{M}_{3}(\mathbb{R})$ be an initial exchange matrix. Suppose that ${\bf C}(B)$ is sign-coherent. Then, we write
\begin{equation}
\mathcal{C}(G^{\bf w})=\mathcal{C}({\bf g}_1^{\bf w},{\bf g}_{2}^{\bf w},{\bf g}_{3}^{\bf w}),
\end{equation}
and call it a {\em $G$-cone}. Moreover, for any $J \subset \{1,2,3\}$, we write $\mathcal{C}_{J}(G^{\bf w})=\mathcal{C}_{J}({\bf g}_{1}^{\bf w},{\bf g}_{2}^{\bf w},{\bf g}_{3}^{\bf w})$.
The set of all $G$-cones and their faces
\begin{equation}
\Delta(B)=\{\mathcal{C}_{J}(G^{\bf w}) \mid J \subset \{1,2,3\}, {\bf w} \in \mathcal{T}\}
\end{equation}
is called the {\em $G$-fan} associated with $B$. For any ${{\bf w}_0} \in \mathcal{T}$, we define the {\em sub $G$-fan} $\Delta^{\geq {\bf w}_0}(B)$ by
\begin{equation}
\Delta^{\geq {\bf w}_0}(B)=\{\mathcal{C}_{J}(G^{\bf w}) \mid J \subset \{1,2,3\},\ {\bf w} \geq {\bf w}_0\}.
\end{equation}
\end{definition}
\begin{proposition}[{\cite{FZ07, Rea14, Nak23}, \cite[Thm.~8.3]{AC25a}}]\label{prop: fan structure}
Let $B \in \mathrm{M}_3(\mathbb{R})$ be an initial exchange matrix. Suppose that ${\bf C}(B)$ is sign-coherent and \Cref{conj: two conjectures} holds for this $B$. Then, the $G$-fan $\Delta(B)$ is really a fan.
\end{proposition}

\subsection{Skew-symmetrizing method}
The matrix patterns are defined by the skew-symmetrizable matrix $B$. On the other hand, the matrix patterns corresponding to $B$ are closely related to the ones corresponding to the following skew-symmetric matrix $\tilde{B}$.
\begin{definition}\label{def: skew-symmetrized matrix}
For any skew-symmetrizable matrix $B \in \mathrm{M}_{3}(\mathbb{R})$ with a skew-symmetrizer $D=\mathrm{diag}(d_1,d_2,d_3)$, let $\tilde{B}=D^{\frac{1}{2}}BD^{-\frac{1}{2}}$, where $D^{\frac{1}{2}}=\mathrm{diag}(\sqrt{d_1},\sqrt{d_2},\sqrt{d_3})$ and $D^{-\frac{1}{2}}=(D^{\frac{1}{2}})^{-1}$. Then, for any ${\bf w}=[k_{1},k_{2},\dots,k_{r}] \in \mathcal{T}$, we define $\tilde{B}^{\bf w} = \mu_{k_r}\cdots \mu_{k_1}(\tilde{B})$.
\end{definition}
These matrices can easily be calculated by the following rule.
\begin{lemma}[{\cite[Lem.~8.3]{FZ03}}]\label{lem: skew-symmetrized matrices formula}
Let $B \in \mathrm{M}_{3}(\mathbb{R})$ be an initial exchange matrix with a skew-symmetrizer $D$. Set $B^{\bf w}=(b_{ij}^{\bf w}) \in \mathrm{M}_{3}(\mathbb{R})$ and $\tilde{B}^{\bf w}=(\tilde{b}_{ij}^{\bf w}) \in \mathrm{M}_{3}(\mathbb{R})$. Then, we have
\begin{equation}
\begin{aligned}
\tilde{B}^{\bf w}&=D^{\frac{1}{2}}B^{\bf w}D^{-\frac{1}{2}},\  
\tilde{b}_{ij}^{\bf w}&=\mathrm{sign}(b_{ij}^{\bf w})\sqrt{|b_{ij}^{\bf w}b_{ji}^{\bf w}|}.
\end{aligned}
\end{equation}
In particular, $\tilde{B}^{\bf w}$ is independent of the choice of $D$ and it is always skew-symmetric. 
\end{lemma}
Based on this fact, we write $\mathrm{Sk}(B)=D^{\frac{1}{2}}BD^{-\frac{1}{2}}$. Since this matrix is independent of the choice of $D$, the operator $\mathrm{Sk}$ can be seen as a function on the set of all skew-symmetrizable matrices.
\par
Now, we may consider the two $C$-, $G$-patterns. One is ${\bf C}(B)=\{C^{\bf w}\}$ and ${\bf G}(B)=\{G^{\bf w}\}$, and the other is ${\bf C}(\mathrm{Sk}(B))=\{\hat{C}^{\bf w}\}$ and ${\bf G}(\mathrm{Sk}(B))=\{\hat{G}^{\bf w}\}$. Let $\hat{C}^{\mathbf{w}}=(\hat{\mathbf{c}}_{1}^{\mathbf{w}},\hat{\mathbf{c}}_{2}^{\mathbf{w}},\hat{\mathbf{c}}_3^{\mathbf{w}})$ and $\hat{G}^{\mathbf{w}}=(\hat{\mathbf{g}}_{1}^{\mathbf{w}},\hat{\mathbf{g}}_{2}^{\mathbf{w}},\hat{\mathbf{g}}_3^{\mathbf{w}})$. Then, the following relations hold.
\begin{proposition}[{\cite[Prop.~3.5]{AC25a}}]\label{prop: skew-symmetrizing method}
For any ${\bf w} \in \mathcal{T}$ and $i=1,2,3$, we have
\begin{equation}\label{eq: skew-symmetrizing method}
\hat{C}^{\bf w}=D^{\frac{1}{2}}C^{\bf w}D^{-\frac{1}{2}},
\quad
\hat{G}^{\bf w}=D^{\frac{1}{2}}G^{\bf w}D^{-\frac{1}{2}},
\quad
\hat{\mathbf{c}}_{i}^{\mathbf{w}}=\frac{1}{\sqrt{d_i}}D^{\frac{1}{2}}\mathbf{c}_{i}^{\mathbf{w}},
\quad
\hat{\mathbf{g}}_{i}^{\mathbf{w}}=\frac{1}{\sqrt{d_i}}D^{\frac{1}{2}}\mathbf{g}_{i}^{\mathbf{w}}.
\end{equation}
In particular, for the $G$-fans, there is a one-to-one correspondence $\Delta(B) \to \Delta(\mathrm{Sk}(B))$ given by
\begin{equation}
\mathcal{C}_{J}(G^{\mathbf{w}}) \mapsto \mathcal{C}_J(\hat{G}^{\mathbf{w}})=D^{\frac{1}{2}}\mathcal{C}_J(G^{\mathbf{w}})
\quad
(J \subset \{1,2,3\},\ \mathbf{w} \in \mathcal{T}).
\end{equation}
\end{proposition}

\begin{remark}
Thanks to \Cref{prop: skew-symmetrizing method}, when we consider the $C$-, $G$-matrices associated with a {\em real} exchange matrix $B$, it suffices to focus on the ones associated with the skew-symmetric matrix $\mathrm{Sk}(B)$. However, since the standard theory of cluster algebras is formulated for skew-symmetrizable matrices, we provide the subsequent proofs and statements in the skew-symmetrizable setting to ensure broader applicability.
\end{remark}
\subsection{Modified $c$-, $g$-vectors}
When we discuss the $G$-fan structure, the following modified vectors are more helpful than the ordinary $c$-, $g$-vectors.
\begin{definition}\label{def: modified c g vectors}
Let $B \in \mathrm{M}_{3}(\mathbb{R})$ be an initial exchange matrix with a skew-symmetrizer $D=\mathrm{diag}(d_1,d_2,d_3)$. We define the {\em modified $C$-matrix} $\tilde{C}^{\bf w}$ and the {\em modified $G$-matrix} $\tilde{G}^{\bf w}$ by
\begin{equation}
\tilde{C}^{\bf w}=C^{\bf w}D^{-\frac{1}{2}},
\quad
\tilde{G}^{\bf w}=G^{\bf w}D^{-\frac{1}{2}}.
\end{equation}
Their column vectors are called
{\em modified $c$-vectors} $\tilde{\bf c}_{i}^{\bf w}$ and {\em modified $g$-vectors} $\tilde{\bf g}_{i}^{\bf w}$, which are given by
\begin{equation}\label{eq: definition of modified c-, g-vectors}
\tilde{\bf c}_{i}^{\bf w}=\frac{1}{\sqrt{d_i}}{\bf c}_{i}^{\bf w},
\quad
\tilde{\bf g}_{i}^{\bf w}=\frac{1}{\sqrt{d_i}}{\bf g}_{i}^{\bf w}.
\end{equation}
\end{definition}
As in Definition~\ref{def: G-cone, G-fan}, we may consider the cone $\mathcal{C}_{J}(\tilde{G}^{\bf w})$ ($J \subset \{1,2,3\}$) spanned by its modified $g$-vectors. However, by \eqref{eq: definition of modified c-, g-vectors}, it holds that
\begin{equation}
\mathcal{C}_{J}(G^{\bf w})=\mathcal{C}_{J}(\tilde{G}^{\bf w}).
\end{equation}
We define the {\em modified standard vectors} $\tilde{\bf e}_i=\sqrt{d_i^{-1}}{\bf e}_i$ for any $i=1,2,3$. Then, we may obtain the following recursion.

\begin{proposition}[{\cite[Prop.~9.3]{AC25a}}]
Let $B \in \mathrm{M}_3(\mathbb{R})$ be an initial exchange matrix.
Suppose that ${\bf C}(B)$ is sign-coherent. 
Let $\tilde{B}^{\bf w}=(\tilde{b}_{ij}^{\bf w}) \in \mathrm{M}_{3}(\mathbb{R})$ be the skew-symmetric matrix defined in \Cref{def: skew-symmetrized matrix}.
Then, the modified $c$-, $g$-vectors may be obtained by the following recursions.
\begin{itemize}
\item $\tilde{\bf c}_{i}^{\emptyset}=\tilde{\bf g}_{i}^{\emptyset}=\tilde{\bf e}_{i}$.
\item For any ${\bf w} \in \mathcal{T}$ and $k=1,2,3$, we have
\begin{equation}\label{eq: mutation of modified c- g-vectors}
\tilde{\bf c}^{{\bf w}[k]}_{i}=\begin{cases}
- \tilde{\bf c}^{\bf w}_{k} & i = k,\\
\tilde{\bf c}^{\bf w}_{i}+[\varepsilon_{k}^{\bf w}\tilde{b}_{ki}^{\bf w}]_{+}\tilde{\bf c}^{\bf w}_{k} & i \neq k,
\end{cases}
\quad
\tilde{\bf g}^{{\bf w}[k]}_{i}=\begin{cases}
-\tilde{\bf g}_{k}^{\bf w}+\sum_{j=1}^{3}[-\varepsilon_{k}^{\bf w}\tilde{b}_{jk}^{\bf w}]_{+}\tilde{\bf g}_{j}^{\bf w} & i=k,\\
\tilde{\bf g}_{i}^{\bf w} & i \neq k.
\end{cases}
\end{equation}
\end{itemize}
\end{proposition}
One advantage of using modified vectors is that certain formulas become simpler. For example, by \eqref{eq: orthogonal relation for usual vectors}, the standard duality holds as follows.
\begin{lemma}[{\cite[Prop.~9.5]{AC25a}}]
The following equality holds: 
\begin{equation}\label{eq: orthogonal relation for modified vectors}
\langle \tilde{\bf g}_{i}^{\bf w}, \tilde{\bf c}_{j}^{\bf w} \rangle_{D}=\begin{cases}
1 & \textup{if $i=j$},\\
0 & \textup{if $i \neq j$}.
\end{cases}
\end{equation}
\end{lemma}
\begin{proof}
This follows from \Cref{prop: orthogonality of c g vectors}.
\end{proof}
\subsection{Stereographic projection}
To draw the $G$-fan in $\mathbb{R}^3$, we often use its stereographic projection (e.g. \cite{Mul16, Rea23, Nak23, Nak26}). We explain how to draw such pictures.
\begin{enumerate}
	\item Firstly, we consider the projection of $G$-fan to the unit sphere $S^2$ in $\mbR^3$.
	\item Secondly, we use the stereographic projection to the tangent plane $P$ of $S^2$ at $(\frac{\sqrt{3}}{3},\frac{\sqrt{3}}{3},\frac{\sqrt{3}}{3})$ from the antipode $N=(-\frac{\sqrt{3}}{3},-\frac{\sqrt{3}}{3},-\frac{\sqrt{3}}{3})$.
\end{enumerate}
This map is a continuous bijection between a certain projectivization $(\mathbb{R}^3\setminus \mathcal{C}(N))/\mathbb{R}_{>0} \cong S^2\setminus\{N\}$ and the plane $P \cong \mathbb{R}^2$. In the context of the $G$-fan, the eliminated point $\mathcal{C}(N)$ does not affect the essential structure. This is because the negative orthant $\mathbb{R}_{\leq 0}^3$ is known to be either a single chamber or the compliment of the $G$-fan (e.g. \cite{NZ12}, \cite[Prop.~5.4]{AC25a}).
\par
In this image, each half line in $\mathbb{R}^3$ is sent to a point, and each plane is sent to a circle. One important fact for the $G$-fan is that each $3$-dimensional cone appears as a triangle.
\par
For example, the three planes $\Ezeroregion{\mathbf{e}_1}$, $\Ezeroregion{\mathbf{e}_2}$, and $\Ezeroregion{\mathbf{e}_3}$ are illustrated as the three circles in \Cref{fig: orthants}, and each connected component in this picture represents an orthant. The center triangle is the cone $\mathcal{C}(\mathbf{e}_1, \mathbf{e}_2, \mathbf{e}_3)$. This is a chamber of every $G$-fan because the initial matrix is the identity.
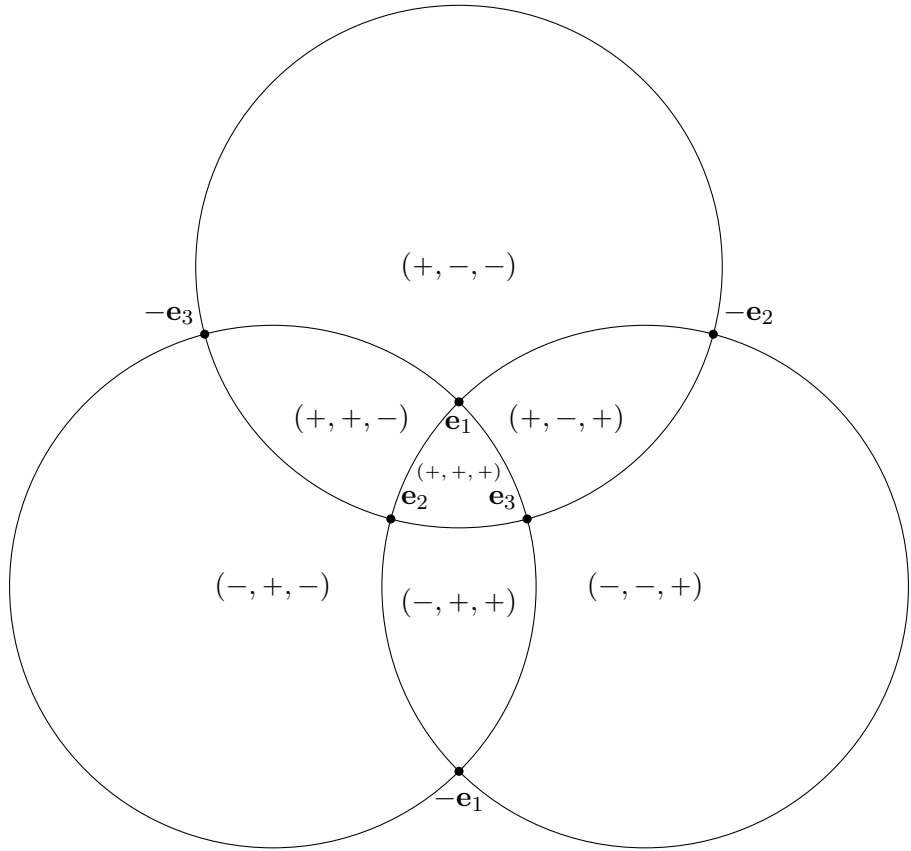
\begin{figure}
\centering
\begin{tikzpicture}
\draw (0,{2*sqrt(2)}) node {$(+,-,-)$} circle [radius= {2*sqrt(3)}];    
\draw ({-sqrt(6)},{-sqrt(2)}) node {$(-,+,-)$} circle [radius= {2*sqrt(3)}];    
\draw ({sqrt(6)},{-sqrt(2)}) node {$(-,-,+)$} circle [radius= {2*sqrt(3)}];    
\draw (0,1.0353) node [below = 1.5pt] {${\bf e}_1$};    
\draw (-0.8966, -0.5176) node [above right] {${\bf e}_2$};   
\draw (0.8966, -0.5176) node [above left] {${\bf e}_3$};   
\fill (0,1.0353) circle [radius=0.06];
\fill (-0.8966, -0.5176) circle [radius=0.06];
\fill (0.8966, -0.5176) circle [radius=0.06];
\fill (0.0,-3.8637) circle [radius=0.06];
\draw (0.0,-3.8637) node [below = 3pt] {$-{\bf e}_1$};
\fill (3.3461,1.9319) circle [radius=0.06];
\draw (3.3461,1.9319) node [above right] {$-{\bf e}_2$};
\fill (-3.3461,1.9319) circle [radius=0.06];
\draw (-3.3461,1.9319) node [above left] {$-{\bf e}_3$};
\draw (0,0.1) node {\tiny $(+,+,+)$};
\draw (1.4142,0.8165) node {$(+,-,+)$};
\draw (-1.4142,0.8165) node {$(+,+,-)$};
\draw (0,-1.633) node {$(-,+,+)$};
\end{tikzpicture}
\caption{The orthants in the stereographic projection.}
\label{fig: orthants}
\end{figure}

\section{Rank $3$ cluster-cyclic framework}\label{sec: cluster-cyclic framework}
In this section, we introduce rank $3$ cluster-cyclic exchange matrices and its classification,  and present the corresponding mutation formulas, which are simplified due to the specific nature of the cluster-cyclic structure.
\subsection{Markov constant and cluster-cyclicity}
From now on, we always focus on rank $3$ cluster-cyclic exchange matrices throughout this paper, which is defined as follows.
\begin{definition}
Let $B \in \mathrm{M}_{3}(\mathbb{R})$ be a skew-symmetrizable matrix. This matrix is said to be {\em cyclic} if it holds that
\begin{equation}\label{eq: definition of cluster cyclic}
\mathrm{sign}(B)=\pm\left(\begin{matrix}
0 & + & -\\
- & 0 & +\\
+ & - & 0
\end{matrix}\right).
\end{equation}
(We assume that all off-diagonal entries are nonzero.)
For any initial exchange matrix $B \in \mathrm{M}_3(\mathbb{R})$, if every $B^{\bf w}$ (${\bf w} \in \mathcal{T}$) is cyclic, $B$ is said to be {\em cluster-cyclic}.
\end{definition}
Let $B$ be a cluster-cyclic matrix.
When we consider $B^{{\bf w}[k]}$, the signs on the $k$th row and the $k$th column are changed. Thus, it holds that
\begin{equation}\label{eq: sign reversed relation}
\mathrm{sign}(B^{{\bf w}[k]})=-\mathrm{sign}(B^{\bf w}).
\end{equation}
\begin{example}
One important example of the cluster-cyclic matrices is the skew-symmetric matrix corresponding to the {\em Markov quiver}, which is given by
\begin{equation}\label{eq: Markov matrix}
\pm \left(\begin{matrix}
0 & -2 & 2\\
2 & 0 & -2\\
-2 & 2 & 0
\end{matrix}\right).
\end{equation}
Let $B$ be one of the above matrices. In this case, the mutation is given by $\mu_i(B)=-B$ for each $i=1,2,3$.
\par
In view of \Cref{lem: skew-symmetrized matrices formula} and \Cref{prop: skew-symmetrizing method}, this property can be generalized to the following skew-symmetrizable matrix:
\begin{equation}\label{eq: value-rigid type}
B=\pm\left(\begin{matrix}
0 & -p' & r\\
p & 0 & -q'\\
-r' & q & 0
\end{matrix}\right),
\end{equation}
where $p,q,r,p',q',r' \in \mathbb{R}_{>0}$ satisfy $pp'=qq'=rr'=4$ and $pqr=p'q'r'$. In this case, $\mathrm{Sk}(B)$ shapes as \eqref{eq: Markov matrix} and the mutation of $B$ is given by $\mu_i(B)=-B$ for each $i=1,2,3$.
\end{example}
The following number plays an important role in the classification of cluster-cyclicity.
\begin{definition}
For any cyclic matrix $B=(b_{ij}) \in \mathrm{M}_3(\mathbb{R})$, define the {\em Markov constant} $C(B)$ by
\begin{equation}
C(B)=|b_{12}b_{21}|+|b_{23}b_{32}|+|b_{31}b_{13}|-|b_{12}b_{23}b_{31}|.
\end{equation}
\end{definition}
It was introduced by \cite{BBH11, Aka24}. At first glance, the expression of the Markov constant seems not symmetric due to the last term $b_{12}b_{23}b_{31}$. However, by \eqref{eq: rank 3 skew-symmetrizable condition}, the last term can be replaced with $|b_{21}b_{32}b_{13}|$, which ensures the symmetry of this definition.
In particular, by setting $p_{ij}=\sqrt{|b_{ij}b_{ji}|}$, we have $|b_{12}b_{23}b_{31}|=p_{12}p_{23}p_{31}$. Thus, we may obtain the following expression.
\begin{equation}
C(B)=p_{12}^2+p_{23}^2+p_{31}^2-p_{12}p_{23}p_{31}.
\end{equation}

\begin{proposition}[{\cite{BBH11,Aka24}}]\label{prop: ordinary classification of cluster-cyclicity}
Let $B=(b_{ij}) \in \mathrm{M}_{3}(\mathbb{R})$ be a skew-symmetrizable and cyclic matrix. Then, the following two conditions are equivalent.
\begin{itemize}
\item $B$ is cluster-cyclic.
\item It holds that $|b_{ij}b_{ji}| \geq 4$ for any $i \neq j$ and $C(B) \leq 4$.
\end{itemize}
\end{proposition}
\begin{remark}
Another classification is known by \cite{Sev12} based on an {\em admissible quasi-Cartan companion}. In this paper, we take a specific one called the {\em pseudo Cartan companion}, see \eqref{eq: definition of pseudo Cartan companion}, and state this fact in \Cref{rem: Seven's result}.
\end{remark}
We now give some other inequalities derived from \Cref{prop: ordinary classification of cluster-cyclicity}. For a real number $p \geq 2$, set $\alpha(p)=\frac{1}{2}(p+\sqrt{p^2-4})$. By a direct calculation, we verify $\alpha(p)^2-p\alpha(p)+1=0$.
\begin{lemma}
Let $B=(b_{ij}) \in \mathrm{M}_3(\mathbb{R})$ be a cluster-cyclic matrix. Set $p_{ij}=\sqrt{|b_{ij}b_{ji}|}$ and $\alpha_{ij}=\alpha(p_{ij})$. Then, for any $\{i,j,k\}=\{1,2,3\}$, the following inequalities hold.
\begin{gather}
\frac{1}{2}\left(p_{ik}p_{kj}-\sqrt{(p_{ik}^2-4)(p_{kj}^2-4)}\right) \leq p_{ij} \leq \frac{1}{2}\left(p_{ik}p_{kj}+\sqrt{(p_{ik}^2-4)(p_{kj}^2-4)}\right),\label{eq: 1st form of cluster-cyclic}\\
\alpha_{ik}p_{kj}-p_{ij} \geq 0,\label{eq: 2nd form of cluster-cyclic}\\
\alpha_{ik}\alpha_{kj}-\alpha_{ij} \geq 0.\label{eq: 3rd form of cluster-cyclic}
\end{gather}
\end{lemma}
In \cite[Lem.~4.3]{FT19}, an equivalent expression of \eqref{eq: 3rd form of cluster-cyclic} is derived, and is called the {\em triangle inequality}. For the reader's convenience, we also give a proof.
\begin{proof}
For simplicity, set $p=p_{ij}$, $q=p_{ik}$, $r=p_{kj}$, and $\alpha_x=\alpha(x)$ for any $x \geq 2$. By \Cref{prop: ordinary classification of cluster-cyclicity}, we have $p^2+q^2+r^2-pqr \leq 4$. This inequality may be rearranged to
\begin{equation}
\left\{p-\frac{1}{2}\left(qr+\sqrt{(q^2-4)(r^2-4)}\right)\right\}\left\{p-\frac{1}{2}\left(qr-\sqrt{(q^2-4)(r^2-4)}\right)\right\} \leq 0.
\end{equation}
Thus, \eqref{eq: 1st form of cluster-cyclic} holds. Moreover, we have
\begin{equation}
\begin{aligned}
\alpha_qr-p &\geq \frac{1}{2}\left(q+\sqrt{q^2-4}\right)r-\frac{1}{2}\left(qr+\sqrt{(q^2-4)(r^2-4)}\right)
\\
&=\frac{1}{2}\sqrt{q^2-4}\left(r-\sqrt{r^2-4}\right)\geq 0.
\end{aligned}
\end{equation}
Thus, \eqref{eq: 2nd form of cluster-cyclic} holds.
Set $p_0=\frac{1}{2}(qr+\sqrt{(q^2-4)(r^2-4)})$. Then, $p \leq p_0$ implies $\alpha(p) \leq \alpha(p_0)$. Thus, we have
\begin{equation}\label{eq: lowerbound of lemma inclusion formula}
\alpha_{q}\alpha_{r}-\alpha_{p} \geq \frac{1}{4}\left(q+\sqrt{q^2-4}\right)\left(r+\sqrt{r^2-4}\right)-\frac{1}{2}\left(p_0+\sqrt{p_0^2-4}\right).
\end{equation}
Since
\begin{equation}
\begin{aligned}
p_0^2-4&=\frac{1}{4}\left(2q^2r^2-4q^2-4r^2+2qr\sqrt{(q^2-4)(r^2-4)}\right)\\
&=\frac{1}{4}\left(q\sqrt{r^2-4}+r\sqrt{q^2-4}\right)^2,
\end{aligned}
\end{equation}
we have
\begin{equation}\label{eq: p0 expression}
\begin{aligned}
p_0+\sqrt{p_0^2-4}&=\frac{1}{2}\left(qr+\sqrt{(q^2-4)(r^2-4)}+q\sqrt{r^2-4}+r\sqrt{q^2-4}\right)\\
&=\frac{1}{2}\left(q+\sqrt{q^2-4}\right)\left(r+\sqrt{r^2-4}\right).
\end{aligned}
\end{equation}
By substituting it into (\ref{eq: lowerbound of lemma inclusion formula}), we obtain \eqref{eq: 3rd form of cluster-cyclic}.
\end{proof}

\subsection{Mutation of tropical signs}
When we concentrate on the cluster-cyclic case, the recursion for the tropical signs $\varepsilon_j^{\mathbf{w}}$ can be obtained as follows.
\begin{theorem}[{\cite[Thm.~3.4]{AC25b}}]\label{thm: recursion for tropical signs}
Let $B \in \mathrm{M}_{3}(\mathbb{R})$ be a cluster-cyclic initial exchange matrix. Then, its $C$-pattern is sign-coherent. Moreover, for any ${\bf w} \in \mathcal{T}\backslash\{\emptyset\}$, let $k$ be the last index of ${\bf w}$. Then, the following statements hold.
\\
\textup{($a$)} There exists a unique $s \in \{1,2,3\} \setminus\{k\}$ such that
\begin{equation}\label{eq: condition of S}
\varepsilon_{s}^{\bf w} \neq \varepsilon_{k}^{\bf w}, \quad \varepsilon_{s}^{\bf w}\mathrm{sign}(b_{ks}^{\bf w})=-1.
\end{equation}
\textup{($b$)} Let $s \in \{1,2,3\}\setminus\{k\}$ be the above index, and let $t \in \{1,2,3\}\setminus\{k,s\}$ be the other index. Then, we have
\begin{equation}\label{eq: recursion of tropical signs}
(\varepsilon_{k}^{{\bf w}[s]},\varepsilon_{s}^{{\bf w}[s]},\varepsilon_{t}^{{\bf w}[s]})=(-\varepsilon_{k}^{{\bf w}},-\varepsilon_{s}^{{\bf w}},\varepsilon_{t}^{{\bf w}}),
\quad
(\varepsilon_{k}^{{\bf w}[t]},\varepsilon_{s}^{{\bf w}[t]},\varepsilon_{t}^{{\bf w}[t]})=(\varepsilon_{k}^{{\bf w}},\varepsilon_{s}^{{\bf w}},-\varepsilon_{t}^{{\bf w}}).
\end{equation}
In particular, all the tropical signs $\varepsilon_j^{\mathbf{w}}$ may be obtained by the above recursion and the following initial conditions.
\begin{equation}\label{eq: initial tropical signs}
\varepsilon_{j}^{\emptyset}=1,\quad \varepsilon_{j}^{[i]}=\begin{cases}
-1 & i=j,\\
1 & i\neq j.
\end{cases}
\end{equation}
\end{theorem}
Moreover, the following claim about the fan structure holds.
\begin{proposition}[{\cite[Prop.~4.11]{AC25b}}]\label{prop: conjecture is true}
\Cref{conj: two conjectures} holds for any cluster-cyclic exchange matrix of rank $3$. In particular, by \Cref{prop: fan structure}, its $G$-fan is really a fan.
\end{proposition}
Thanks to \Cref{prop: conjecture is true}, we may consider the $G$-fan structure without any assumption. Beforehand, we introduce the following notation for later use.
\begin{definition}\label{def: K S T labeling}
Let $K,S,T:\mathcal{T}\setminus\{\emptyset\} \to \{1,2,3\}$ be the maps defined by $K({\bf w})=k$, $S({\bf w})=s$, and $T({\bf w})=t$, where $k,s,t$ are the indices defined in \Cref{thm: recursion for tropical signs}. For any ${\bf w} \in \mathcal{T}\backslash\{\emptyset\}$ and $M,M'=K,S,T$, we write
\begin{equation}
\varepsilon_{M}^{\bf w}=\varepsilon_{M({\bf w})}^{\bf w},
\quad
b^{\bf w}_{MM'}=b^{\bf w}_{M({\bf w}),M'({\bf w})},
\quad
{\bf g}_{M}^{\bf w}={\bf g}_{M({\bf w})}^{\bf w},
\end{equation}
and so on.
\par
Let $\mathcal{M}$ be the free monoid generated by two letters $S$ and $T$. We define the right monoid action of $\mathcal{M}$ on $\mathcal{T}\backslash\{\emptyset\}$ by
\begin{equation}\label{eq: monoid action definition}
{\bf w}S={\bf w}[S({\bf w})],
\quad
{\bf w}T={\bf w}[T({\bf w})].
\end{equation}
Fix one initial mutation direction $i=1,2,3$, and let us focus on the subset $\mathcal{T}^{\geq [i]}$. We define the {\em trunk} $\mathcal{T}^{<[i]S^{\infty}}$ of $\mathcal{T}^{\geq [i]}$ by
\begin{equation}
\mathcal{T}^{<[i]S^{\infty}}=\{[i]S^n \mid n \in \mathbb{Z}_{\geq 0}\}.
\end{equation}
For each $X \in \mathcal{M}$, the subset $\mathcal{T}^{\geq [i]X}$ is defined by \eqref{eq: after subset}. This subset $\mathcal{T}^{\geq [i]X}$ is called a {\em branch} of $\mathcal{T}^{\geq [i]}$ if $X$ has at least one letter $T$. In particular, for each $n \in \mathbb{Z}_{\geq 0}$, $\mathcal{T}^{\geq [i]S^nT}$ is called the {\em $n$th maximal branch} of $\mathcal{T}^{\geq [i]}$.
\end{definition}
Note that each $\mathcal{T}^{\geq [i]}$ is decomposed into the trunk $\mathcal{T}^{<[i]S^{\infty}}$ and the maximal branches $\mathcal{T}^{\geq [i]S^nT}$ with $n \in \mathbb{Z}_{\geq 0}$, and each branch $\mathcal{T}^{\geq [i]X}$ is a subset of some maximal branch.
\par
These indices $K(\mathbf{w})$, $S(\mathbf{w})$, and $T(\mathbf{w})$ obey the following recursive rules.
\begin{lemma}[{\cite[Lem.~6.4]{AC25b}}]\label{lem: K S T mutation formula}
The following recurrence formulas of indices hold:
\begin{align}
(K({\bf w}S), S({\bf w}S), T({\bf w}S))&=(S({\bf w}), K({\bf w}), T({\bf w})), \label{eq: K S T mutation formula by S}
\\
(K({\bf w}T), S({\bf w}T), T({\bf w}T))&=
\begin{cases}
(T({\bf w}),S({\bf w}), K({\bf w})) & \textup{if ${\bf w}$ is in a trunk},\\
(T({\bf w}),K({\bf w}),S({\bf w})) & \textup{if ${\bf w}$ is in a branch}.
\end{cases}
\end{align}
\end{lemma}
\begin{example}\label{ex: tropical signs}
We give an example how the tropical signs $(\varepsilon_1^{\mathbf{w}},\varepsilon_2^{\mathbf{w}},\varepsilon_3^{\mathbf{w}})$ are obtained and these indices $k=K(\mathbf{w})$, $s=S(\mathbf{w})$, $t=T(\mathbf{w})$ change. Consider
\begin{equation}\label{eq: initial conditions for example of tropical signs}
\mathrm{sign}(B)=\left(\begin{matrix}
0 & - & +\\
+ & 0 & -\\
- & + & 0
\end{matrix}\right),
\quad
i=1.
\end{equation}
By \eqref{eq: initial tropical signs} and
\begin{equation}
\mathrm{sign}(B^{[i]})=\left(\begin{matrix}
0 & + & -\\
- & 0 & +\\
+ & - & 0
\end{matrix}\right),
\end{equation}
we obtain $K([1])=1$, $S([1])=3$, and $T([1])=2$.
After the single mutation, by using \eqref{eq: recursion of tropical signs} and \Cref{lem: K S T mutation formula}, we can obtain the tropical signs $(\varepsilon_1^{\mathbf{w}},\varepsilon_2^{\mathbf{w}},\varepsilon_3^{\mathbf{w}})$ and $k=K(\mathbf{w})$, $s=S(\mathbf{w})$, $t=T(\mathbf{w})$ in \Cref{fig: tropical signs}.
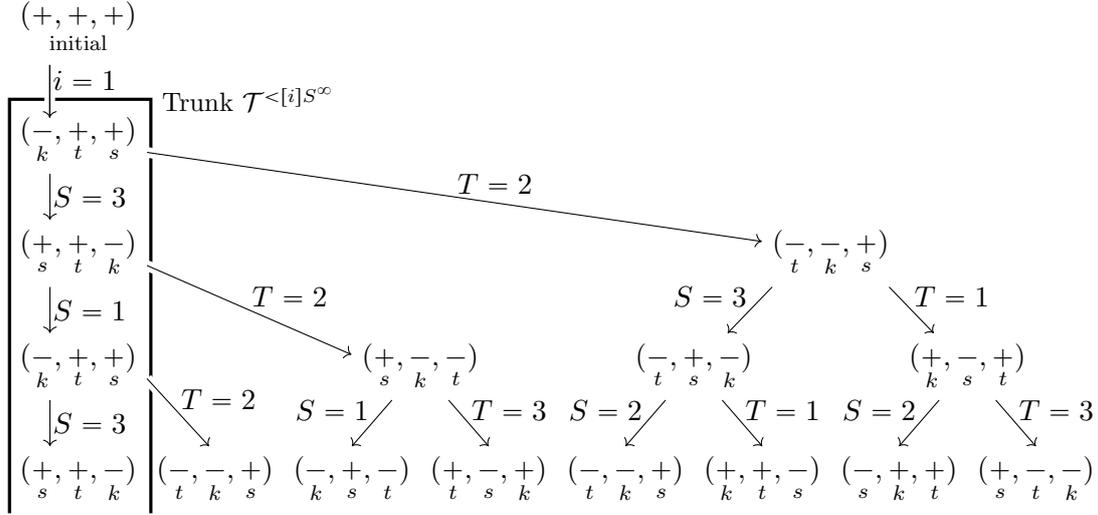
\begin{figure}[htbp]
\centering
\newcommand{\unitdist}{0.9}
\begin{tikzpicture}
\node (O) at (0,0) {$\underset{\textup{initial}}{(+,+,+)}$};
\node (i) at (0,-1.5) {$(\underset{k}{-},\underset{t}{+},\underset{s}{+})$};
\node (iS) at (0,-3) {$(\underset{s}{+},\underset{t}{+},\underset{k}{-})$};
\node (iSS) at (0,-4.5) {$(\underset{k}{-},\underset{t}{+},\underset{s}{+})$};
\node (iSSS) at (0,-6) {$(\underset{s}{+},\underset{t}{+},\underset{k}{-})$};

\node (iSST) at (2*\unitdist,-6) {$(\underset{t}{-},\underset{k}{-},\underset{s}{+})$};
\node (iSTS) at (4*\unitdist,-6) {$(\underset{k}{-},\underset{s}{+},\underset{t}{-})$};
\node (iSTT) at (6*\unitdist,-6) {$(\underset{t}{+},\underset{s}{-},\underset{k}{+})$};
\node (iTSS) at (8*\unitdist,-6) {$(\underset{t}{-},\underset{k}{-},\underset{s}{+})$};
\node (iTST) at (10*\unitdist,-6) {$(\underset{k}{+},\underset{t}{+},\underset{s}{-})$};
\node (iTTS) at (12*\unitdist,-6) {$(\underset{s}{-},\underset{k}{+},\underset{t}{+})$};
\node (iTTT) at (14*\unitdist,-6) {$(\underset{s}{+},\underset{t}{-},\underset{k}{-})$};

\node (iST) at (5*\unitdist,-4.5) {$(\underset{s}{+},\underset{k}{-},\underset{t}{-})$};
\node (iTS) at (9*\unitdist,-4.5) {$(\underset{t}{-},\underset{s}{+},\underset{k}{-})$};
\node (iTT) at (13*\unitdist,-4.5) {$(\underset{k}{+},\underset{s}{-},\underset{t}{+})$};
\node (iT) at (11*\unitdist,-3) {$(\underset{t}{-},\underset{k}{-},\underset{s}{+})$};

\draw[very thick] (iSSS.south west)--($(i.north west)+(0,0.1)$)--($(i.north east)+(0.05,0.1)$) node [right] {\small Trunk $\mathcal{T}^{<[i]S^{\infty}}$}--($(iSSS.south east)+(0.05,0)$);

\draw[preaction={draw=white, line width=4pt}] (-0.375,-0.5)--(-0.375,-1.2) node [pos=0.3, auto=left, inner sep=1pt] {$i=1$};
\draw[->] (-0.375,-0.5)->(-0.375,-1.2);
\draw[->] (i.-130)->(iS.130) node [pos=0.5, auto=left, inner sep=1pt] {$S=3$};
\draw[->] (iS.-130)->(iSS.130) node [pos=0.5, auto=left, inner sep=1pt] {$S=1$};
\draw[->] (iSS.-130)->(iSSS.130) node [pos=0.5, auto=left, inner sep=1pt] {$S=3$};

\draw[preaction={draw=white, line width=4pt} ,->] (i.-10)->(iT.170) node [pos=0.5, auto=left, inner sep=1pt] {$T=2$}; 
\draw[preaction={draw=white, line width=4pt} ,->] (iS.-10)->(iST.170) node [pos=0.5, auto=left, inner sep=1pt] {$T=2$}; 
\draw[preaction={draw=white, line width=4pt} ,->] (iSS.-10)->(iSST.100) node [pos=0.5, auto=left, inner sep=1pt] {$T=2$};

\draw[->] (iT.-150)->(iTS.45) node [pos=0.5, auto=right, inner sep=1pt] {$S=3$};
\draw[->] (iT.-30)->(iTT.135) node [pos=0.5, auto=left, inner sep=1pt] {$T=1$};
\draw[->] (iST.-130)->(iSTS.90) node [pos=0.5, auto=right, inner sep=1pt] {$S=1$};
\draw[->] (iST.-50)->(iSTT.90) node [pos=0.5, auto=left, inner sep=1pt] {$T=3$};
\draw[->] (iTS.-130)->(iTSS.90) node [pos=0.5, auto=right, inner sep=1pt] {$S=2$};
\draw[->] (iTS.-50)->(iTST.90) node [pos=0.5, auto=left, inner sep=1pt] {$T=1$};
\draw[->] (iTT.-130)->(iTTS.90) node [pos=0.5, auto=right, inner sep=1pt] {$S=2$};
\draw[->] (iTT.-50)->(iTTT.90) node [pos=0.5, auto=left, inner sep=1pt] {$T=3$};
\end{tikzpicture}
\caption{Tropical signs $(\varepsilon_1,\varepsilon_2,\varepsilon_3)$ in the case of \eqref{eq: initial conditions for example of tropical signs}. The below subscripts $k$, $s$, $t$ are the ones in \Cref{thm: recursion for tropical signs}.}\label{fig: tropical signs}
\end{figure}
\par
Let $k=K(\mathbf{w})$ and $s=S(\mathbf{w})$. By \eqref{eq: K S T mutation formula by S}, repeating $S$-mutation means repeating mutations in two directions $s$ and $k$. Namely, it holds that
\begin{equation}\label{eq: index S-mutations}
\mathbf{w}S^n=\mathbf{w}[\underset{n \text{terms}}{s,k,s,k,\dots}].
\end{equation}
In particular, for each $i=1,2,3$, by setting $k_0=K([i]) (=i)$ and $s_0=S([i])$, the trunk $\mathcal{T}^{<[i]S^{\infty}}$ consists of the finite sequence of the form $[k_0,s_0,k_0,\dots]$.
\end{example}

The mutation formulas for $c$- and $g$-vectors depend on $\varepsilon_{M}^{\bf w}b_{M'M}^{\bf w}$. We summarize these signs in the following lemma.
\begin{lemma}\label{lem: for mutation formula}
Let ${\bf w} \in \mathcal{T}\setminus\{\emptyset\}$. Then, the following statements hold.
\\
\textup{($a$)} Suppose that ${\bf w}$ is in a trunk. Then, $\varepsilon_{M}^{\bf w}\mathrm{sign}(b_{M'M}^{\bf w})$ are given in the following list.
\begin{equation}
\begin{array}{c||c|c|c|c|c|c}
(M,M') & (K,S) & (K,T) & (S,K) & (S,T) & (T,K) & (T,S) \\
\hline
\varepsilon_{M}^{\bf w}\mathrm{sign}(b_{M'M}^{\bf w})& -1 & 1 & -1 & 1 & 1 & -1 
\end{array}
\end{equation}
\textup{($b$)} Suppose that ${\bf w}$ is in a branch. Then, $\varepsilon_{M}^{\bf w}\mathrm{sign}(b_{M'M}^{\bf w})$ are given in the following list.
\begin{equation}
\begin{array}{c||c|c|c|c|c|c}
(M,M') & (K,S) & (K,T) & (S,K) & (S,T) & (T,K) & (T,S) \\
\hline
\varepsilon_{M}^{\bf w}\mathrm{sign}(b_{M'M}^{\bf w})& -1 & 1 & -1 & 1 & -1 & 1 
\end{array}
\end{equation}
\end{lemma}
To prove this lemma, we mention the following basic facts.
\begin{lemma}[{\cite[Lem.~6.5, Prop.~6.7]{AC25b}}]\label{lem: tropical signs and K,S,T labeling}
Fix an initial mutation direction $i=1,2,3$.
\\
\textup{($a$)} If ${\bf w} \in \mathcal{T}^{< [i]S^{\infty}}$ is in a trunk, then we have
\begin{equation}
\varepsilon_{K}^{\bf w}=-1,
\quad
\varepsilon_{S}^{\bf w}=\varepsilon_{T}^{\bf w}=1.
\end{equation}
\textup{($b$)} Suppose that ${\bf w} \in \mathcal{T}\backslash\{\emptyset\}$ is in a branch. We express ${\bf w}=[i]X$, where $X \in \mathcal{M}$ has at least one letter $T$. Then, we have
\begin{equation}
\varepsilon_{K}^{[i]X}=\varepsilon_{T}^{[i]X}=(-1)^{\#_{T}(X)},
\quad
\varepsilon_{S}^{[i]X}=-(-1)^{\#_{T}(X)},
\end{equation}
where $\#_{T}(X) \in \mathbb{Z}_{\geq 1}$ is the number of the letter $T$ appearing in $X$.
\end{lemma}
\begin{lemma}\label{lem: sign rule of B}
For any distinct $M,M',M''\in \{K,S,T\}$, the following statements hold. 
\\
\textup{($a$)} For any ${\bf w} \in \mathcal{T}\setminus\{\emptyset\}$, we have $\mathrm{sign}(b_{MM'}^{\bf w})=-\mathrm{sign}(b_{M'M}^{\bf w})$.
\\
\textup{($b$)} For any ${\bf w} \in \mathcal{T}\setminus\{\emptyset\}$, we have $\mathrm{sign}(b_{MM'}^{\bf w})=-\mathrm{sign}(b_{MM''}^{\bf w})$ and $\mathrm{sign}(b_{M'M}^{\bf w})=-\mathrm{sign}(b_{M''M}^{\bf w})$.
\end{lemma}
\begin{proof}
The claim ($a$) follows from the fact that $B^{\bf w}$ is sign-skew-symmetric. To prove that $\mathrm{sign}(b_{MM'}^{\bf w})=-\mathrm{sign}(b_{MM''}^{\bf w})$, we consider the $M({\bf w})$th row of \eqref{eq: definition of cluster cyclic}. Since $M'({\bf w}), M''({\bf w}) \neq M({\bf w})$, exactly one of the corresponding entries is positive and the other is negative. Hence, this implies $\mathrm{sign}(b_{MM'}^{\bf w})=-\mathrm{sign}(b_{MM''}^{\bf w})$. By focusing on the $M({\bf w})$th column instead, we similarly obtain $\mathrm{sign}(b_{M'M}^{\bf w})=-\mathrm{sign}(b_{M''M}^{\bf w})$.
\end{proof}

\begin{proof}[Proof of \Cref{lem: for mutation formula}]
In the case of $(M,M')=(S,K)$ for any $\mfw$, this follows from the second condition of \eqref{eq: condition of S}. If we know this special case, the others are automatically determined by considering the following two facts.
\begin{itemize}
\item[($A$)] By Lemma~\ref{lem: tropical signs and K,S,T labeling}, if ${\bf w}$ is in a trunk, we have $\varepsilon_{K}^{\bf w} = -\varepsilon_{S}^{\bf w}=-\varepsilon_{T}^{\bf w}$. If ${\bf w}$ is in a branch, we have $\varepsilon_{K}^{\bf w} = -\varepsilon_{S}^{\bf w} = \varepsilon_{T}^{\bf w}$.
\item[($B$)] By \Cref{lem: sign rule of B}, for any distinct three indices $M,M',M'' \in \{K,S,T\}$, we obtain that $\mathrm{sign}(b_{MM'}^{\bf w})=-\mathrm{sign}(b_{M'M}^{\bf w})$, $\mathrm{sign}(b_{MM'}^{\bf w}) = -\mathrm{sign}(b_{MM''}^{\bf w})$, and $\mathrm{sign}(b_{M'M}^{\bf w}) = -\mathrm{sign}(b_{M''M}^{\bf w})$.
\end{itemize}
For example, $\varepsilon_{K}^{\bf w}\mathrm{sign}(b_{TK}^{\bf w})=1$ can be shown as follows:
\begin{equation}
\varepsilon_{K}^{\bf w}\mathrm{sign}(b_{TK}^{\bf w})\overset{(A)}{=}-\varepsilon_{S}^{\bf w}\mathrm{sign}(b_{TK}^{\bf w})\overset{(B)}{=}\varepsilon_{S}^{\bf w}\mathrm{sign}(b_{SK}^{\bf w})\overset{(B)}{=}-\varepsilon_{S}^{\bf w}\mathrm{sign}(b_{KS}^{\bf w})=1,
\end{equation}
where the last equality may be shown by the result of $(M,M')=(S,K)$.
\end{proof}

\subsection{Mutation formulas}\label{sec: mutation formulas}
To study $G$-fan structures, we introduce modified $c$--vectors $\tilde{\bf c}^{\bf w}_{i}$ and modified $g$-vectors $\tilde{\bf g}^{\bf w}_{i}$. For this purpose, the skew-symmetric matrices $\tilde{B}$ defined in \Cref{def: skew-symmetrized matrix} are more convenient and useful than the original ones. We write
\begin{equation}\label{eq: definition of p alpha}
p^{\bf w}_{ij}=\sqrt{|b_{ij}^{\bf w}b_{ji}^{\bf w}|}=|\tilde{b}_{ij}^{\bf w}|,
\quad
\alpha^{\mfw}_{ij}=\alpha(p^{\bf w}_{ij})=\frac{1}{2}\left(p^{\bf w}_{ij}+\sqrt{(p^{\bf w}_{ij})^2-4}\right),
\end{equation}
where $\tilde{b}_{ij}^{\bf w}$ is the $(i,j)$th entry of $\tilde{B}^{\bf w}$. By definition, it holds that
\begin{equation}\label{eq: minimum polynomial of alpha}
(\alpha_{ij}^{\bf w})^2 - p_{ij}^{\bf w} \alpha_{ij}^{\bf w}+1 = 0.
\end{equation} When we consider ${\bf w}=\emptyset$, for brevity, we omit it. Namely, we write $b_{ij}=b_{ij}^{\emptyset}$, $p_{ij}=p_{ij}^{\emptyset}=\sqrt{|b_{ij}b_{ji}|}$ and $\alpha_{ij}=\alpha^{\emptyset}_{ij}$.
\par
As in \Cref{def: K S T labeling}, for any $M,M'=K,S,T$ and ${\bf w} \in \mathcal{T}\backslash\{\emptyset\}$, we write $p^{\bf w}_{MM'}=p^{\bf w}_{M({\bf w})M'({\bf w})}$, $\alpha^{\mfw}_{MM'}=\alpha^{\mfw}_{M(\mfw)M'(\mfw)}$, $\tilde{\bf g}^{\bf w}_{M}=\tilde{\bf g}^{\bf w}_{M({\bf w})}$, and so on. Under these notations, the mutation rules \eqref{eq: mutation of b} and \eqref{eq: mutation of modified c- g-vectors} can be expressed as follows.
\begin{lemma}[$S$-mutation]\label{lem: S-mutations of modified c- g-vectors}
Let ${\bf w} \in \mathcal{T}\backslash\{\emptyset\}$. Then, the $S$-mutation data are given by
\begin{equation}\label{eq: S-mutations of p}
p_{MM'}^{{\bf w}S}=\begin{cases}
p_{SK}^{\bf w} & (M,M')=(K,S),(S,K),\\
p_{KS}^{\bf w}p_{ST}^{\bf w}-p_{KT}^{\bf w}& (M,M')=(S,T),(T,S),\\
p_{ST}^{\bf w} & (M,M')=(K,T),(T,K),
\end{cases}
\end{equation}
\begin{equation}\label{eq: S mutation of c- g-vectors}
\tilde{\bf c}^{{\bf w}S}_{M}=\begin{cases}
-\tilde{\bf c}^{\bf w}_{S} & M=K,\\
\tilde{\bf c}^{\bf w}_{K}+p_{SK}^{\bf w}\tilde{\bf c}_{S}^{\bf w} & M=S,\\
\tilde{\bf c}^{\bf w}_{T} & M=T,
\end{cases}
\quad
\tilde{\bf g}^{{\bf w}S}_{M}=\begin{cases}
-\tilde{\bf g}_{S}^{\bf w}+p_{SK}^{\bf w}\tilde{\bf g}_{K}^{\bf w} & M=K,\\
\tilde{\bf g}_{K}^{\bf w} & M=S,\\
\tilde{\bf g}_{T}^{\bf w} & M=T.
\end{cases}
\end{equation}
\end{lemma}
\begin{lemma}[$T$-mutation]\label{lem: T-mutations of modified c- g-vectores}
Let ${\bf w} \in \mathcal{T}\backslash\{\emptyset\}$. Then, the $T$-mutation data are given by the following rules.
\\
\textup{($a$)} If ${\bf w}$ is in a trunk, we have
\begin{equation}\label{eq: T mutation of B in trunks}
p_{MM'}^{{\bf w}T}=\begin{cases}
p_{TS}^{\bf w} & (M,M')=(K,S),(S,K),\\
p_{ST}^{\bf w}p_{TK}^{\bf w}-p_{SK}^{\bf w} & (M,M')=(S,T),(T,S),\\
p_{KT}^{\bf w} & (M,M')=(K,T),(T,K),\\
\end{cases}
\end{equation}
\begin{equation}\label{eq: T mutation of c- g-vectors in trunks}
\tilde{\bf c}^{{\bf w}T}_{M}=\begin{cases}
-\tilde{\bf c}^{\bf w}_{T} & M=K,\\
\tilde{\bf c}^{\bf w}_{S}+p_{ST}^{\bf w}\tilde{\bf c}_{T}^{\bf w} & M=S,\\
\tilde{\bf c}^{\bf w}_{K} & M=T,
\end{cases}
\quad
\tilde{\bf g}^{{\bf w}T}_{M}=\begin{cases}
-\tilde{\bf g}_{T}^{\bf w}+p_{ST}^{\bf w}\tilde{\bf g}_{S}^{\bf w} & M=K,\\
\tilde{\bf g}_{S}^{\bf w} & M=S,\\
\tilde{\bf g}_{K}^{\bf w} & M=T.
\end{cases}
\end{equation}
\textup{($b$)} If ${\bf w}$ is in a branch, we have
\begin{equation}\label{eq: T mutation of B in branches}
p_{MM'}^{{\bf w}T}=\begin{cases}
p_{TK}^{\bf w} & (M,M')=(K,S),(S,K),\\
p_{KT}^{\bf w}p_{TS}^{\bf w}-p_{KS}^{\bf w} & (M,M')=(S,T),(T,S),\\
p_{ST}^{\bf w} & (M,M')=(K,T),(T,K),\\
\end{cases}
\end{equation}
\begin{equation}\label{eq: T mutation of c- g-vectors in branches}
\tilde{\bf c}^{{\bf w}T}_{M}=\begin{cases}
-\tilde{\bf c}^{\bf w}_{T} & M=K,\\
\tilde{\bf c}^{\bf w}_{K}+p_{KT}^{\bf w}\tilde{\bf c}_{T}^{\bf w} & M=S,\\
\tilde{\bf c}^{\bf w}_{S} & M=T,
\end{cases}
\quad
\tilde{\bf g}^{{\bf w}T}_{M}=\begin{cases}
-\tilde{\bf g}_{T}^{\bf w}+p_{KT}^{\bf w}\tilde{\bf g}_{K}^{\bf w} & M=K,\\
\tilde{\bf g}_{K}^{\bf w} & M=S,\\
\tilde{\bf g}_{S}^{\bf w} & M=T.
\end{cases}
\end{equation}
\end{lemma}

\begin{proof}[Proof of \Cref{lem: S-mutations of modified c- g-vectors} and \Cref{lem: T-mutations of modified c- g-vectores}]
By using \Cref{lem: K S T mutation formula} and \Cref{lem: for mutation formula}, we can check each formula. Here, we show $p_{ST}^{{\bf w}S}=p_{KS}^{\bf w}p_{ST}^{\bf w}-p_{KT}^{\bf w}$ and $\tilde{\bf c}^{{\bf w}S}_{S}=\tilde{\bf c}^{\bf w}_{K}+p_{SK}^{\bf w}\tilde{\bf c}_{S}^{\bf w}$, and the others can be shown by a similar argument. 
\par
By \Cref{lem: K S T mutation formula}, we have $S({\bf w}S)=K({\bf w})$ and $T({\bf w}S)=T({\bf w})$.
Let $\tau=\mathrm{sign}(\tilde{b}_{KT}^{\bf w})$. Then, we have $p_{KT}^{\bf w}=\tau\tilde{b}_{KT}^{\bf w}$. By \eqref{eq: sign reversed relation}, it implies that $\mathrm{sign}(\tilde{b}^{{\bf w}S}_{ST})=\mathrm{sign}(\tilde{b}^{{\bf w}S}_{K({\bf w}),T({\bf w})})=-\mathrm{sign}(\tilde{b}_{KT}^{\bf w})=-\tau$ and $p_{ST}^{{\bf w}S}=-\tau\tilde{b}_{ST}^{{\bf w}S}$.
Moreover, by \Cref{lem: sign rule of B}, it holds that $\mathrm{sign}(\tilde{b}_{KS}^{\bf w})=-\tau$ and $\mathrm{sign}(\tilde{b}_{ST}^{\bf w})=-\tau$. In particular, $\tilde{b}_{KS}^{\bf w}\tilde{b}_{ST}^{\bf w} \geq 0$ holds. By \eqref{eq: mutation of b}, we have
\begin{equation}
\begin{aligned}
p_{ST}^{{\bf w}S}=&-\tau\tilde{b}_{ST}^{{\bf w}S}=-\tau \tilde{b}_{K({\bf w}),T({\bf w})}^{{\bf w}[S({\bf w})]}=-\tau(\tilde{b}_{KT}^{\bf w}+\mathrm{sign}(\tilde{b}_{KS}^{\bf w})|\tilde{b}_{KS}^{\bf w}\tilde{b}_{ST}^{\bf w}|)\\
=&-\tau(\tau p_{KT}^{\bf w} - \tau p_{KS}^{\bf w}p_{ST}^{\bf w})=p_{KS}^{\bf w}p_{ST}^{\bf w}-p_{KT}^{\bf w}.
\end{aligned}
\end{equation}
\par
By \Cref{lem: K S T mutation formula}, we have $S({\bf w}S)=K({\bf w})$. Thus, by \eqref{eq: mutation of modified c- g-vectors}, it implies that 
\begin{equation}
\tilde{\bf c}^{{\bf w}S}_{S}=\tilde{\bf c}_{K({\bf w})}^{{\bf w}[S({\bf w})]}=\tilde{\bf c}^{\bf w}_{K}+[\varepsilon_{S}^{\bf w}\tilde{b}_{SK}^{\bf w}]_{+}\tilde{\bf c}_{S}^{\bf w}.
\end{equation}
Since $\tilde{b}_{SK}^{\bf w}=-\tilde{b}_{KS}^{\bf w}$, we have $[\varepsilon_{S}^{\bf w}\tilde{b}_{SK}^{\bf w}]_{+}=[-\varepsilon_{S}^{\bf w}\tilde{b}_{KS}^{\bf w}]_{+}$. By \Cref{lem: for mutation formula} in the case of $(M,M')=(S,K)$, we obtain $\varepsilon_{S}^{\bf w}\tilde{b}_{KS}^{\bf w} < 0$. It implies $[-\varepsilon_{S}^{\bf w}\tilde{b}_{KS}^{\bf w}]_{+}=|\tilde{b}_{KS}^{\bf w}|=p_{KS}^{\bf w}$. Thus, we have $\tilde{\bf c}^{{\bf w}S}_{S}=\tilde{\bf c}^{\bf w}_{K}+p_{SK}^{\bf w}\tilde{\bf c}_{S}^{\bf w}$.
\end{proof}

\begin{example}
In the stereographic projection, according to \Cref{lem: S-mutations of modified c- g-vectors} and \Cref{lem: T-mutations of modified c- g-vectores}, the mutation rules of $g$-vectors can be understood as in \Cref{fig: mutations in trunks} and \Cref{fig: mutations in branches}. Note that the thick colored lines are on the same plane. By \eqref{eq: orthogonal relation for modified vectors}, each plane is orthogonal to a $c$-vector.
\begin{figure}[htbp]
\centering
\begin{minipage}[b]{0.45\textwidth}
\centering
\begin{tikzpicture}
\coordinate (K0) at (1.0116,-1.5575);
\coordinate (T0) at (-0.8966,-0.5176);
\coordinate (S0) at (0.8966,-0.5176);
\coordinate (KS) at (0.8526,-2.4612);
\coordinate (SS) at (K0);
\coordinate (TS) at (T0);
\coordinate (KT) at (1.8547,-0.0973);
\coordinate (ST) at (S0);
\coordinate (TT) at (K0);

\draw (1.0116,-1.5575) arc [radius=3.4641, start angle=-2.3711, end angle=15.0];
\draw (0.8966,-0.5176) arc [radius=3.4641, start angle=-75.0, end angle=-105.0];
\draw (-0.8966,-0.5176) arc [radius=2.4596, start angle=-144.8056, end angle=-92.3711];

\draw (0.8526,-2.4612) arc [radius=3.4641, start angle=-17.5926, end angle=-2.3711];
\draw (1.0116,-1.5575) arc [radius=2.4596, start angle=-92.3711, end angle=-144.8056];
\draw (-0.8966,-0.5176) arc [radius=2.582, start angle=-168.4349, end angle=-107.5926];

\draw (1.8547,-0.0973) arc [radius=3.4641, start angle=-57.6289, end angle=-75.0];
\draw (0.8966,-0.5176) arc [radius=3.4641, start angle=15.0, end angle=-2.3711];
\draw (1.0116,-1.5575) arc [radius=2.0069, start angle=-54.8389, end angle=-5.1611];

\draw (K0) node [above left] {\tiny $K$};
\draw (S0) node [below left] {\tiny $S$};
\draw (T0) node [xshift = 15, yshift = -7] {\tiny $T$};

\draw (0.3372,-1.0375) node {\tiny $\mathbf{w}$};

\draw (KS) node [xshift = -3, yshift=6] {\tiny $K$};
\draw (SS) node [below left] {\tiny $S$};
\draw (TS) node [xshift = 15, yshift = -23] {\tiny $T$};

\draw (0.3372,-1.8) node {\tiny $\mathbf{w}S$};

\draw (KT) node [xshift=-7, yshift=-10] {\tiny $K$};
\draw (ST) node [xshift=5, yshift=-2] {\tiny $S$};
\draw (TT) node [xshift=4, yshift=7] {\tiny $T$};

\draw (1.3,-0.9) node {\tiny $\mathbf{w}T$};

\draw[red, thick] (0.8966,-0.5176) arc [radius=3.4641, start angle=15.0, end angle=-17.5926];
\draw[blue, thick] (-0.8966,-0.5176) arc [radius=3.4641, start angle=-105.0, end angle=-57.6289];
\end{tikzpicture}
\caption{In trunks.}\label{fig: mutations in trunks}
\end{minipage}
\begin{minipage}[b]{0.45\textwidth}
\centering
\begin{tikzpicture}
\coordinate (K0) at (1.8547,-0.0973);
\coordinate (S0) at (0.8966,-0.5176);
\coordinate (T0) at (1.0116,-1.5575);

\coordinate (KS) at (2.7099,0.6705);
\coordinate (SS) at (K0);
\coordinate (TS) at (T0);

\coordinate (KT) at (1.517,1.2095);
\coordinate (ST) at (K0);
\coordinate (TT) at (S0);

\draw (1.8547,-0.0973) arc [radius=3.4641, start angle=-57.6289, end angle=-75.0];
\draw (0.8966,-0.5176) arc [radius=3.4641, start angle=15.0, end angle=-2.3711];
\draw (1.0116,-1.5575) arc [radius=2.0069, start angle=-54.8389, end angle=-5.1611];

\draw (2.7099,0.6705) arc [radius=3.4641, start angle=-38.5307, end angle=-57.6289];
\draw (1.8547,-0.0973) arc [radius=2.0069, start angle=-5.1611, end angle=-54.8389];
\draw (1.0116,-1.5575) arc [radius=2.1483, start angle=-78.0078, end angle=3.3782];

\draw (1.517,1.2095) arc [radius=2.0069, start angle=34.1393, end angle=-5.1611];
\draw (1.8547,-0.0973) arc [radius=3.4641, start angle=-57.6289, end angle=-75.0];
\draw (0.8966,-0.5176) arc [radius=2.498, start angle=-41.3099, end angle=1.7918];

\draw (K0) node [xshift=-7, yshift=-10] {\tiny $K$};
\draw (S0) node [xshift=5, yshift=-2] {\tiny $S$};
\draw (T0) node [xshift=4, yshift=7] {\tiny $T$};

\draw (1.3,-0.9) node {\tiny $\mathbf{w}$};

\draw (KS) node [xshift=-7, yshift=-13] {\tiny $K$};
\draw (SS) node [xshift=5, yshift=-1] {\tiny $S$};
\draw (TS) node [xshift=20, yshift=15] {\tiny $T$};

\draw (2.1,-0.5) node {\tiny $\mathbf{w}S$};

\draw (KT) node [xshift=2, yshift=-16] {\tiny $K$};
\draw (ST) node [xshift=-4, yshift=3] {\tiny $S$};
\draw (TT) node [xshift=10, yshift=8] {\tiny $T$};

\draw (1.6,0.3) node {\tiny $\mathbf{w}T$};

\draw[red, thick] (0.8966,-0.5176) arc [radius=3.4641, start angle=-75.0, end angle=-38.5307];
\draw[blue, thick] (1.0116,-1.5575) arc [radius=2.0069, start angle=-54.8389, end angle=34.1393];
\end{tikzpicture}
\caption{In branches.}\label{fig: mutations in branches}
\end{minipage}
\end{figure}
\end{example}

\subsection{Initial setup}
We fix the initial mutation $i\in \{1,2,3\}$. After a single mutation $[i]$, the mutation rule is described in  \Cref{lem: S-mutations of modified c- g-vectors} and \Cref{lem: T-mutations of modified c- g-vectores}. In contrast, the first mutation should be computed directly. We give some relevant formulas for the initial mutation as follows.
\par
Let $k_0=K([i])$, $s_0=S([i])$, and $t_0=T([i])$. By the same argument in \Cref{ex: tropical signs}, these indices are listed in \Cref{tab: List of initial indices}.
\begin{table}[htbp]
\captionsetup{skip=8pt}
\centering
\begin{tabular}{c||ccc|ccc}
$B$
&
\multicolumn{3}{c|}{$\left(\begin{smallmatrix}
0 & - & +\\
+ & 0 & -\\
- & + & 0
\end{smallmatrix}\right)$}
& 
\multicolumn{3}{c}{$\left(\begin{smallmatrix}
0 & + & -\\
- & 0 & +\\
+ & - & 0
\end{smallmatrix}\right)$}
\\
$i$ & $1$ & $2$ & $3$ & $1$ & $2$ & $3$
\\
\hline
$(k_0,s_0,t_0)$ 
&
$(1,3,2)$
&
$(2,1,3)$
&
$(3,2,1)$
&
$(1,2,3)$
&
$(2,3,1)$
&
$(3,2,1)$
\end{tabular}
\caption{The list of $k_0=K([i])$, $s_0=S([i])$, $t_0=T([i])$.}
\label{tab: List of initial indices}
\end{table}
\par
For a given initial exchange matrix $B$, if we consider the relation between two different initial mutation directions $i \neq j$, the following fact is useful:
\begin{equation}
S([i]) \neq S([j]), \quad T([i]) \neq T([j]).
\end{equation}
For example, set $k_0=K([i])$, $s_0=S([i])$, and $t_0=T([i])$. For the initial mutation $j=s_0$, since $T([j]) \neq t_0$, we have
\begin{equation}\label{eq: i to s0 indices}
K([j])=s_0,
\quad
S([j])=t_0,
\quad
T([j])=k_0.
\end{equation}
Similarly, for $j=t_0$, since $S([j]) \neq s_0$, we have
\begin{equation}\label{eq: i to t0 indices}
K([j])=t_0,
\quad
S([j])=k_0,
\quad
T([j])=s_0.
\end{equation}
\par
As explained in Subsection~\ref{sec: mutation formulas}, we omit the empty sequence $\emptyset=[\ ]$ in the superscript. By using the data at $\emptyset$, we give the expressions for $p_{MM'}^{[i]}$, $\tilde{\mathbf{g}}_{M}^{[i]}$, and $\tilde{\mathbf{c}}_{M}^{[i]}$.
\begin{lemma}\label{lem: kst}
The following equalities hold:
\begin{equation}\label{eq: initial mutations}
\left\{\begin{aligned}
p_{ST}^{[i]}&=p_{k_0s_0}p_{k_0t_0}-p_{s_0t_0},
\\
p_{KT}^{[i]}&=p_{k_0t_0},
\\
p_{KS}^{[i]}&=p_{k_0s_0},
\end{aligned}\right.
\quad
\left\{\begin{aligned}
\tilde{\bf g}_{K}^{[i]}&=-\tilde{\bf e}_{k_0}+p_{k_0s_0}\tilde{\bf e}_{s_0},
\\
\tilde{\bf g}_{S}^{[i]}&=\tilde{\bf e}_{s_0},
\\
\tilde{\bf g}_{T}^{[i]}&=\tilde{\bf e}_{t_0},
\end{aligned}\right.
\quad
\left\{\begin{aligned}
\tilde{\bf c}_{K}^{[i]}&=-\tilde{\bf e}_{k_0},
\\
\tilde{\bf c}_{S}^{[i]}&=\tilde{\bf e}_{s_0}+p_{k_0s_0}\tilde{\bf e}_{k_0},
\\
\tilde{\bf c}_{T}^{[i]}&=\tilde{\bf e}_{t_0}.
\end{aligned}\right.
\end{equation}
\end{lemma}
\begin{proof}
This follows from the same argument in the proofs of \Cref{lem: S-mutations of modified c- g-vectors} and \Cref{lem: T-mutations of modified c- g-vectores}.
\end{proof}

\section{Global upper bound}\label{sec: global}
From now on, we fix an cluster-cyclic initial exchange matrix $B \in \mathrm{M}_3(\mathbb{R})$ together with its skew-symmetrizer $D=\mathrm{diag}(d_1,d_2,d_3)$.
In this section, we exhibit the global upper bounds for the $G$-fan, which can be referred to the thick blue lines in \Cref{fig: main graph}.
\subsection{Quadratic surface for $g$-vectors}
In \cite{LL24}, a restriction on $g$-vectors was established.
To state it, we introduce the following notion.
\begin{definition}
For any cyclic matrix $B=(b_{ij}) \in \mathrm{M}_3(\mathbb{R})$, we define the {\em pseudo Cartan companion} of $B$ as
\begin{equation}\label{eq: definition of pseudo Cartan companion}
A=\left(\begin{matrix}
2 & |b_{12}| & |b_{13}|\\
|b_{21}| & 2 & |b_{23}|\\
|b_{31}| & |b_{32}| & 2
\end{matrix}\right).
\end{equation}
For any cluster-cyclic initial exchange matrix $B$ and any reduced sequence $\mathbf{w} \in \mathcal{T}$, we write the pseudo Cartan companion of $B^{\mathbf{w}}$ as $A^{\mathbf{w}}$.
\end{definition}
\begin{remark}\label{rem: Seven's result}
In \cite[Thm.~2.6]{Sev12}, the following classification of the cluster-cyclicity was established: a given cyclic matrix $B \in \mathrm{M}_3(\mathbb{R})$ is cluster-cyclic if and only if the corresponding pseudo Cartan companion $A$ has exactly one positive eigenvalue and two non-positive eigenvalues. Moreover, the following relationship between the pseudo Cartan companion and the Markov constant was found.
\begin{equation}\label{eq: Seven's equation}
\det A=2(4-C(B)).
\end{equation}
\end{remark}
Each pseudo Cartan companion $A$ is a symmetrizable matrix with the same symmetrizer $D$ of $B$. Now, we can exhibit the following restriction. Note that in \cite{LL24}, it is constructed in the skew-symmetric case. Here, we generalize this result to the skew-symmetrizable case.
\begin{lemma}[cf.~{\cite{LL24}}]\label{lem: surface for g-vectors}
Let $B \in \mathrm{M}_3(\mathbb{R})$ be a cluster-cyclic initial exchange matrix, and let $A^{\mathbf{w}}$ be the pseudo Cartan companion of $B^{\mathbf{w}}$.
Fix an initial mutation direction $i=1,2,3$. Then, for any ${\bf w} \in \mathcal{T}^{\geq [i]}$, we have
\begin{equation}\label{eq: congruence relation for pseudo Cartan companion}
({G}^{\bf w})^{\top}D{A}^{[i]}{G}^{\bf w}=D{A}^{{\bf w}K},
\end{equation}
where ${{\bf w}K}=\mathbf{w}[K(\mathbf{w})]$ is the sequence obtained by deleting the last index of ${\bf w}$. In particular, every $g$-vector ${\bf g}_{l}^{\bf w}=(x_1,x_2,x_3)^{\top}$ \textup{($l=1,2,3$)} satisfies
\begin{equation}\label{eq: quadratic equation for g-vectors}
d_1x_1^2+d_2x_2^2+d_3x_3^2+d_1|b_{12}^{[i]}|x_1x_2+d_2|b_{23}^{[i]}|x_2x_3+d_3|b_{31}^{[i]}|x_3x_1 = d_{l}.
\end{equation}
\end{lemma}
\begin{proof}
By doing the same argument in \cite{LL24}, we obtain this result for any {\em real} skew-symmetric matrix with $d_1=d_2=d_3=1$. Let $\tilde{B}=D^{\frac{1}{2}}BD^{-\frac{1}{2}}$. We write the $g$-vector with the initial exchange matrix $\tilde{B}$ by $\hat{G}^{\bf w}$. Then, for any $\mathbf{w} \in \mathcal{T}^{\geq [i]}$, by the result for the skew-symmetric case, we have
\begin{equation}\label{eq: result for skew-symmetric case}
(\hat{G}^{\bf w})^{\top}\tilde{A}^{[i]}\hat{G}^{\bf w}=\tilde{A}^{{\bf w}K},
\end{equation}
where $\tilde{A}^{[i]}$ and $\tilde{A}^{\mathbf{w}K}$ are the pseudo Cartan companion of $\tilde{B}^{[i]}$ and $\tilde{B}^{\mathbf{w}K}$.
By Lemma~\ref{lem: skew-symmetrized matrices formula}, we have $\tilde{B}^{[i]}=D^{\frac{1}{2}}B^{[i]}D^{-\frac{1}{2}}$ and $\tilde{B}^{{\bf w}K}=D^{\frac{1}{2}}B^{\mathbf{w}K}D^{-\frac{1}{2}}$. Thus, we obtain $\tilde{A}^{[i]}=D^{\frac{1}{2}}A^{[i]}D^{-\frac{1}{2}}$ and $\tilde{A}^{{\bf w}K}=D^{\frac{1}{2}}A^{{\bf w}K}D^{-\frac{1}{2}}$. By \eqref{eq: skew-symmetrizing method}, we have $\hat{G}^{\mathbf{w}}=D^{\frac{1}{2}}G^{\mathbf{w}}D^{-\frac{1}{2}}$. By substituting these equalities into \eqref{eq: result for skew-symmetric case}, we obtain $(G^{\bf w})^{\top}DA^{[i]}G^{\bf w}=DA^{{\bf w}K}$. By focusing on the diagonal entries of this equality, we may obtain \eqref{eq: quadratic equation for g-vectors}.
\end{proof}
We state this result by using modified $g$-vectors. To this end, we define $\tilde{A}^{\mathbf{w}}=D^{\frac{1}{2}}A^{\mathbf{w}}D^{-\frac{1}{2}}$, which is the pseudo Cartan companion of $\tilde{B}^{\bf w}=D^{\frac{1}{2}}B^{\mathbf{w}}D^{-\frac{1}{2}}$.
\begin{lemma}\label{lem: surface for modified g-vectors}
Fix an initial mutation direction $i=1,2,3$. For any ${\bf w} \in \mathcal{T}^{\geq [i]}$, every modified $G$-matrix $\tilde{G}^{\bf w}$ satisfies
\begin{equation}
(D^{\frac{1}{2}}\tilde{G}^{\bf w})^{\top}\tilde{A}^{[i]}(D^{\frac{1}{2}}\tilde{G}^{\bf w})=\tilde{A}^{{\bf w}K}.
\end{equation}
In particular, every modified $g$-vector $\tilde{g}^{\bf w}_{l}$ \textup{($l=1,2,3$)} satisfies
\begin{equation}\label{eq: quadratic form for modified g vectors}
(D^{\frac{1}{2}}\tilde{\bf g}_{l}^{\bf w})^{\top}\tilde{A}^{[i]}(D^{\frac{1}{2}}\tilde{\bf g}_{l}^{\bf w})=2.
\end{equation}
\end{lemma}
\begin{proof}
By substituting $G^{\bf w}= \tilde{G}^{\bf w}D^{\frac{1}{2}}$ into \eqref{eq: congruence relation for pseudo Cartan companion}, we obtain the claim.
\end{proof}
The equality \eqref{eq: quadratic form for modified g vectors} states that, after the initial mutation $i=1,2,3$, all modified $g$-vectors lie on the same quadratic surface. For later use, we start by studying more detailed properties of this surface.

\subsection{General fact of the quadratic surface}\label{sec: preliminary for the quadratic form}
In this subsection, we restrict our attention to the region indicated by the quadratic equation
\begin{equation}\label{eq: quadratic form}
x_1^2+x_2^2+x_3^2 +p_{12}x_1x_2+p_{23}x_2x_3+p_{31}x_{3}x_{1} =1,
\end{equation}
and we temporarily set aside the cluster algebraic structures, such as $B$-matrices and $g$-vectors. Namely, we fix $p_{12},p_{23},p_{31} \geq 2$ satisfying $C(p_{12},p_{23},p_{31})=p_{12}^2+p_{23}^2+p_{31}^2 - p_{12}p_{23}p_{31} \leq 4$. To simplify the statement, we set $p_{ji}=p_{ij}$ and
\begin{equation}
\tilde{A}=\left(\begin{matrix}
2 & p_{12} & p_{13}\\
p_{21} & 2 & p_{23}\\
p_{31} & p_{32} & 2
\end{matrix}\right).
\end{equation}
Note that the equality \eqref{eq: quadratic form} may be expressed as $\frac{1}{2}{\bf x}^{\top} \tilde{A}{\bf x}=1$. (This $\tilde{A}$ corresponds to the pseudo Cartan companion of $\tilde{B}$.)
Then, the region corresponding to the equality \eqref{eq: quadratic form} can be depicted as follows.
\begin{lemma}\label{lem: eigenvalues lemma for cluster-cyclic matrix}
We have the following claims.
\\
\textup{($a$)} The matrix $\tilde{A}$ has three real eigenvalues $\lambda > \nu_1 \geq \nu_2$, where $\lambda$ is strictly positive and $\nu_1,\nu_2$ are non-positive allowing $\nu_1=\nu_2$. Moreover, $\lambda$ satisfies
\begin{equation}\label{eq: lambda inequality}
\lambda > 2+\sqrt{p_{12}^2+p_{23}^{2}+p_{31}^{2}}.
\end{equation}
\textup{($b$)} An eigenvector of $\tilde{A}$ with respect to $\lambda$ is expressed as
\begin{equation}\label{eq: positive eigenvector of lambda}
{\bf v}=\left(\begin{matrix}
(\lambda-2)^2+(p_{12}+p_{13})(\lambda-2)+p_{12}p_{23}+p_{13}p_{32}-p_{23}^2
\\
(\lambda-2)^2+(p_{23}+p_{21})(\lambda-2)+p_{23}p_{31}+p_{21}p_{13}-p_{31}^2
\\
(\lambda-2)^2+(p_{31}+p_{32})(\lambda-2)+p_{31}p_{12}+p_{32}p_{21}-p_{12}^2
\end{matrix}\right),
\end{equation}
and every component of this vector is strictly positive.
\\
\textup{($c$)} The region of $\frac{1}{2}{\bf x}^{\top}\tilde{A}{\bf x}=1$ is any of hyperboloid of two sheets (when $\nu1,\nu2 < 0$), hyperbolic cylinder (when $\nu_1=0$, $\nu_2< 0$), or real parallel planes (when $\nu_1=\nu_2=0$) in the sense of \cite[\S~4.20]{Zwi95}. In particular, for each $\tilde{A}$, there are two connected components, and they are separated by the plane $\Ezeroregion{{\bf v}}$ orthogonal to $\mathbf{v}$. 
\end{lemma}
\begin{figure}[htbp]
\centering
\begin{minipage}{0.49\linewidth}
\centering
\begin{tikzpicture}
    \begin{axis}[
        hide axis,
        title={},
        xlabel={},
        ylabel={},
        zlabel={},
        domain=-2:2,
        y domain=-2:2,
        samples=20,
        view={35}{-5}, 
    ]
    
    \addplot3[
        surf,
        fill=gray,
        opacity=0.5,
        faceted color= gray
    ] {-(x+y)+2};

    \addplot3[
        surf,
        fill=gray,
        opacity=0.5,
        faceted color= gray
    ] {-(x+y)-2};

    \addplot3[->, thin, -stealth] coordinates {(0,0,0) (2,0,0)} node [right] {$x_1$};
    \addplot3[->, thin, -stealth] coordinates {(0,0,0) (0,2,0)} node [right] {$x_2$};
    \addplot3[->, thin, -stealth] coordinates {(0,0,0) (0,0,4)} node [right] {$x_3$};

    \addplot3[only marks, mark size=0.6pt, black] coordinates {(0,0,0)} node [left] {$O$};

    \draw (0,0,5) node [left] {$H^{+}$};
    \draw (0,0,-4) node [left] {$H^{-}$};
    \end{axis}
\end{tikzpicture}
\caption{Real parallel planes.}\label{fig: two plane}
\end{minipage}
\begin{minipage}{0.49\linewidth}
\centering
\begin{tikzpicture}
    \begin{axis}[
        hide axis,
        title={},
        xlabel={},
        ylabel={},
        zlabel={},
        domain=-2:2,
        y domain=-2:2,
        samples=20,
        view={35}{-5}, 
    ]
    
    \addplot3[
        surf,
        fill=gray,
        opacity=0.5,
        faceted color= gray
    ] {(-4*(x+y)+sqrt(12*x^2+16*x*y+12*y^2+32))/2};

    \addplot3[
        surf,
        fill=gray,
        opacity=0.5,
        faceted color= gray
    ] {(-4*(x+y)-sqrt(12*x^2+16*x*y+12*y^2+32))/2};

    \addplot3[->, thin, -stealth] coordinates {(0,0,0) (2,0,0)} node [right] {$x_1$};
    \addplot3[->, thin, -stealth] coordinates {(0,0,0) (0,2,0)} node [right] {$x_2$};
    \addplot3[->, thin, -stealth] coordinates {(0,0,0) (0,0,4)} node [right] {$x_3$};

    \addplot3[only marks, mark size=0.6pt, black] coordinates {(0,0,0)} node [left] {$O$};

    \draw (0,0,10) node [left] {$H^{+}$};
    \draw (0,0,-8) node [left] {$H^{-}$};
    \end{axis}
\end{tikzpicture}
\caption{Hyperboloid of two sheets.}\label{fig: hyperboloid}
\end{minipage}
\end{figure}
\begin{remark}
According to the classification of quadratic forms, the quadratic form \eqref{eq: quadratic form} falls into one of the three types described in \Cref{lem: eigenvalues lemma for cluster-cyclic matrix}~($c$). However, in this paper, we are not concerned with these distinctions. The only property we require is that the region arising from the quadratic form \eqref{eq: quadratic form} has two connected components separated by the plane $\Ezeroregion{{\bf v}}$.
\end{remark}
As stated in \Cref{rem: Seven's result}, the signs of $\lambda$, $\nu_1$, and $\nu_2$ can be obtained by \cite{Sev12}. However, for the later purpose, we need to know more detailed properties such as \eqref{eq: lambda inequality} and ($b$). Hence, we provide an alternative proof here.
\begin{proof}
Let $f(t)=\det(tI-\tilde{A})$. Then, by using $p_{ij}=p_{ji}$, we may expand it as
\begin{equation}
f(t)=(t-2)^{3}-(p_{12}^2+p_{23}^2+p_{31}^2)(t-2)-2p_{12}p_{23}p_{31}.
\end{equation}
Consider $f'(t)=3(t-2)^2-(p_{12}^2+p_{23}^2+p_{31}^2)$. Then, by $p_{12},p_{23},p_{31} \geq 2$, we have 
\begin{equation}\label{eq: first derivation}
f'(0)=12-(p_{12}^2+p_{23}^2+p_{31}^2) \leq 0.
\end{equation}
Since $f'(t)$ has a positive leading coefficient, one root of $f'(t)=0$ is non-negative and the other root is non-positive. In particular, $f(t)$ is unimodal when $t \geq 0$. By a direct calculation with \eqref{eq: rank 3 skew-symmetrizable condition}, we have 
\begin{equation}\label{eq: Markov constant inequality}
f(0)= -2(4-C(p_{12},p_{23},p_{31})) \leq 0.
\end{equation}
Thus, we may show that there exists precisely one positive solution $\lambda > 0$. Since $\tilde{A}$ is symmetric, there should exist three real eigenvalues, counting multiplicity. Thus, there exist other two eigenvalues $\nu_1,\nu_2 \leq 0$. Moreover, by $f(2+\sqrt{p_{12}^2+p_{23}^2+p_{31}^2})=-2p_{12}p_{23}p_{31}<0$, the unique positive eigenvalue $\lambda$ should satisfy $2+\sqrt{p_{12}^2+p_{23}^2+p_{31}^2}< \lambda$. Next, we show that an eigenvector of $\lambda$ may be given by (\ref{eq: positive eigenvector of lambda}). This is because we can calculate the following equatity:
\begin{equation}\label{eq: eigenvector of DA}
\left(\begin{matrix}
\lambda-2 & -p_{12} & -p_{13}\\
-p_{21} & \lambda-2 & -p_{23}\\
-p_{31} & -p_{32} & \lambda-2
\end{matrix}\right){\bf v}={\bf 0}.
\end{equation}
For example, the equality of the first component may be verified as follows.
\begin{equation}
\begin{aligned}
(\lambda-2,-p_{12},-p_{13}){\bf v}&=(\lambda-2)^3+(p_{12}+p_{13})(\lambda-2)^2+(p_{12}p_{23}+p_{13}p_{32}-p_{23}^2)(\lambda-2)\\
&\qquad -p_{12}(\lambda-2)^2-p_{12}(p_{23}+p_{21})(\lambda-2)-p_{12}(p_{23}p_{31}+p_{21}p_{13}-p_{31}^2)\\
&\qquad -p_{13}(\lambda-2)^2-p_{13}(p_{31}+p_{32})(\lambda-2)-p_{13}(p_{31}p_{12}+p_{32}p_{21}-p_{12}^2)
\\
&=(\lambda-2)^3-(p_{12}^2+p_{23}^2+p_{31}^2)(\lambda-2)-2p_{12}p_{23}p_{31}=0,
\end{aligned}
\end{equation}
where the last equality follows from $f(\lambda)=0$.
Last, we show that all entries of ${\bf v}$ are positive. For example, we may show that the first component is strictly positive by $\lambda >2$ and 
\begin{equation}
(\lambda-2)^2-p_{23}^2>\sqrt{p_{12}^2+p_{23}^2+p_{31}^2}^2-p_{23}^2=p_{12}^2+p_{31}^2>0,
\end{equation}
where the first inequality follows from $\lambda-2 > \sqrt{p_{12}^2+p_{23}^2+p_{31}^2}$.
\end{proof}
\begin{remark}\label{rem: classification of surface type}
From \eqref{eq: first derivation} and \eqref{eq: Markov constant inequality}, we have the following classification of these three surface types.
\begin{itemize}
\item $\nu_1=\nu_2=0 \Longleftrightarrow p_{12}=p_{23}=p_{31}=2$.
\item $\nu_1=0, \nu_2<0 \Longleftrightarrow C(p_{12},p_{23},p_{31})=4$ and $\max(p_{12},p_{23},p_{31})>2$.
\item $\nu_1,\nu_2 < 0 \Longleftrightarrow C(p_{12},p_{23},p_{31})<4$.
\end{itemize}
In particular, the case of real parallel planes ($\nu_1=\nu_2=0$) occurs only when the matrix $B$ is given by \eqref{eq: value-rigid type}.
\end{remark}
For later purposes, we include information on a skew-symmetrizer $D=\mathrm{diag}(d_1,d_2,d_3)$. (However, we treat it as a triple of positive numbers $d_1,d_2,d_3 > 0$ in this subsection.)
\begin{definition}\label{def: for global upper bounds}
For each $\tilde{A}$ and $D=\mathrm{diag}(d_1,d_2,d_3)$, let ${\bf v}$ be the positive eigenvector of $\tilde{A}$ given by (\ref{eq: positive eigenvector of lambda}). 
Define the following sets:
\begin{equation}\label{eq: set Q}
\begin{aligned}
Q&=\{{\bf x} \in \mathbb{R}^3 \mid (D^{\frac{1}{2}}{\bf x})^{\top}\tilde{A}(D^{\frac{1}{2}}{\bf x}) > 0\} \cup \{{\bf 0}\},&
\quad
H&=\left\{{\bf x} \in \mathbb{R}^3\ \left|\ (D^{\frac{1}{2}}{\bf x})^{\top}\tilde{A}(D^{\frac{1}{2}}{\bf x}) =2\right.\right\},
\\
Q^{+}&=\{{\bf x} \in Q \mid  \langle D^{\frac{1}{2}}{\bf x},{\bf v} \rangle \geq 0\},&
\quad
H^{+}&=\{{\bf x} \in H \mid \langle D^{\frac{1}{2}}{\bf x},{\bf v} \rangle \geq 0\},\\
Q^{-}&=\{{\bf x} \in Q \mid \langle D^{\frac{1}{2}}{\bf x},{\bf v} \rangle \leq 0\},&
\quad
H^{-}&=\{{\bf x} \in H \mid \langle D^{\frac{1}{2}}{\bf x},{\bf v} \rangle \leq 0\}.
\end{aligned}
\end{equation}
\end{definition}
Later, we take $D$ to be a skew-symmetrizer of a given skew-symmetrizable matrix $B$. In this case, the three sets $Q, Q^{+}$, and $Q^{-}$ are independent of the choice of $D$ since all skew-symmetrizers of $B$ are identical up to a scalar. However, $H, H^{+}$, and $H^{-}$ depend on $D$ due to this scalar indefiniteness.
\par
Since $\mathbf{v}$ is an eigenvector with respect to the unique positive eigenvalue $\lambda$ of $\tilde{A}$, we have $Q^{+}=\{{\bf x} \in Q \mid  \langle D^{\frac{1}{2}}{\bf x},{\bf v} \rangle > 0\} \cup \{{\bf 0}\}$ and $Q^{-}=\{{\bf x} \in Q \mid  \langle D^{\frac{1}{2}}{\bf x},{\bf v} \rangle < 0\} \cup \{{\bf 0}\}$.

\begin{lemma}\label{lem: convexity of Q}
The sets $Q^{+}$ and $Q^{-}$ are both convex cones.
\end{lemma}
\begin{proof}
We prove the case for $Q^{+}$.
Let $\mathbf{v}'=\frac{\mathbf{v}}{|\mathbf{v}|}$.
Since $\tilde{A}$ is symmetric, there exists an orthogonal eigenbasis $[{\bf v}',{\bf u}_{1},{\bf u}_2]$ corresponding to the eigenvalues $(\lambda,\nu_1,\nu_2)$ of $\tilde{A}$. Note that $[D^{-\frac{1}{2}}{\bf v}',D^{-\frac{1}{2}}{\bf u}_1,D^{-\frac{1}{2}}{\bf u}_2]$ is also a basis of $\mathbb{R}^3$.
Take any elements ${\bf x}_1,{\bf x}_2 \in Q^{+}$ and $a_1,a_2 \in \mathbb{R}_{> 0}$. If either ${\bf x}_1$ or ${\bf x}_2$ is ${\bf 0}$, we can prove $a_1{\bf x}_1+a_2{\bf x}_2 \in Q^{+}$ by a direct calculation. Thus, we may assume ${\bf x}_1,{\bf x}_2 \neq {\bf 0}$. For each $j=1,2$, we express
\begin{equation}\label{eq: matrix expression of xj}
{\bf x}_j=r_{j}D^{-\frac{1}{2}}{\bf v}'+\alpha_{j}D^{-\frac{1}{2}}{\bf u}_{1}+\beta_{j}D^{-\frac{1}{2}}{\bf u}_2=D^{-\frac{1}{2}}({\bf v}',{\bf u}_1,{\bf u}_2)\left(\begin{smallmatrix}
r_j\\
\alpha_j\\
\beta_j
\end{smallmatrix}\right).
\end{equation}
for some $r_j,\alpha_j,\beta_j \in \mathbb{R}$. Since $\langle D^{\frac{1}{2}}{\bf x}_j,{\bf v}' \rangle = r_j$, we have $r_j \geq 0$. Thus, we may show that $\langle D^{\frac{1}{2}}(a_1{\bf x}_1+a_2{\bf x}_2), {\bf v}' \rangle=a_1r_1+a_2r_2 \geq 0$. Now, it suffices to prove that  \begin{align}\left\{D^{\frac{1}{2}}(a_1{\bf x}_1+a_2{\bf x}_2)\right\}^{\top} \tilde{A}\left\{D^{\frac{1}{2}}(a_1{\bf x}_1+a_2{\bf x}_2)\right\} > 0.\end{align} 
In fact, we may express
\begin{equation}
\begin{aligned}
&\ \left\{D^{\frac{1}{2}}(a_1{\bf x}_1+a_2{\bf x}_2)\right\}^{\top}\tilde{A}\left\{D^{\frac{1}{2}}(a_1{\bf x}_1+a_2{\bf x}_2)\right\}
\\
=&\ a_1^2(D^{\frac{1}{2}}{\bf x}_1)^{\top}\tilde{A}(D^{\frac{1}{2}}{\bf x}_1)+2a_1a_2(D^{\frac{1}{2}}{\bf x}_{1})^{\top}\tilde{A}(D^{\frac{1}{2}}{\bf x}_2)+a_2^2(D^{\frac{1}{2}}{\bf x}_2)^{\top}\tilde{A}(D^{\frac{1}{2}}{\bf x}_2)
\\
>&\ 2a_1a_2(D^{\frac{1}{2}}{\bf x}_1)^{\top}\tilde{A}(D^{\frac{1}{2}}{\bf x}_2).
\end{aligned}
\end{equation}
Hence, we only need to prove that $(D^{\frac{1}{2}}{\bf x}_1)^{\top}\tilde{A}(D^{\frac{1}{2}}{\bf x}_2) \geq 0$. Since $({\bf v}',{\bf u}_1,{\bf u}_2)^{\top}\tilde{A}({\bf v}',{\bf u}_1,{\bf u}_2)=\mathrm{diag}(\lambda,\nu_1,\nu_2)$, by using \eqref{eq: matrix expression of xj}, we obtain that 
\begin{equation}
\begin{aligned}
(D^{\frac{1}{2}}{\bf x}_1)^{\top}\tilde{A}(D^{\frac{1}{2}}{\bf x}_2) 
=\lambda r_1r_2+\nu_1 \alpha_1\alpha_2 + \nu_2\beta_1\beta_2.
\end{aligned}
\end{equation}
Now, our desired inequality is reduced to
\begin{equation}\label{eq: desired inequality for the convexity of Q}
\lambda r_1r_2  \geq (-\nu_1)\alpha_1\alpha_2+(-\nu_2)\beta_1\beta_2. 
\end{equation}
This inequality can be shown as follows. Note that $(D^{\frac{1}{2}}{\bf x}_j)^{\top}\tilde{A} (D^{\frac{1}{2}}{\bf x}_{j}) \geq 0$. Then, we have $\lambda r_j^2 \geq (-\nu_1)\alpha_j^2+(-\nu_2)\beta_j^2$. Since both sides are non-negative and $r_j\geq 0$, we have $\sqrt{\lambda}r_j \geq \sqrt{(-\nu_1)\alpha_j^2+(-\nu_2)\beta_j^2}$.  In particular, we obtain
\begin{equation}\label{eq: lemma 1 for the convexity of Q}
\lambda r_1r_2=(\sqrt{\lambda}r_1)(\sqrt{\lambda}r_2) \geq \sqrt{(-\nu_1)\alpha_1^2+(-\nu_2)\beta_1^2}\sqrt{(-\nu_1)\alpha_2^2+(-\nu_2)\beta_2^2}.
\end{equation}
Consider the vectors ${\bf a}_{j}=(\sqrt{-\nu_{1}}\alpha_{j},\sqrt{-\nu_2}\beta_{j})^{\top}$ for $j=1,2$. Since $\nu_1,\nu_2 \leq 0$, these vectors lie in $\mathbb{R}^2$. Thus, by the Cauchy-Schwarz inequality $|{\bf a}_1||{\bf a}_2| \geq \langle {\bf a}_1,{\bf a}_2\rangle$ with respect to the Euclidean inner product, we have
\begin{equation}\label{eq: lemma 2 for the convexity of Q}
\sqrt{(-\nu_1)\alpha_1^2+(-\nu_2)\beta_1^2}\sqrt{(-\nu_1)\alpha_2^2+(-\nu_2)\beta_2^2} \geq (-\nu_1)\alpha_1\alpha_2+(-\nu_2)\beta_1\beta_2.
\end{equation}
By combining these two inequalities (\ref{eq: lemma 1 for the convexity of Q}) and (\ref{eq: lemma 2 for the convexity of Q}), we obtain (\ref{eq: desired inequality for the convexity of Q}) as desired, which implies that $a_1{\bf x}_1+a_2{\bf x}_2 \in Q^{+}$.
\end{proof}
By \Cref{lem: eigenvalues lemma for cluster-cyclic matrix}, $H$ is decomposed into the connected components $H^{+}$ and $H^{-}$ such that
\begin{equation}
H^{+} = H \cap \Epositiveregion{D^{\frac{1}{2}}\mathbf{v}},
\quad
H^{-} = H \cap \Enegativeregion{D^{\frac{1}{2}}\mathbf{v}}.
\end{equation}
By \Cref{lem: convexity of Q}, $Q$ is the minimum convex cone containing $H$ and the origin $\mathbf{0}$. Hence, the above situation also occurs in $Q \setminus\{\mathbf{0}\}$. More precisely, $Q^{+}\setminus\{\mathbf{0}\}$ and $Q^{-}\setminus\{\mathbf{0}\}$ are the connected component of $Q\setminus\{\mathbf{0}\}$ such that
\begin{equation}\label{eq: separateness for Q}
Q^{+}\setminus\{\mathbf{0}\} = (Q\setminus\{\mathbf{0}\}) \cap \Epositiveregion{D^{\frac{1}{2}}\mathbf{v}},
\quad
Q^{-}\setminus\{\mathbf{0}\} = (Q\setminus\{\mathbf{0}\}) \cap \Enegativeregion{D^{\frac{1}{2}}\mathbf{v}}.
\end{equation}
For later use, we establish the following preliminary lemma.
\begin{lemma}\label{lem: separation lemma for Q}
For any ${\bf x} \in H^{+}$ and ${\bf y} \in H^{-}$, we have $ {\bf x} + {\bf y}  \notin Q \setminus \{{\bf 0}\}$.
\end{lemma}
\begin{proof}
It suffices to show that 
\begin{equation}
\{D^{\frac{1}{2}}({\bf x}+{\bf y})\}^{\top}\tilde{A}\{D^{\frac{1}{2}}({\bf x}+{\bf y})\} \leq 0.
\end{equation}
Let $[{\bf v}',{\bf u}_1,{\bf u}_2]$ be an orthogonal eigenbasis corresponding to the eigenvalues $(\lambda,\nu_1,\nu_2)$ of $\tilde{A}$. Then, we may express
\begin{equation}
{\bf x}=\alpha D^{-\frac{1}{2}}{\bf v}'+a_1D^{-\frac{1}{2}}{\bf u}_1+a_2D^{-\frac{1}{2}}{\bf u}_2,
\quad
{\bf y}=\beta D^{-\frac{1}{2}}{\bf v}'+b_1 D^{-\frac{1}{2}}{\bf u}_1+b_2 D^{-\frac{1}{2}}{\bf u}_2
\end{equation}
for some $\alpha,\beta,a_i,b_i \in \mathbb{R}$. Since ${\bf x} \in H^{+}$ and ${\bf y} \in H^{-}$, we have $\beta < 0 <\alpha$ and $(D^{\frac{1}{2}}{\bf x})^{\top}\tilde{A}(D^{\frac{1}{2}}{\bf x})=(D^{\frac{1}{2}}{\bf y})^{\top}\tilde{A}(D^{\frac{1}{2}}{\bf y})=2$. Thus, we have
\begin{equation}
\{D^{\frac{1}{2}}({\bf x}+{\bf y})\}^{\top}\tilde{A}\{D^{\frac{1}{2}}({\bf x}+{\bf y})\}=4+2(D^{\frac{1}{2}}{\bf x})^{\top}\tilde{A}(D^{\frac{1}{2}}{\bf y})=4+2(\lambda \alpha\beta + \nu_1a_1b_1+\nu_2a_2b_2),
\end{equation}
where the last equality follows from the fact that $[{\bf v}', {\bf u}_1, {\bf u}_2]$ is an orthogonal eigenbasis of $\tilde{A}$. Thus, our desired inequality is reduced to
\begin{equation}\label{eq: desired inequality for separation}
-\lambda \alpha \beta \geq 2 + \nu_1a_1b_1 + \nu_2a_2b_2.
\end{equation}
Since $(D^{\frac{1}{2}}{\bf x})^{\top}\tilde{A}(D^{\frac{1}{2}}{\bf x})=2$, we have
\begin{equation}
\lambda \alpha^2=2-\nu_1a_1^2-\nu_2a_2^2.
\end{equation}
Since $\alpha \geq 0$ and $\lambda > 0$, this implies $\sqrt{\lambda}\alpha=\sqrt{2-\nu_1a_1^2-\nu_2a_2^2}$.
Using a similar argument, we may obtain $-\sqrt{\lambda}\beta=\sqrt{2-\nu_1b_1^2-\nu_2b_2^2}$. (Note that $\beta \leq 0$.) Thus, the inequality (\ref{eq: desired inequality for separation}) is equivalent to
\begin{equation}\label{eq: desired inequality for separation 2}
\sqrt{2-\nu_1a_1^2-\nu_2a_2^2}\sqrt{2-\nu_1b_1^2-\nu_2b_2^2} \geq 2+\nu_1a_1b_1+\nu_2a_2b_2.
\end{equation}
In fact, this inequality can be shown as follows. Let ${\bf a}=(\sqrt{2},\sqrt{-\nu_1}a_1,\sqrt{-\nu_2}a_2)^{\top}$ and ${\bf b}=(\sqrt{2},-\sqrt{-\nu_1}b_1,-\sqrt{-\nu_2}b_2)^{\top}$. Since $\nu_1,\nu_2 \leq 0$, both vectors lie in $\mathbb{R}^3$. By the Cauchy-Schwarz inequality $|{\bf a}||{\bf b}| \geq \langle {\bf a},{\bf b} \rangle$ with respect to the Euclidean inner product, we may obtain the desired inequality (\ref{eq: desired inequality for separation 2}).
\end{proof}

\subsection{Main theorem}
In the previous subsection, we have investigated several geometric properties of the quadratic form. In this subsection, we apply these properties to $g$-vectors.
\begin{definition}
Let $B \in \mathrm{M}_3(\mathbb{R})$ be a cluster-cyclic initial exchange matrix. For each initial mutation direction $i=1,2,3$, define $Q_i$, $Q^{+}_i$, $Q^{-}_i$, $H_i$, $H_i^{+}$, and $H_i^{-}$ as the corresponding sets  in (\ref{eq: set Q}) with respect to $\tilde{A}=\tilde{A}^{[i]}$ and a skew-symmetrizer $D$ of $B^{[i]}$. We call $Q^{+}_{i}$ the {\em global upper bound} for the direction $i$.
\end{definition}
Note that these sets $Q_i$, $Q^{+}_i$, and $Q^{-}_i$ are independent of the choice of $D$. Thanks to Lemma~\ref{lem: surface for g-vectors}, all $g$-vectors belong to $Q_i$. The main result of this section shows that the upper bound can be refined as follows.
\begin{theorem}[Global upper bound]\label{thm: global upper bound theorem}
Let $B \in \mathrm{M}_3(\mathbb{R})$ be a cluster-cyclic initial exchange matrix. Fix any initial mutation direction $i=1,2,3$. Then, it holds that 
\begin{equation}
|\Delta^{\geq [i]}(B)| \subset Q^{+}_i.
\end{equation} In particular, we have $|\Delta(B)| \subset Q^{+}_1 \cup Q^{+}_2 \cup Q^{+}_3.$
\end{theorem}
\begin{remark}
More strongly, by \Cref{lem: surface for modified g-vectors}, all the modified $g$-vectors $\tilde{\mathbf{g}}_l^{\mathbf{w}}$ ($\mathbf{w} \in \mathcal{T}^{\geq [i]}$, $l=1,2,3$) are on the same plane $H_i^{+}$.
\end{remark}
\begin{proof}
Let ${\bf w} \geq [i]$. Since $Q^{+}_i$ is a convex cone by Lemma~\ref{lem: convexity of Q}, it suffices to show that every modified $g$-vector $\tilde{\bf g}_{l}^{\bf w}$ belongs to $H^{+}_i \subset Q^{+}_i$. By Lemma~\ref{lem: surface for modified g-vectors}, we have already known that $\tilde{\bf g}_{l}^{\bf w}\in H_i$.
\par
Suppose that the claim does not hold. Then, there exists a modified $g$-vector $\tilde{\bf g}_l^{{\bf w}}$ with ${\bf w} \geq [i]$ and $l\in \{1,2,3\}$, such that it does not belong to $H^{+}_{i}$. Without loss of generality, we may assume that such ${\bf w}$ is a minimal one. That is to say, every proper subsequence  ${\bf u} \leq {\bf w}$ satisfies $\tilde{\bf g}_{j}^{\bf u} \in H^{+}_i$. Note that $\tilde{\bf g}_{l}^{[i]}\in H^{+}_i$ ($i,l=1,2,3$). It implies that  $|{\bf w}| \geq 2$. Let ${\bf u}$ be the mutation subsequence that is one step shorter than ${\bf w}$. Hence, there exists $M\in \{S,T\}$, such that ${\bf w}={\bf u}M$. 
\par
The key point is the following three conditions.
\begin{equation}\label{eq: key point for the global upper bound}
\tilde{\bf g}_{M}^{\bf u} \in H^{+}_{i},
\quad
\tilde{\bf g}_{K}^{\bf w} \in H^{-}_{i}, 
\quad
\tilde{\bf g}_{K}^{\bf w}+\tilde{\bf g}_{M}^{\bf u} \in Q_i\setminus\{\mathbf{0}\}.
\end{equation}
For simplicity, we might assume that $M=S$ and we can do a similar argument for $M=T$.
The first condition is direct since ${\bf u}$ is a proper subsequence of ${\bf w}$.
By Lemma~\ref{lem: S-mutations of modified c- g-vectors}, $\tilde{\bf g}_{S}^{\bf w}$ and $\tilde{\bf g}_{T}^{\bf w}$ have already appeared in $\{\tilde{\bf g}_{K}^{\bf u},\tilde{\bf g}_{S}^{\bf u},\tilde{\bf g}_{T}^{\bf u}\}$. Since we assumed that \Cref{thm: global upper bound theorem} breaks at this ${\bf w}$, $\tilde{\bf g}_{K}^{\bf w}$ should do so. Namely, the first two conditions of \eqref{eq: key point for the global upper bound} hold. By \Cref{lem: S-mutations of modified c- g-vectors}, we have $\tilde{\mathbf{g}}_{K}^{\mathbf{w}}+\tilde{\mathbf{g}}_{S}^{\mathbf{u}}=p_{SK}^{\bf u}\tilde{\mathbf{g}}_{K}^{\bf u}$. Since $\tilde{\mathbf{g}}_{K}^{\mathbf{u}} \in Q_i\setminus\{\mathbf{0}\}$, we have the third condition of \eqref{eq: key point for the global upper bound}. 

However, due to Lemma~\ref{lem: separation lemma for Q}, the first and the second conditions should imply that $\tilde{\bf g}_{K}^{\bf w}+\tilde{\bf g}_{M}^{\bf u} \notin Q_i \setminus\{{\bf 0}\}$, which is a contradiction. Hence, all the modified $g$-vectors lie on $H^{+}_i$ and $|\Delta^{\geq [i]}(B)| \subset Q^{+}_i.$ Moreover, it is direct that $\tilde{\bf g}_j^{\emptyset}=\tilde{\mathbf{e}}_j\in Q^{+}_i$ for any $i,j=1,2,3$, which implies that $|\Delta(B)| \subset Q^{+}_1 \cup Q^{+}_2 \cup Q^{+}_3.$
\end{proof}
\begin{example}
We can respectively refer to the thick blue lines in \Cref{fig: pictures of G-fans} as the global upper bounds for the case of $\nu_1=\nu_2=0$ and for the general cluster-cyclic type ($\nu_1,\nu_2<0$). In general, the three global upper bounds $Q_1^{+}$, $Q_2^{+}$, and $Q_3^{+}$ are different. However, when an initial exchange matrix $B$ is given by \eqref{eq: value-rigid type}, it holds that
\begin{equation}
Q_1^{+}=Q_2^{+}=Q_3^{+}=\{x_1\tilde{\mathbf{e}}_1+x_2\tilde{\mathbf{e}}_2+x_3\tilde{\mathbf{e}}_3 \in \mathbb{R}^3 \mid x_1+x_2+x_3 > 0\} \cup \{{\bf 0}\}.
\end{equation}
\end{example}

\section{$S$-mutations}\label{sec: S-mutations}
Following \Cref{def: K S T labeling}, for any $M=K,S,T$, we write the modified $c$-vectors as $\tilde{\bf c}_{M}^{\bf w}=\tilde{\bf c}_{M({\bf w})}^{\bf w}$ and the modified $g$-vectors as $\tilde{\bf g}_{M}^{\bf w}=\tilde{\bf g}_{M({\bf w})}^{\bf w}$, and similarly for the remaining data. In Subsection~\ref{sec: mutation formulas}, we have seen that the mutations can be distinguished by $S$ and $T$. In this section, we focus on the structure of repeating $S$-mutations. Note that the behavior of $\mathbf{w}S^n$ is detailed in \eqref{eq: index S-mutations}.
\subsection{Formulas for $S$-mutations}
Firstly, we give an explicit $S$-mutation formula.
In fact, it is closely related to the {\em Chebyshev polynomials} $U_{n}(p)$ of the second kind, which is defined by the following recursion:
\begin{equation}
\begin{aligned}
U_{-2}(p)&=-1,\ U_{-1}(p)=0,\\
U_{n+1}(p)&=2pU_{n}(p)-U_{n-1}(p)\quad (n \in \mathbb{Z}_{\geq -1}).
\end{aligned}
\end{equation}
Set $u_n(p)=U_n\left(\frac{p}{2}\right)$. Then, it obeys the recursion
\begin{equation}\label{eq: recursion for u}
u_{n+1}(p)=pu_{n}(p)-u_{n-1}(p).
\end{equation} 
For a given $p \geq 2$, an important fact is that $\{u_n(p)\}$ is monotonically increasing for $n$. We also obtain the following formula for $u_{n}(p)$.
\begin{equation}\label{eq: explicit values of u_n}
u_{n}(p)=\begin{cases}
n+1 & p=2,\\ 
\frac{1}{\sqrt{p^2-4}}(\alpha_p^{n+1}-\alpha_p^{-n+1}) & p>2,
\end{cases}
\end{equation}
where $\alpha_p=\frac{p+\sqrt{p^2-4}}{2}$ and $\alpha_p^{-1}=\frac{p-\sqrt{p^2-4}}{2}$. In particular, we have
\begin{equation}\label{eq: limit of Chebychev polynomials}
\lim_{n \to \infty}\frac{\sqrt{p^2-4}}{\alpha_{p}^{n}}u_n(p)=\alpha_p \ (p>2),
\qquad
\lim_{n \to \infty} \frac{1}{n}u_n(2)=\alpha_2=1.
\end{equation}
By using $u_n(p)$, we can express the $S$-mutation formulas as follows.
\begin{lemma}\label{lem: infinite S mutations}
Let ${\bf w} \in \mathcal{T}\backslash\{\emptyset\}$ and $n \in \mathbb{Z}_{\geq 0}$. Then, we have the following formulas:
\\
\begin{enumerate} 
\item Formulas for $p_{MM'}^{\mathbf{w}S^n}$:
\begin{equation}\label{eq: infinite S mutations for p}
\begin{aligned}
p_{ST}^{{\bf w}S^n}&=p_{TS}^{{\bf w}S^n}=-u_{n-1}(p^{\bf w}_{SK})p_{KT}^{\bf w}+u_{n}(p^{\bf w}_{SK})p_{ST}^{\bf w},\\
p_{KT}^{{\bf w}S^n}&=p_{TK}^{{\bf w}S^n}=-u_{n-2}(p^{\bf w}_{SK})p_{KT}^{\bf w}+u_{n-1}(p^{\bf w}_{SK})p_{ST}^{\bf w},\\
p_{SK}^{{\bf w}S^n}&=p_{KS}^{{\bf w}S^n}=p^{\bf w}_{SK}.\\
\end{aligned}
\end{equation}
\item Formulas for modified $g$-vectors $\tilde{\mathbf{g}}_{M}^{\mathbf{w}S^n}$:
\begin{equation}\label{eq: infinite S mutations for g vectors}
\begin{aligned}
\tilde{{\bf g}}_{K}^{{\bf w}S^n}&=-u_{n-1}(p^{\bf w}_{SK})\tilde{\bf g}_{S}^{\bf w}+u_{n}(p^{\bf w}_{SK})\tilde{\bf g}_{K}^{\bf w},\\
\tilde{{\bf g}}_{S}^{{\bf w}S^n}&=-u_{n-2}(p^{\bf w}_{SK})\tilde{\bf g}_{S}^{\bf w}+u_{n-1}(p^{\bf w}_{SK})\tilde{\bf g}_{K}^{\bf w},\\
\tilde{{\bf g}}_{T}^{{\bf w}S^n}&=\tilde{\bf g}_{T}^{{\bf w}}.\\
\end{aligned}
\end{equation}
\item Formulas for modified $c$-vectors $\tilde{\mathbf{c}}_{M}^{\mathbf{w}S^n}$:
\begin{equation}\label{eq: infinite S mutations for c vectors}
\begin{aligned}
\tilde{\bf c}_{K}^{{\bf w}S^n}&=-u_{n-2}(p^{\bf w}_{SK})\tilde{\bf c}_{K}^{\bf w}-u_{n-1}(p^{\bf w}_{SK})\tilde{\bf c}_{S}^{\bf w},\\
\tilde{\bf c}_{S}^{{\bf w}S^n}&=u_{n-1}(p^{\bf w}_{SK})\tilde{\bf c}_{K}^{\bf w}+u_{n}(p^{\bf w}_{SK})\tilde{\bf c}_{S}^{\bf w},\\
\tilde{\bf c}_{T}^{{\bf w}S^n}&=\tilde{\bf c}_{T}^{\bf w}.
\end{aligned}
\end{equation}
\end{enumerate}
\end{lemma}

\begin{proof}
They can be established by the induction on $n$, where the inductive step follows from the formulas in \Cref{lem: S-mutations of modified c- g-vectors}. More precisely, we exhibit the proof for $\tilde{{\bf g}}_{K}^{{\bf w}S^{n+1}}$ and others are similar. Note that by \eqref{eq: recursion for u} and the induction conditions for $n$, we have 
\begin{equation}\label{eq: proof on behalf}
\begin{aligned}
\tilde{{\bf g}}_{K}^{{\bf w}S^{n+1}} &= -\tilde{{\bf g}}_{S}^{{\bf w}S^{n}}+  p^{{\bf w}S^n}_{SK}\tilde{{\bf g}}_{K}^{{\bf w}S^{n}} \\ &= u_{n-2}(p^{\bf w}_{SK})\tilde{\bf g}_{S}^{\bf w}-u_{n-1}(p^{\bf w}_{SK})\tilde{\bf g}_{K}^{\bf w}+ p^{{\bf w}}_{SK} [-u_{n-1}(p^{\bf w}_{SK})\tilde{\bf g}_{S}^{\bf w}+u_{n}(p^{\bf w}_{SK})\tilde{\bf g}_{K}^{\bf w}] \\ &= -u_{n}(p^{\bf w}_{SK})\tilde{\bf g}_{S}^{\bf w}+u_{n+1}(p^{\bf w}_{SK})\tilde{\bf g}_{K}^{\bf w}.
\end{aligned}
\end{equation}
\end{proof}

\subsection{Asymptotic structure} In this subsection, we exhibit the asymptotic phenomenon about $S$-mutations.
Now, we consider the inner product $\langle {\bf a}, {\bf b} \rangle_{D}={\bf a}^{\top}D{\bf b}$ on $\mathbb{R}^3$ and write its norm by $\|{\bf a}\|_{D}=\sqrt{\langle {\bf a},{\bf a}\rangle_{D}}$.
We introduce the following equivalence relation $\sim$ on $\mathbb{R}^{3} \setminus\{{\bf 0}\}$:
\begin{equation}
{\bf a} \sim {\bf b} \Longleftrightarrow {\bf b}=\lambda {\bf a}\ \textup{for some $\lambda > 0$}.
\end{equation}
Then, the sphere $S^2_D=\{{\bf a} \in \mathbb{R}^3 \mid \|{\bf a}\|_{D}=1\} \subset \mathbb{R}^3$ is a representative of this equivalence class. By identifying $(\mathbb{R}^3\setminus\{{\bf 0}\})/{\sim}$ with $S^2_D$, we introduce the topology. Namely, for any sequence of nonzero vectors ${\bf v}_n \in \mathbb{R}^3 \setminus \{{\bf 0}\}$, if $\lim_{n \to \infty}\, \frac{{\bf v}_n}{\|{\bf v}_n\|_{D}}$ exists, we write
\begin{equation}
\tildelim_{n \to \infty} {\bf v}_{n}=\lim_{n \to \infty} \frac{{\bf v}_n}{\|{\bf v}_n\|_{D}}.
\end{equation}
Since the sphere $S^2_D$ is a closed set, $\left\lVert\tildelim\limits_{n \to \infty} {\bf v}_n\right\rVert_{D}=1$ holds if it exists.
\par
We abuse this symbol $\sim$ to express the following equivalence relation on $\mathbb{R}$:
\begin{equation}
a \sim b \Longleftrightarrow \mathrm{sign}(a)=\mathrm{sign}(b).
\end{equation}
For any $\mathbf{a},\mathbf{b} \in \mathbb{R}^3 \setminus \{\mathbf{0}\}$ and $\mathbf{x} \in \mathbb{R}^3$, the relation $\mathbf{a} \sim \mathbf{b}$ on $\mathbb{R}^3 \setminus \{\mathbf{0}\}$ implies $\langle \mathbf{a}, \mathbf{x} \rangle_D \sim \langle \mathbf{b}, \mathbf{x}\rangle_D$ on $\mathbb{R}$.
\par
Firstly, we calculate $\tildelim\limits_{n \to \infty} \tilde{\bf c}^{{\bf w}S^n}_{S}$ and $\tildelim\limits_{n \to \infty} \tilde{\bf g}^{{\bf w}S^n}_{K}$, which play essential roles in the later argument. We have already known that $\tilde{\bf c}^{{\bf w}S^n}_{S}$ and $\tilde{\bf g}^{{\bf w}S^n}_{K}$ can be expressed as in \eqref{eq: infinite S mutations for g vectors} and \eqref{eq: infinite S mutations for c vectors}. To calculate their limits, we show the following lemma.
\begin{lemma}\label{lem: limit of Chebyshev difference}
Let ${\bf a},{\bf b} \in \mathbb{R}^3$ be linearly independent vectors, and let $p \geq 2$. Set ${\bf v}_{n}=u_{n-1}(p){\bf a}+u_{n}(p){\bf b}$. Then, we have
\begin{equation}
\tildelim_{n \to \infty}{{\bf v}_{n}} \sim {{\bf a}+\alpha_{p}{\bf b}},
\end{equation}
where $\alpha_{p}=\frac{p+\sqrt{p^2-4}}{2}$.
\end{lemma}
\begin{proof}
Suppose $p>2$. Then, by \eqref{eq: limit of Chebychev polynomials}, we have
\begin{equation}
\lim_{n \to \infty}\frac{\sqrt{p^2-4}}{\alpha_p^{n}}{\bf v}_n=\mathbf{a}+\alpha_p\mathbf{b}.
\end{equation}
Since $\mathbf{a}$ and $\mathbf{b}$ are linearly independent, the right hand side is a nonzero vector. Thus, the claim holds.
If $p=2$, then it holds that $u_n(p)=n+1$ by \eqref{eq: explicit values of u_n}.  Thus, we obtain $\mathbf{v}_n=n(\mathbf{a}+\mathbf{b})+\mathbf{b}$. Note that $\alpha_2=1$. Hence, we have
\begin{equation}
\lim_{n \to \infty}\frac{{\bf v}_n}{n}={\bf a}+{\bf b}.
\end{equation}
\end{proof}
As in \Cref{def: K S T labeling}, we define $p_{MM'}^{\bf w}=p_{M(\mathbf{w}),M'(\mathbf{w})}^{\mathbf{w}}$ and $\alpha_{MM'}^{\bf w}=\alpha_{M(\mathbf{w}),M'(\mathbf{w})}^{\mathbf{w}}$.
Based on these lemmas, we may obtain the limit of $S$-mutations.
\begin{lemma}\label{lem: limit of S mutations}
For any ${\bf w} \in \mathcal{T}$, we have
\begin{equation}\label{eq: limit of S mutations}
\begin{aligned}
\tildelim_{n \to \infty} \tilde{\bf g}_{K}^{{\bf w}S^n}&\sim \alpha_{SK}^{\bf w}\tilde{\bf g}_{K}^{\bf w}-\tilde{\bf g}_{S}^{\bf w},& \tildelim_{n \to \infty} \tilde{\bf c}^{{\bf w}S^n}_{S}&\sim\alpha_{SK}^{\bf w}\tilde{\bf c}^{\bf w}_{S}+\tilde{\bf c}_{K}^{\bf w}.
\end{aligned}
\end{equation}
\end{lemma}
\begin{proof}
This is immediately shown by Lemma~\ref{lem: infinite S mutations} and Lemma~\ref{lem: limit of Chebyshev difference}.
\end{proof}
\begin{example}
Fix ${\bf w} \in \mathcal{T}\backslash\{\emptyset\}$. Then, all cones $\mathcal{C}(G^{{\bf w}S^n})$ obtained by $S$-mutations can be illustrated as in Figure~\ref{fig: S mutation} in the stereographic projection.
\begin{figure}[htbp]
\centering
\begin{tikzpicture}
\fill ({5*cos(120)},{5*sin(120)}) node [above left] {$\tilde{\bf g}_{S}^{\bf w}$} circle (0.05);
\draw ({5*cos(120)},{5*sin(120)}) arc (120:45:5);
\fill ({3*cos(120)},{3*sin(120)}) node [below left] {$\tilde{\bf g}_{T}^{\bf w}=\tilde{\bf g}_{T}^{{\bf w}S^n}$} circle (0.05);
\draw ({5*cos(45)},{5*sin(45)}) node [right] {$\zeroregion{\tilde{\bf c}_{T}^{\bf w}}$};

\foreach \n in {1,1.5,2,2.5,3,4,5,...,10,12,14,...,50}
    {
    \fill ({5*cos(60+60/\n)},{5*sin(60+60/\n)}) circle (0.05);
    \draw ({3*cos(120)},{3*sin(120)}) to [out=-30,in=-150,relative] ({5*cos(60+60/\n)},{5*sin(60+60/\n)});
    };
\fill[blue] ({5*cos(60)},{5*sin(60)}) node [above right] {$\displaystyle{\tildelim_{n \to \infty} \tilde{\bf g}_{K}^{{\bf w}S^n}\sim\alpha_{SK}^{\bf w}\tilde{\bf g}_{K}^{\bf w}-\tilde{\bf g}_{S}^{\bf w}}$} circle (0.05);
\draw[blue, dashed, very thick] ({3*cos(120)},{3*sin(120)}) to [out=-30,in=-150,relative] ({5*cos(60)},{5*sin(60)});
\draw ({5*cos(60+60/1.5)},{5*sin(60+60/1.5)}) node [above left] {$\tilde{\bf g}_{K}^{\bf w}$};
\draw  ({5*cos(60+60/2)},{5*sin(60+60/2)}) node [above] {$\tilde{\bf g}_{K}^{{\bf w}S}$};
\draw  ({5*cos(60+60/2.5)},{5*sin(60+60/2.5)}) node [above right] {$\tilde{\bf g}_{K}^{{\bf w}S^2}$};
\draw ({(5*cos(120)+5*cos(100)+3*cos(120))/3+0.25},{(5*sin(120)+5*sin(100)+3*sin(120))/3}) node {${\bf w}$};
\end{tikzpicture}
\caption{$G$-cones obtained by $S$-mutations.}\label{fig: S mutation}
\end{figure}
\end{example}

\section{Support of trunks}\label{sec: trunks}
Fix an initial mutation direction $i\in \{1,2,3\}$.
Recall from \Cref{def: K S T labeling} that we defined the two kinds of subsets of the subtree $\mathcal{T}^{\geq [i]}$: the trunk $\mathcal{T}^{<[i]S^{\infty}}$ and the branches $\mathcal{T}^{\geq [i]X}$, where $X$ contains at least one letter $T$. Correspondingly, for the $G$-fan, we define the {\em trunk} of $\Delta^{\geq [i]}(B)$ by
\begin{equation}
\Delta^{<[i]S^{\infty}}(B)=\{\mathcal{C}_{J}(G^{\mathbf{w}}) \mid \mathbf{w} \in \mathcal{T}^{<[i]S^{\infty}},\ J \subset \{1,2,3\}\},
\end{equation}
and a {\em branch} by
\begin{equation}
\Delta^{\geq [i]X}(B)=\{\mathcal{C}_J(G^{\mathbf{w}}) \mid \mathbf{w} \in \mathcal{T}^{\geq [i]X},\ J \subset \{1,2,3\}\},
\end{equation}
where $X \in \mathcal{M}$ contains at least one letter $T$.
In this section, we give an explicit description of the trunks $\Delta^{<[i]S^{\infty}}(B)$.
\begin{lemma}\label{lem: infinite S mutation in trunks}
Fix an initial mutation direction $i=1,2,3$. Set $k_0,s_0,t_0 \in \{1,2,3\}$ as in \Cref{tab: List of initial indices}. For each $n \in \mathbb{Z}_{\geq 0}$, we have the following formulas.
\begin{enumerate}
\item Formulas for $p_{MM'}^{[i]S^n}$:
\begin{equation}\label{eq: p in trunk}
\begin{aligned}
p_{ST}^{[i]S^n}&=p_{TS}^{[i]S^n}=-u_{n}(p_{k_0s_0})p_{s_0t_0}+u_{n+1}(p_{k_0s_0})p_{k_0t_0},\\
p_{KT}^{[i]S^n}&=p_{TK}^{[i]S^n}=-u_{n-1}(p_{k_0s_0})p_{s_0t_0}+u_{n}(p_{k_0s_0})p_{k_0t_0},\\
p_{SK}^{[i]S^n}&=p_{KS}^{[i]S^n}=p_{k_0s_0}.
\end{aligned}
\end{equation}
\item Formulas for modified $g$-vectors $\tilde{\mathbf{g}}_{M}^{[i]S^n}$:
\begin{equation}\label{eq: g vectors in trunk}
\begin{aligned}
\tilde{{\bf g}}_{K}^{[i]S^n}&=-u_{n}(p_{k_0s_0})\tilde{\bf e}_{k_0}+u_{n+1}(p_{k_0s_0})\tilde{\bf e}_{s_0},\\
\tilde{{\bf g}}_{S}^{[i]S^n}&=-u_{n-1}(p_{k_0s_0})\tilde{\bf e}_{k_0}+u_{n}(p_{k_0s_0})\tilde{\bf e}_{s_0},\\
\tilde{{\bf g}}_{T}^{[i]S^n}&=\tilde{\bf e}_{t_0}.
\end{aligned}
\end{equation}
\item Formulas for modified $c$-vectors $\tilde{\mathbf{c}}_{M}^{[i]S^n}$:
\begin{equation}\label{eq: c vectors in trunk}
\begin{aligned}
\tilde{\bf c}_{K}^{[i]S^n}&=-u_{n-1}(p_{k_0s_0})\tilde{\bf e}_{s_0}-u_{n}(p_{k_0s_0})\tilde{\bf e}_{k_0},\\
\tilde{\bf c}_{S}^{[i]S^n}&=u_{n}(p_{k_0s_0})\tilde{\bf e}_{s_0}+u_{n+1}(p_{k_0s_0})\tilde{\bf e}_{k_0},\\
\tilde{\bf c}_{T}^{[i]S^n}&=\tilde{\bf e}_{t_0}.
\end{aligned}
\end{equation}
\end{enumerate}
\end{lemma}
\begin{proof}
This can be directly proved according to \Cref{lem: kst} and \Cref{lem: infinite S mutations}.
\end{proof}

\begin{proposition}
Fix an initial mutation direction $i=1,2,3$, and set $k_0=K([i])$, $s_0=S([i])$, and $t_0=T([i])$. Then, we have
\begin{equation}
\tildelim_{n \to \infty}  \tilde{\bf g}_{K}^{[i]S^n} \sim \alpha_{s_0k_0}\tilde{\bf e}_{s_0}-\tilde{\bf e}_{k_0},
\quad
\tildelim_{n \to \infty} \tilde{\bf c}_S^{[i]S^n} \sim \alpha_{s_0k_0}\tilde{\bf e}_{k_0}+\tilde{\bf e}_{s_0}.
\end{equation}
Moreover, we have
\begin{equation}\label{eq: support of the trunk}
\begin{aligned}
|\Delta^{< [i]S^\infty}(B)| &= \mathcal{C}^{\circ}(\tilde{\bf e}_{t_0},\tilde{\bf e}_{s_0},\alpha_{k_0s_0}\tilde{\bf e}_{s_0}-\tilde{\bf e}_{k_0}) \cup \mathcal{C}(\tilde{\mathbf{e}}_{s_0},\tilde{\mathbf{e}}_{t_0}) \cup \mathcal{C}^{\circ}(\tilde{\mathbf{e}}_{s_0},\alpha_{k_0s_0}\tilde{\bf e}_{s_0}-\tilde{\bf e}_{k_0})
\\
&= \left(\negativeclosure{\tilde{\mathbf{e}}_{k_0}} \cap \positiveclosure{\tilde{\mathbf{e}}_{t_0}} \cap \positiveregion{\alpha_{s_0k_0}\tilde{\mathbf{e}}_{k_0}+\tilde{\mathbf{e}}_{s_0}}\right) \cup \mathcal{C}(\tilde{\mathbf{e}}_{t_0}).
\end{aligned}
\end{equation}
\end{proposition}
\begin{proof}
This can be shown by \Cref{lem: limit of S mutations}.
\end{proof}
\begin{example}
The $G$-cones in the trunks $\Delta^{<[i]S^{\infty}}(B)$ are illustrated as in Figure~\ref{fig: trunks}.
\begin{figure}[htbp]
\centering
\begin{tikzpicture}
\clip (-2.5,-2) rectangle (2.5,2.25);

\fill[fill=gray, fill opacity=0.5] (-0.8966,-0.5176) arc [radius=3.4641, start angle=-105.0, end angle=-75.0]--(0.8966,-0.5176) arc [radius=3.4641, start angle=15.0, end angle=-5.2644]--(1.0,-1.7321) arc [radius=2.4495, start angle=-95.2644, end angle=-150.0];

\fill[fill=gray, fill opacity=0.5] (-0.8966,-0.5176) arc [radius=3.4641, start angle=165.0, end angle=135.0]--(0.0,1.0353) arc [radius=2.4495, start angle=90.0, end angle=144.7356]--(-2.0,0.0) arc [radius=3.4641, start angle=-125.2644, end angle=-105.0];

\fill[fill=gray, fill opacity=0.5] (0.8966,-0.5176) arc [radius=3.4641, start angle=15.0, end angle=45.0]--(0.0,1.0353) arc [radius=3.4641, start angle=135.0, end angle=114.7356]--(1.0,1.7321) arc [radius=2.4495, start angle=24.7356, end angle=-30.0];

\draw (1.0086,-1.6176) arc [radius=3.4641, start angle=-3.3659, end angle=-3.1201];
\draw (1.0068,-1.6469) arc [radius=2.4518, start angle=-93.852, end angle=-147.5104];
\draw (-0.8966,-0.5176) arc [radius=2.4537, start angle=-146.6335, end angle=-93.3659];
\draw (1.0086,-1.6176) arc [radius=3.4641, start angle=-3.3659, end angle=-3.5613];
\draw (0.8966,-0.5176) arc [radius=3.4641, start angle=15.0, end angle=45.0];
\draw (1.013,-1.5206) arc [radius=3.4641, start angle=-1.7598, end angle=-2.3711];
\draw (1.0105,-1.5835) arc [radius=3.4641, start angle=-2.8012, end angle=-2.3711];
\draw (1.013,-1.5206) arc [radius=3.4641, start angle=-1.7598, end angle=-0.8227];
\draw (1.0079,-1.6294) arc [radius=2.4529, start angle=-93.5613, end angle=-146.9872];
\draw (-0.8966,-0.5176) arc [radius=2.4514, start angle=-147.7094, end angle=-93.963];
\draw (1.0073,-1.639) arc [radius=3.4641, start angle=-3.7202, end angle=-3.5613];
\draw (1.0116,-1.5575) arc [radius=2.4596, start angle=-92.3711, end angle=-144.8056];
\draw (-0.8966,-0.5176) arc [radius=2.4646, start angle=-143.6598, end angle=-91.7598];
\draw (1.0053,-1.1609) arc [radius=2.582, start angle=-85.8058, end angle=-131.5651];
\draw (0.0,1.0353) arc [radius=3.4641, start angle=135.0, end angle=165.0];
\draw (-0.8966,-0.5176) arc [radius=2.4567, start angle=-145.6013, end angle=-92.8012];
\draw (1.0105,-1.5835) arc [radius=3.4641, start angle=-2.8012, end angle=-3.1201];
\draw (1.0073,-1.639) arc [radius=3.4641, start angle=-3.7202, end angle=-3.852];
\draw (1.0095,-1.6028) arc [radius=2.4549, start angle=-93.1201, end angle=-146.1859];
\draw (1.0063,-1.6536) arc [radius=3.4641, start angle=-3.963, end angle=-3.852];
\draw (-0.8966,-0.5176) arc [radius=2.498, start angle=-138.6901, end angle=-89.209];
\draw (1.0143,-1.3664) arc [radius=3.4641, start angle=0.791, end angle=-0.8227];
\draw (-0.8966,-0.5176) arc [radius=3.4641, start angle=-105.0, end angle=-75.0];
\draw (-0.8966,-0.5176) arc [radius=2.4523, start angle=-147.2737, end angle=-93.7202];
\draw (1.0143,-1.464) arc [radius=2.4744, start angle=-90.8227, end angle=-141.8699];
\draw (1.0143,-1.3664) arc [radius=3.4641, start angle=0.791, end angle=4.1942];
\draw (0.8966,-0.5176) arc [radius=3.4641, start angle=15.0, end angle=4.1942];

\draw (0.8966,-0.5176) arc [radius=2.4646, start angle=-23.6598, end angle=28.2402];
\draw (0.9071,1.6876) arc [radius=2.4529, start angle=26.4387, end angle=-26.9872];
\draw (0.9157,1.6918) arc [radius=3.4641, start angle=116.2798, end angle=116.148];
\draw (0.7607,1.6103) arc [radius=2.4744, start angle=29.1773, end angle=-21.8699];
\draw (0.9229,1.6954) arc [radius=2.4518, start angle=26.148, end angle=-27.5104];
\draw (0.8966,-0.5176) arc [radius=3.4641, start angle=15.0, end angle=45.0];
\draw (0.8104,1.6376) arc [radius=3.4641, start angle=118.2402, end angle=117.6289];
\draw (0.8966,-0.5176) arc [radius=2.4523, start angle=-27.2737, end angle=26.2798];
\draw (0.8661,1.6668) arc [radius=3.4641, start angle=117.1988, end angle=116.8799];
\draw (0.8661,1.6668) arc [radius=3.4641, start angle=117.1988, end angle=117.6289];
\draw (0.8966,1.6823) arc [radius=3.4641, start angle=116.6341, end angle=116.8799];
\draw (0.0,1.0353) arc [radius=3.4641, start angle=135.0, end angle=124.1942];
\draw (0.8966,-0.5176) arc [radius=2.4567, start angle=-25.6013, end angle=27.1988];
\draw (0.6762,1.5616) arc [radius=3.4641, start angle=120.791, end angle=119.1773];
\draw (0.0,1.0353) arc [radius=3.4641, start angle=135.0, end angle=165.0];
\draw (0.8104,1.6376) arc [radius=3.4641, start angle=118.2402, end angle=119.1773];
\draw (0.6762,1.5616) arc [radius=3.4641, start angle=120.791, end angle=124.1942];
\draw (0.5027,1.4511) arc [radius=2.582, start angle=34.1942, end angle=-11.5651];
\draw (0.8966,1.6823) arc [radius=3.4641, start angle=116.6341, end angle=116.4387];
\draw (0.9157,1.6918) arc [radius=3.4641, start angle=116.2798, end angle=116.4387];
\draw (-0.8966,-0.5176) arc [radius=3.4641, start angle=-105.0, end angle=-75.0];
\draw (0.8966,-0.5176) arc [radius=2.4514, start angle=-27.7094, end angle=26.037];
\draw (0.843,1.6549) arc [radius=2.4596, start angle=27.6289, end angle=-24.8056];
\draw (0.9289,1.6983) arc [radius=3.4641, start angle=116.037, end angle=116.148];
\draw (0.8833,1.6756) arc [radius=2.4549, start angle=26.8799, end angle=-26.1859];
\draw (0.8966,-0.5176) arc [radius=2.498, start angle=-18.6901, end angle=30.791];
\draw (0.8966,-0.5176) arc [radius=2.4537, start angle=-26.6335, end angle=26.6341];

\draw (-1.8928,-0.0729) arc [radius=2.4549, start angle=146.8799, end angle=93.8141];
\draw (0.0,1.0353) arc [radius=2.4523, start angle=92.7263, end angle=146.2798];
\draw (0.8966,-0.5176) arc [radius=3.4641, start angle=15.0, end angle=45.0];
\draw (-1.8234,-0.117) arc [radius=3.4641, start angle=-121.7598, end angle=-120.8227];
\draw (-1.8234,-0.117) arc [radius=3.4641, start angle=-121.7598, end angle=-122.3711];
\draw (0.0,1.0353) arc [radius=2.4646, start angle=96.3402, end angle=148.2402];
\draw (-1.9052,-0.0647) arc [radius=3.4641, start angle=-123.3659, end angle=-123.1201];
\draw (-1.9052,-0.0647) arc [radius=3.4641, start angle=-123.3659, end angle=-123.5613];
\draw (-1.9231,-0.0529) arc [radius=3.4641, start angle=-123.7202, end angle=-123.852];
\draw (0.0,1.0353) arc [radius=2.4567, start angle=94.3987, end angle=147.1988];
\draw (0.0,1.0353) arc [radius=2.4514, start angle=92.2906, end angle=146.037];
\draw (-1.9352,-0.0447) arc [radius=3.4641, start angle=-123.963, end angle=-123.852];
\draw (-1.508,-0.2902) arc [radius=2.582, start angle=154.1942, end angle=108.4349];
\draw (-1.9297,-0.0484) arc [radius=2.4518, start angle=146.148, end angle=92.4896];
\draw (0.0,1.0353) arc [radius=3.4641, start angle=135.0, end angle=165.0];
\draw (0.0,1.0353) arc [radius=2.498, start angle=101.3099, end angle=150.791];
\draw (-1.9151,-0.0582) arc [radius=2.4529, start angle=146.4387, end angle=93.0128];
\draw (-1.7749,-0.1464) arc [radius=2.4744, start angle=149.1773, end angle=98.1301];
\draw (-1.9231,-0.0529) arc [radius=3.4641, start angle=-123.7202, end angle=-123.5613];
\draw (-1.8766,-0.0833) arc [radius=3.4641, start angle=-122.8012, end angle=-122.3711];
\draw (-0.8966,-0.5176) arc [radius=3.4641, start angle=-105.0, end angle=-75.0];
\draw (-0.8966,-0.5176) arc [radius=3.4641, start angle=-105.0, end angle=-115.8058];
\draw (-1.6905,-0.1952) arc [radius=3.4641, start angle=-119.209, end angle=-115.8058];
\draw (-1.8547,-0.0973) arc [radius=2.4596, start angle=147.6289, end angle=95.1944];
\draw (0.0,1.0353) arc [radius=2.4537, start angle=93.3665, end angle=146.6341];
\draw (-1.8766,-0.0833) arc [radius=3.4641, start angle=-122.8012, end angle=-123.1201];
\draw (-1.6905,-0.1952) arc [radius=3.4641, start angle=-119.209, end angle=-120.8227];

\draw (0,{2*sqrt(2)}) circle [radius= {2*sqrt(3)}];    
\draw ({-sqrt(6)},{-sqrt(2)}) circle [radius= {2*sqrt(3)}];    
\draw ({sqrt(6)},{-sqrt(2)}) circle [radius= {2*sqrt(3)}];    

\draw[thick, blue] (0.0,1.0353) arc [radius=3.4641, start angle=45.0, end angle=15.0];
\draw[thick, blue] (1.0,-1.7321) arc [radius=3.4641, start angle=-5.2644, end angle=15.0];
\draw[thick, blue] (0.8966,-0.5176) arc [radius=3.4641, start angle=-75.0, end angle=-105.0];
\draw[thick, blue, dashed] (1.0,-1.7321) arc [radius=2.4495, start angle=-95.2644, end angle=-150.0];
\draw[thick, blue] (0.0,1.0353) arc [radius=3.4641, start angle=135.0, end angle=165.0];

\draw[thick, blue] (-2.0,0.0) arc [radius=3.4641, start angle=-125.2644, end angle=-105.0];
\draw[thick, blue, dashed] (-2.0,0.0) arc [radius=2.4495, start angle=144.7356, end angle=90.0];

\draw[thick, blue] (1.0,1.7321) arc [radius=3.4641, start angle=114.7356, end angle=135.0];
\draw[thick, blue, dashed] (1.0,1.7321) arc [radius=2.4495, start angle=24.7356, end angle=-30.0];

\end{tikzpicture}
\caption{Trunks.}\label{fig: trunks}
\end{figure}
\end{example}

\section{Local upper bounds of branches}\label{sec: support of cluster-cyclic G-fan}
In this section, we introduce the local upper bound for each branch $\Delta^{\geq \mathbf{w}}(B)$ that possesses different properties from the global upper bound introduced in \Cref{sec: global}.
\par
Let $\Delta^{\geq \mathbf{w}}(B)$ be a branch. Namely, we assume that $\mathbf{w}=[i]X$, where $X \in \mathcal{M}$ contains at least one letter $T$. In this case, by \eqref{eq: T mutation of c- g-vectors in branches} and \eqref{eq: limit of S mutations}, we have
\begin{equation}
\tildelim_{n \to \infty}\tilde{\mathbf{g}}_{K}^{\mathbf{w}TS^n} \sim (p_{TK}^{\mathbf{w}}\alpha_{TK}^{\mathbf{w}}-1)\tilde{\mathbf{g}}_{K}^{\mathbf{w}}-\alpha_{TK}^{\mathbf{w}}\tilde{\mathbf{g}}_{T}^{\mathbf{w}},
\quad
\tildelim_{n \to \infty}\tilde{\mathbf{c}}_{S}^{\mathbf{w}TS^n} \sim (p_{TK}^{\mathbf{w}}\alpha_{TK}^{\mathbf{w}}-1)\tilde{\mathbf{c}}_{T}^{\mathbf{w}}+\alpha_{TK}^{\mathbf{w}}\tilde{\mathbf{c}}_{K}^{\mathbf{w}}.
\end{equation}
By \eqref{eq: minimum polynomial of alpha}, the equality $p_{TK}^{\mathbf{w}}\alpha_{TK}^{\mathbf{w}}-1=(\alpha_{TK}^{\mathbf{w}})^2$ holds. Thus, we have
\begin{equation}
\tildelim_{n \to \infty}\tilde{\mathbf{g}}_{K}^{\mathbf{w}TS^n} \sim \alpha_{TK}^{\mathbf{w}}\tilde{\mathbf{g}}_{K}^{\mathbf{w}}-\tilde{\mathbf{g}}_{T}^{\mathbf{w}},
\quad
\tildelim_{n \to \infty}\tilde{\mathbf{c}}_{S}^{\mathbf{w}TS^n} \sim \alpha_{TK}^{\mathbf{w}}\tilde{\mathbf{c}}_{T}^{\mathbf{w}}+\tilde{\mathbf{c}}_{K}^{\mathbf{w}}.
\end{equation}
Now, we focus on the two planes $\zeroregion{\alpha_{SK}^{\mathbf{w}}\tilde{\mathbf{c}}_{S}^{\mathbf{w}}+\tilde{\mathbf{c}}_K^{\mathbf{w}}}$ and $\zeroregion{\alpha_{TK}^{\mathbf{w}}\tilde{\mathbf{c}}_{T}^{\mathbf{w}}+\tilde{\mathbf{c}}_{K}^{\mathbf{w}}}$. The first one appears as the limit of $\tilde{\mathbf{c}}_{S}^{\mathbf{w}S^n}$, and the second one appears as the limit of $\tilde{\mathbf{c}}_{S}^{\mathbf{w}TS^n}$.
Since such two vectors $\alpha_{SK}^{\mathbf{w}}\tilde{\mathbf{c}}_{S}^{\mathbf{w}}+\tilde{\mathbf{c}}_K^{\mathbf{w}}$ and $\alpha_{TK}^{\mathbf{w}}\tilde{\mathbf{c}}_{T}^{\mathbf{w}}+\tilde{\mathbf{c}}_K^{\mathbf{w}}$ are linearly independent, their intersection is spanned by one nonzero vector, which will be denoted by $\bar{\mathbf{g}}^{\mathbf{w}}$ later.
This is illustrated in \Cref{fig: local upper bound of 3D picture}.
\begin{figure}[htbp]
\centering
\tdplotsetmaincoords{60}{-75}
\begin{tikzpicture}[tdplot_main_coords, scale=0.7]
\coordinate (limS) at (-0.5,-8,9.5); 
\coordinate (limT) at (8.5,-9,1.5); 
\fill[blue, opacity=0.2] (0,0,0)--(-1,0,2)--(limS)--(3,-3,-3)--cycle;
\fill[blue, opacity=0.2] (0,0,0)--(0,0,1)--(limT)--(3,-3,-3)--cycle;
\draw[->] (0,0,0)->(-1,0,2);
\draw[->] (0,0,0)->(0,-1,2);
\draw[->] (0,0,0)->(0,0,1);
\draw[->] (0,0,0)->(1,-2,2);
\fill[gray, opacity=0.4] (8,-9,2)--(-1,0,2)--(0,0,1)--(limT)--cycle;
\fill[gray, opacity=0.4] (0,-8,9)--(0,0,1)--(-1,0,2)--(limS)--cycle;
\foreach \n in{1,2,...,7}
    {
    \draw[->] (0,0,0)->(0,-\n-1,\n+2);
    };
\draw (0,-8.3,9.3) node {\rotatebox{45}{$\cdots$}};
\draw (0,-9,10) node {$\tilde{\bf g}_{K}^{{\bf w}S^n}$};
\foreach \n in{1,2,...,7}
    {
    \draw[->] (0,0,0)->(\n+1,-\n-2,2);
    };
\draw (8.3,-9.3,2) node {\rotatebox{25}{$\cdots$}};
\draw (9,-10,2) node {$\tilde{\bf g}_{K}^{{\bf w}TS^n}$};
\draw (-1/3,-1/3,5/3) node {${\bf w}$};
\draw[blue] (-1,0,2)--(limS) node [left=4pt] {\small $\zeroregion{\tildelim\limits_{n \to \infty} \tilde{\bf c}_{S}^{{\bf w}S^n}}$};
\draw[blue] (0,0,1)--(limT) node [below right] {\small $\zeroregion{\tildelim\limits_{n \to \infty} \tilde{\bf c}_{S}^{{\bf w}TS^n}}$};
\draw[->, red, very thick] (0,0,0)->(1,-1,-1) node [below] {$\bar{\bf g}^{\bf w}$};
\draw[red] (0,0,0)--(3,-3,-3);
\draw (0,0,0) node [below left] {$O$};
\end{tikzpicture}
\caption{$\bar{\mathbf{g}}^{\bf w}$ and $\mathscr{V}^{\bf w}$.}\label{fig: local upper bound of 3D picture}
\end{figure}
\par
We first give an explicit expression for such a vector $\bar{\mathbf{g}}^{\mathbf{w}}$.
\begin{lemma}\label{lem: observation of g for local upper bound}
Let $\Delta^{\geq {\bf w}}(B)$ be a branch. Let
\begin{equation}
\bar{\mathbf{g}}^{\mathbf{w}}=\tilde{\bf g}_{K}^{\bf w}-(\alpha_{SK}^{\bf w})^{-1}\tilde{\bf g}_{S}^{\bf w}-(\alpha_{TK}^{\bf w})^{-1}\tilde{\bf g}_{T}^{\bf w}.
\end{equation}
Then, we have
\begin{equation}
\zeroregion{\alpha_{SK}^{\bf w}\tilde{\bf c}_{S}^{\bf w}+\tilde{\bf c}_{K}^{\bf w}} \cap \zeroregion{\alpha_{TK}^{\bf w}\tilde{\bf c}_{T}^{\bf w}+\tilde{\bf c}_{K}^{\bf w}}=\langle \bar{\mathbf{g}}^{\mathbf{w}}  \rangle_{\mathrm{vec}}.
\end{equation}
\end{lemma}
\begin{proof}
Note that $\zeroregion{\alpha_{SK}^{\bf w}\tilde{\bf c}_{S}^{\bf w}+\tilde{\bf c}_{K}^{\bf w}}$ and $\zeroregion{\alpha_{TK}^{\bf w}\tilde{\bf c}_{T}^{\bf w}+\tilde{\bf c}_{K}^{\bf w}}$ are $2$-dimensional subspaces of $\mathbb{R}^3$. Since $\alpha_{SK}^{\bf w}\tilde{\bf c}_{S}^{\bf w}+\tilde{\bf c}_{K}^{\bf w}$ and $\alpha_{TK}^{\bf w}\tilde{\bf c}_{T}^{\bf w}+\tilde{\bf c}_{K}^{\bf w}$ are linearly independent, these two subspaces $\zeroregion{\alpha_{SK}^{\bf w}\tilde{\bf c}_{S}^{\bf w}+\tilde{\bf c}_{K}^{\bf w}}$ and $\zeroregion{\alpha_{TK}^{\bf w}\tilde{\bf c}_{T}^{\bf w}+\tilde{\bf c}_{K}^{\bf w}}$ are different. In particular, $\zeroregion{\alpha_{SK}^{\bf w}\tilde{\bf c}_{S}^{\bf w}+\tilde{\bf c}_{K}^{\bf w}} \cap \zeroregion{\alpha_{TK}^{\bf w}\tilde{\bf c}_{T}^{\bf w}+\tilde{\bf c}_{K}^{\bf w}}$ is a proper subspace of $\zeroregion{\alpha_{SK}^{\bf w}\tilde{\bf c}_{S}^{\bf w}+\tilde{\bf c}_{K}^{\bf w}}$. Hence, the dimension of their intersection is one. By \eqref{eq: orthogonal relation for modified vectors}, the vector $\bar{\bf g}^{\bf w}$ is orthogonal to both $\alpha_{SK}^{\bf w}\tilde{\bf c}_{S}^{\bf w}+\tilde{\bf c}_{K}^{\bf w}$ and $\alpha_{TK}^{\bf w}\tilde{\bf c}_{T}^{\bf w}+\tilde{\bf c}_{K}^{\bf w}$. Thus, this vector is in $\zeroregion{\alpha_{SK}^{\bf w}\tilde{\bf c}_{S}^{\bf w}+\tilde{\bf c}_{K}^{\bf w}} \cap \zeroregion{\alpha_{TK}^{\mathbf{w}}\tilde{\mathbf{c}}^{\mathbf{w}}_{T}+\tilde{\mathbf{c}}^{\mathbf{w}}_{K}}$. Since it is nonzero vector, the subspace $\zeroregion{\alpha_{SK}^{\bf w}\tilde{\bf c}_{S}^{\bf w}+\tilde{\bf c}_{K}^{\bf w}} \cap \zeroregion{\alpha_{TK}^{\mathbf{w}}\tilde{\mathbf{c}}^{\mathbf{w}}_{T}+\tilde{\mathbf{c}}^{\mathbf{w}}_{K}}$ is spanned by $\bar{\bf g}^{\bf w}$.
\end{proof}
\begin{lemma}\label{lem: observation of c for local upper bound}
Let $\Delta^{\geq {\bf w}}(B)$ be a branch. Let $\bar{\bf g}^{\bf w}=\tilde{\bf g}_{K}^{\bf w}-(\alpha_{SK}^{\bf w})^{-1}\tilde{\bf g}_{S}^{\bf w}-(\alpha_{TK}^{\bf w})^{-1}\tilde{\bf g}_{T}^{\bf w}$. Then, one of the normal vectors of $\mathcal{C}({\bf g}_{K}^{\bf w}, \bar{\bf g}^{\bf w})$ is
\begin{equation}
\alpha_{SK}^{\bf w}\tilde{\bf c}_{S}^{\bf w}-\alpha_{TK}^{\bf w}\tilde{\bf c}_{T}^{\bf w}.
\end{equation}
\end{lemma}
\begin{proof}
By \eqref{eq: orthogonal relation for modified vectors}, this can be shown by a direct calculation.
\end{proof}
Based on these facts, we define the following vectors and sets.
\begin{definition}\label{def: local upper bound}
For each branch $\Delta^{\geq {\bf w}}(B)$, let
\begin{equation}\label{eq: definition of gbar and cbar}
\bar{\bf g}^{\bf w}=\tilde{\bf g}_{K}^{\bf w}-(\alpha_{SK}^{\bf w})^{-1}\tilde{\bf g}_{S}^{\bf w}-(\alpha_{TK}^{\bf w})^{-1}\tilde{\bf g}_{T}^{\bf w},
\quad
\bar{\bf c}^{\bf w}=\alpha_{SK}^{\bf w}\tilde{\bf c}_{S}^{\bf w}-\alpha_{TK}^{\bf w}\tilde{\bf c}_{T}^{\bf w}.
\end{equation}
We also define the following two subsets of $\mathbb{R}^3$.
\begin{equation}\label{eq: upper bound region}
\mathscr{V}^{\bf w}_{\circ}=\mathcal{C}^{\circ}(\tilde{\bf g}_{S}^{\bf w},\tilde{\bf g}_{T}^{\bf w},\bar{\bf g}^{\bf w}),\quad
\mathscr{V}^{\bf w}=\mathcal{C}(\tilde{\bf g}_{S}^{\bf w},\tilde{\bf g}_{T}^{\bf w}) \cup \mathscr{V}^{\bf w}_{\circ}.
\end{equation}
The set $\mathscr{V}^{\mathbf{w}}$ is called the {\em local upper bound} of the branch $\Delta^{\geq \mathbf{w}}(B)$.
\end{definition}
Note that $\mathscr{V}_{\circ}^{\mathbf{w}}$ may be expressed as
\begin{equation}\label{eq: hyperplane expression of local upper bound}
\mathscr{V}_{\circ}^{\mathbf{w}}=\positiveregion{\tildelim\limits_{n \to \infty}\tilde{\mathbf{c}}_{S}^{\mathbf{w}S^n}} \cap \positiveregion{\tildelim\limits_{n \to \infty}\tilde{\mathbf{c}}_{S}^{\mathbf{w}TS^n}} \cap \positiveregion{\tilde{\mathbf{c}}_{K}^{\mathbf{w}}}.
\end{equation}
By definition, these two sets $\mathscr{V}_{\circ}^{\bf w}$ and $\mathscr{V}^{\bf w}$ are convex cones. Since
\begin{equation}\label{eq: initail inclusion lemma}
\tilde{\bf g}_{K}^{\bf w}=\bar{\bf g}^{\bf w}+(\alpha_{SK}^{\bf w})^{-1}\tilde{\bf g}_{S}^{\bf w}+(\alpha_{TK}^{\bf w})^{-1}\tilde{\bf g}_{T}^{\bf w} \in \mathscr{V}^{\bf w}_{\circ},
\end{equation}
we obtain $\mathcal{C}(G^{\bf w}) \subset \mathscr{V}^{\bf w}$. Hence, these vectors and sets can be illustrated as in Figure~\ref{fig: local upper bound of 3D picture} and Figure~\ref{fig: stereo of Vw}.

\begin{figure}[htbp]
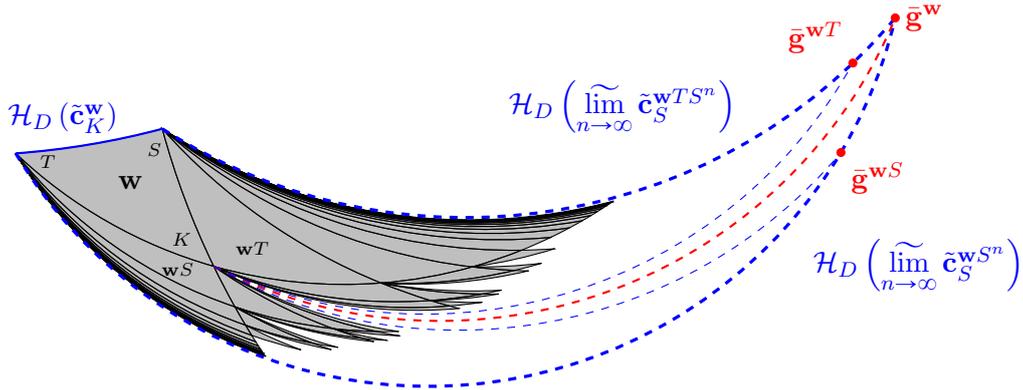

\centering
\subfile{figures/local_up_st/local_up_st}
\caption{Local upper bounds. The thick blue lines form $\mathscr{V}^{\mathbf{w}}$, and the thin blue lines form $\mathscr{V}^{\mathbf{w}S}$ and $\mathscr{V}^{\mathbf{w}T}$. The central red line is $\zeroregion{\bar{\mathbf{c}}^{\mathbf{w}}}$.}
\label{fig: stereo of Vw}
\end{figure}
The following theorem gives a strong restriction for the structure of branches in the $G$-fan.
\begin{theorem}[Local upper bound]\label{thm: local upper bound theorem}
For any branch $\Delta^{\geq {\bf w}}(B)$, the following inclusions hold.
\begin{equation}\label{eq: local upper bound inclusions}
|\Delta^{\geq {\bf w}}(B)| \subset \mathcal{C}(G^{\bf w}) \cup \mathscr{V}_{\circ}^{{\bf w}S} \cup \mathscr{V}_{\circ}^{{\bf w}T} \subset \mathscr{V}^{\bf w}.
\end{equation}
Moreover, the following statements hold.\\
\textup{($a$)} Two sets $\mathscr{V}_{\circ}^{{\bf w}S}$ and $\mathscr{V}_{\circ}^{{\bf w}T}$ are separated by the plane $\zeroregion{\bar{\bf c}^{\bf w}}$, that is, $\mathscr{V}_{\circ}^{{\bf w}S} \subset \negativeregion{\bar{\bf c}^{\bf w}}$ and $\mathscr{V}_{\circ}^{{\bf w}T} \subset \positiveregion{\bar{\bf c}^{\bf w}}$ hold.
Moreover, the union $\mathcal{C}(G^{\mathbf{w}}) \cup \mathscr{V}_{\circ}^{\mathbf{w}S} \cup \mathscr{V}_{\circ}^{\mathbf{w}T}$ is disjoint.
\\
\textup{($b$)} The local upper bounds are monotonically decreasing. Namely, $\mathscr{V}^{\bf w} \supset \mathscr{V}^{{\bf w}S}$ and $\mathscr{V}^{\bf w} \supset\mathscr{V}^{{\bf w}T}$ hold.
\end{theorem}
This theorem can be illustrated in Figure~\ref{fig: stereo of Vw}. To proceed, we first prove the following relations, which form a key point in the proof of \Cref{thm: local upper bound theorem}.
\begin{lemma}\label{lem: next g bar}
Let $\Delta^{\geq {\bf w}}(B)$ be a branch. Then, we have
\begin{equation}
\begin{aligned}
\bar{\bf g}^{{\bf w}S} &\sim \bar{\bf g}^{\bf w}+\left[(\alpha_{TK}^{\bf w})^{-1}-(\alpha_{TS}^{\bf w}\alpha_{SK}^{\bf w})^{-1}\right]\tilde{\bf g}_{T}^{\bf w},
\\
\bar{\bf g}^{{\bf w}T} &\sim \bar{\bf g}^{\bf w}+ \left[(\alpha_{SK}^{\bf w})^{-1}-(\alpha_{TS}^{\bf w}\alpha_{TK}^{\bf w})^{-1}\right]{\bf g}_{S}^{\bf w}.
\end{aligned}
\end{equation}
In particular, $\bar{\bf g}^{{\bf w}S} \in \mathcal{C}(\bar{\bf g}^{\bf w},\tilde{\bf g}_{T}^{\bf w})$ and $\bar{\bf g}^{{\bf w}T} \in \mathcal{C}(\bar{\bf g}^{\bf w},\tilde{\bf g}_{S}^{\bf w})$ hold.
\end{lemma}
\begin{proof}
We aim to prove the first relation for ${\bf w}S$. (The second one can be shown by a similar argument.)
By \Cref{def: local upper bound}, we have
\begin{equation}
\bar{\bf g}^{{\bf w}S}=\tilde{\bf g}_{K}^{{\bf w}S}-(\alpha_{SK}^{{\bf w}S})^{-1}\tilde{\bf g}_{S}^{{\bf w}S}-(\alpha_{TK}^{{\bf w}S})^{-1}\tilde{\bf g}_{T}^{{\bf w}S}.
\end{equation}
By \Cref{lem: S-mutations of modified c- g-vectors}, we may obtain $\alpha_{SK}^{{\bf w}S}=\alpha_{SK}^{\bf w}$ and $\alpha_{TK}^{{\bf w}S}=\alpha_{TS}^{\bf w}$. Moreover, by substituting the equalities in \Cref{lem: S-mutations of modified c- g-vectors}, we obtain
\begin{equation}
\begin{aligned}
\bar{\bf g}^{{\bf w}S}&=(-\tilde{\bf g}_{S}^{\bf w}+p_{SK}^{\bf w}\tilde{\bf g}_{K}^{\bf w})-(\alpha_{SK}^{\bf w})^{-1}\tilde{\bf g}_{K}^{\bf w}-(\alpha_{TS}^{\bf w})^{-1}\tilde{\bf g}_{T}^{\bf w}\\
&= \left[p_{SK}^{\bf w}-(\alpha_{SK}^{\bf w})^{-1}\right] \tilde{\bf g}_{K}^{\bf w} -\tilde{\bf g}_{S}^{\bf w} - (\alpha_{TS}^{\bf w})^{-1}\tilde{\bf g}_{T}^{\bf w}.
\end{aligned}
\end{equation}
By \eqref{eq: minimum polynomial of alpha}, we have $p_{SK}^{\bf w}-(\alpha_{SK}^{\bf w})^{-1}=\alpha_{SK}^{\bf w}$. Thus, it holds that
\begin{equation}\label{eq: g bar mutation}
\begin{aligned}
\bar{\bf g}^{{\bf w}S} &= \alpha_{SK}^{\bf w}\tilde{\bf g}_{K}^{\bf w}-\tilde{\bf g}_{S}^{\bf w}-(\alpha_{TS}^{\bf w})^{-1}\tilde{\bf g}_{T}^{\bf w}
\\
&\sim \tilde{\bf g}_{K}^{\bf w}-(\alpha_{SK}^{\bf w})^{-1}\tilde{\bf g}_{S}^{\bf w}-(\alpha_{SK}^{\bf w}\alpha_{TS}^{\bf w})^{-1}\tilde{\bf g}_{T}^{\bf w}\\
&= \bar{\bf g}^{\bf w}+\left[(\alpha_{TK}^{\bf w})^{-1}-(\alpha_{TS}^{\bf w}\alpha_{SK}^{\bf w})^{-1}\right]\tilde{\bf g}_{T}^{\bf w},
\end{aligned}
\end{equation}
where $\sim$ on the second row is obtained by multiplying $(\alpha_{SK}^{\bf w})^{-1}$ to the previous one. Moreover, by \eqref{eq: 3rd form of cluster-cyclic}, the coefficient of $\tilde{\bf g}^{\bf w}_{T}$ is non-negative. Thus, $\bar{\bf g}^{{\bf w}S} \in \mathcal{C}(\bar{\bf g},\tilde{\bf g}_{T}^{\bf w})$ holds.
\end{proof}

\begin{proof}[Proof of \Cref{thm: local upper bound theorem}]
Firstly, we show ($a$). The inclusion $\mathscr{V}_{\circ}^{{\bf w}S} \subset \positiveregion{\bar{\bf c}^{\bf w}}$ can be proved by:
\begin{equation}
\begin{aligned}
\langle \tilde{\bf g}_{S}^{{\bf w}S}, \bar{\bf c}^{\bf w} \rangle_{D}&=\langle \tilde{\bf g}_{K}^{{\bf w}}, \alpha_{SK}^{\bf w}\tilde{\bf c}_{S}^{\bf w}-\alpha_{TK}^{\bf w}\tilde{\bf c}_{T}^{\bf w} \rangle_{D}=0,
\\
\langle \tilde{\bf g}_{T}^{{\bf w}S}, \bar{\bf c}^{\bf w} \rangle_{D}&=\langle \tilde{\bf g}_{T}^{{\bf w}}, \alpha_{SK}^{\bf w}\tilde{\bf c}_{S}^{\bf w}-\alpha_{TK}^{\bf w}\tilde{\bf c}_{T}^{\bf w} \rangle_{D}=-\alpha_{TK}^{\bf w}<0,
\\
\langle \bar{\bf g}^{{\bf w}S}, \bar{\bf c}^{\bf w} \rangle_{D}& \sim \langle \bar{\bf g}^{\bf w}+\beta \tilde{\bf g}_{T}^{\bf w}, \alpha_{SK}^{\bf w}\tilde{\bf c}_{S}^{\bf w}-\alpha_{TK}^{\bf w}\tilde{\bf c}_{T}^{\bf w} \rangle_{D}=-\beta\alpha_{TK}^{\bf w} \leq 0,
\end{aligned}
\end{equation}
where we set $\beta=(\alpha_{TK}^{\bf w})^{-1}-(\alpha_{TS}^{\bf w}\alpha_{SK}^{\bf w})^{-1}\geq 0$. Note that the intersection $\mathcal{C}(\bar{\bf g}^{{\bf w}S},\tilde{\bf g}^{{\bf w}S}_{S},\tilde{\bf g}^{{\bf w}S}_{T}) \cap \zeroregion{\bar{\bf c}^{\bf w}}$ is a subset of $\mathcal{C}(\bar{\bf g}^{{\bf w}S},\tilde{\bf g}^{{\bf w}S}_{S})$, which does not intersect with $\mathscr{V}_{\circ}^{{\bf w}S}$. By the same argument, we may obtain $\mathscr{V}_{\circ}^{{\bf w}T} \subset \negativeregion{\bar{\bf c}^{\bf w}}$. The remaining problem is to show $\mathcal{C}(G^{\bf w}) \cap \mathscr{V}_{\circ}^{{\bf w}S}=\mathcal{C}(G^{\bf w}) \cap \mathscr{V}_{\circ}^{{\bf w}T}=\emptyset$. This is shown by $\mathcal{C}(G^{\mathbf{w}}) \subset \positiveclosure{\tilde{\mathbf{c}}^{\mathbf{w}}_{S}} \cap \positiveclosure{\tilde{\mathbf{c}}^{\mathbf{w}}_{T}}$, $\mathscr{V}_{\circ}^{\mathbf{w}S} \subset \positiveregion{\tilde{\mathbf{c}}_{K}^{\mathbf{w}S}}=\negativeregion{\tilde{\mathbf{c}}_{S}^{\mathbf{w}}}$, and $\mathscr{V}_{\circ}^{\mathbf{w}T} \subset \positiveregion{\tilde{\mathbf{c}}_{T}^{\mathbf{w}S}}=\negativeregion{\tilde{\mathbf{c}}_{T}^{\mathbf{w}}}$.
\par
Next, we show ($b$). By definition and \eqref{eq: initail inclusion lemma}, we have $\tilde{\bf g}_{S}^{{\bf w}S}=\tilde{\bf g}_{K}^{\bf w} \in \mathscr{V}^{\bf w}$ and $\tilde{\bf g}_{T}^{{\bf w}S}=\tilde{\bf g}_{T}^{\bf w} \in \mathscr{V}^{\bf w}$. By \Cref{lem: next g bar}, we have $\bar{\bf g}^{{\bf w}S} \in \mathcal{C}({\bf g}_{T}^{{\bf w}S},\bar{\bf g}^{\bf w})$, which is a subset of the boundary of $\mathscr{V}^{\bf w}$. Thus, we have $\mathscr{V}^{{\bf w}S} \subset \mathscr{V}^{\bf w}$.
\par
Lastly, we show the inclusions \eqref{eq: local upper bound inclusions}. The second one has already been shown by $(b)$. We aim to show the first one. For any cone $\mathcal{C}(G^{{\bf w}X})$ with $X=M_1M_2\cdots M_r \in \mathcal{M}$ ($r \geq 1$, $M_i \in \{S,T\}$),  the following inclusions hold.
\begin{equation}
\mathscr{V}^{{\bf w}M_1} \supset \mathscr{V}^{{\bf w}M_1M_2} \supset \cdots \supset \mathscr{V}^{{\bf w}X} \supset \mathcal{C}(G^{{\bf w}X}).
\end{equation}
Thus, we have $\mathcal{C}(G^{{\bf w}X}) \subset \mathscr{V}^{{\bf w}M_1} \subset \mathcal{C}(G^{\bf w}) \cup \mathscr{V}_{\circ}^{{\bf w}M_1}$. 
\end{proof}

\begin{example}
We give two examples in \Cref{fig: pictures of G-fans}. The left one is the case of $\nu_1=\nu_2=0$ and the right one is the case of $\nu_1,\nu_2 < 0$.
There is remarkable, special phenomenon when $\nu_1=\nu_2=0$, which is equivalent to $p_{12}=p_{23}=p_{31}=2$ as in \Cref{rem: classification of surface type}. In this case, for any ${\bf w} \in \mathcal{T}$, $p_{12}^{\bf w}=p_{23}^{\bf w}=p_{31}^{\bf w}=2$ and $\alpha_{12}^{\bf w}=\alpha_{23}^{\bf w}=\alpha_{31}^{\bf w}=1$ hold. By considering the equality in \Cref{lem: next g bar}, this condition implies that $\bar{\bf g}^{\bf w}\sim \bar{\bf g}^{{\bf w}S} \sim \bar{\bf g}^{{\bf w}T}$. Thus, if $\nu_1=\nu_2=0$, the direction of $\bar{\bf g}^{\bf w}$ depends on its maximal branch $\Delta^{\geq {[i]S^nT}}(B)$. In fact, more strongly, this vector $\bar{\mathbf{g}}^{\mathbf{w}}$ depends on the initial mutation $i=1,2,3$ only, which is given by 
\begin{equation}\bar{\mathbf{g}}^{\mathbf{w}}=\tilde{\mathbf{e}}_{k_0}-\tilde{\mathbf{e}}_{s_0}-\tilde{\mathbf{e}}_{t_0}.\end{equation}

\begin{figure}[htbp]
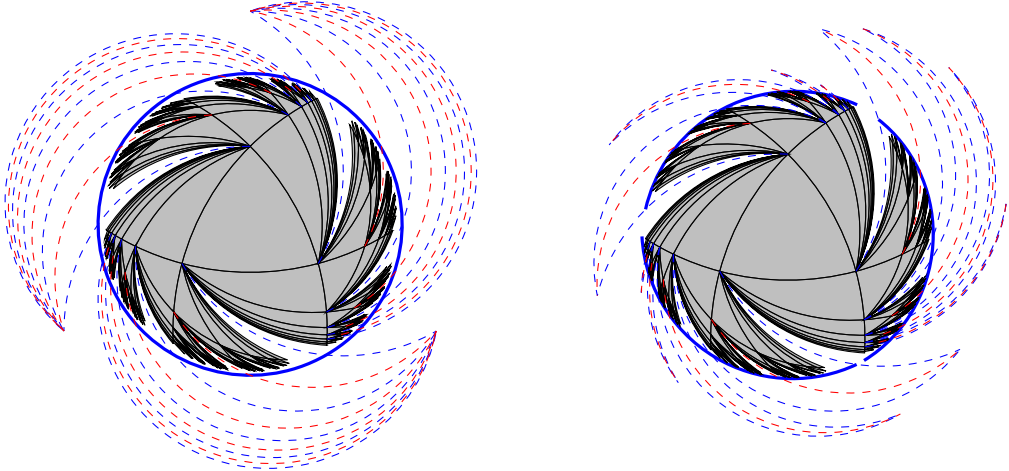
 
\centering
\begin{minipage}{0.45\textwidth}
\centering
\subfile{figures/Markov_G_fan/Markov_G_fan}
\end{minipage}
\begin{minipage}{0.45\textwidth}
\centering
\subfile{figures/G_fan_hyperbolic_case/G_fan_hyperbolic_case}
\end{minipage}
\caption{Pictures of $G$-fans. The left one corresponds to the case $\nu_1=\nu_2=0$ (e.g. the Markov quiver), and the right one corresponds to the case $\nu_1,\nu_2<0$ (general cluster-cyclic type). The thick blue lines indicate the global upper bounds. The dashed blue lines indicate the local upper bounds of maximal branches.}
\label{fig: pictures of G-fans}
\end{figure}
\end{example}

\section{Separateness among local upper bounds}\label{sec: separateness}
Recall that \Cref{thm: local upper bound theorem} guarantees that every $G$-cone except the initial one $\mathcal{C}(G^{\emptyset})$ is a subset of either $\mathscr{V}^{[i]S^nT}$ or $|\Delta^{<[i]S^{\infty}}(B)|$, where $i=1,2,3$ and $n \in \mathbb{Z}_{\geq 0}$. Let $\mathscr{T}_i$ be the interior of $|\Delta^{<[i]S^{\infty}}|$. By \eqref{eq: support of the trunk}, it is given by
\begin{equation}\label{eq: definition of T}
\mathscr{T}_i= \mathcal{C}^{\circ}(\tilde{\bf e}_{t_0},\tilde{\bf e}_{s_0},\alpha_{k_0s_0}\tilde{\bf e}_{s_0}-\tilde{\bf e}_{k_0})=\negativeregion{\tilde{\mathbf{e}}_{k_0}} \cap \positiveregion{\tilde{\mathbf{e}}_{t_0}} \cap \positiveregion{\alpha_{s_0k_0}\tilde{\mathbf{e}}_{k_0}+\tilde{\mathbf{e}}_{s_0}}.
\end{equation}
In this section, we establish the following proposition concerning the
separateness among the local upper bounds and the interiors of trunks.
\begin{proposition}\label{prop: separateness of upper bounds}
For any $W,U \in \{\mathscr{T}_1, \mathscr{T}_2, \mathscr{T}_3\}\cup\{\mathscr{V}_{\circ}^{[i]S^nT} \mid i=1,2,3;\ n \in \mathbb{Z}_{\geq 0}\}$, if $W \neq U$, we have
\begin{equation}
W \cap U = \emptyset.
\end{equation}
\end{proposition}
\subsection{Relations among maximal branches}
In this subsection, we fix an initial mutation direction $i\in \{1,2,3\}$, and set $k_0=K([i])$, $s_0=S([i])$, and $t_0=T([i])$ as in \Cref{tab: List of initial indices}.
\par
Firstly, we give the expressions of some vectors for maximal branches $\mathscr{V}_{\circ}^{[i]S^nT}$.
\begin{lemma}\label{lem: explicit formulas for the root of maximal branches}
Let $i=1,2,3$, and set $k_0,s_0,t_0 \in \{1,2,3\}$ as in \Cref{tab: List of initial indices}. Then, for any $n \in \mathbb{Z}_{\geq 0}$, we have
\begin{align}
\tilde{\bf g}_{M}^{[i]S^nT}&=\begin{cases}
-\tilde{\bf e}_{t_0}+p_{ST}^{[i]S^n}(u_n(p_{k_0s_0})\tilde{\bf e}_{s_0}-u_{n-1}(p_{k_0s_0})\tilde{\bf e}_{k_0}) & M=K,\\
u_n(p_{k_0s_0})\tilde{\bf e}_{s_0}-u_{n-1}(p_{k_0s_0})\tilde{\bf e}_{k_0} & M=S,\\
u_{n+1}(p_{k_0s_0})\tilde{\bf e}_{s_0}-u_{n}(p_{k_0s_0})\tilde{\bf e}_{k_0} & M=T,
\end{cases}\label{eq: g vectors at the root of maximal branches}\\
\bar{\bf g}^{[i]S^nT}
&
=
\begin{aligned}[t]
&-\tilde{\bf e}_{t_0}-\left[(\alpha_{TK}^{[i]S^n})^{-1}u_{n+1}(p_{k_0s_0})-\alpha_{ST}^{[i]S^n}u_{n}(p_{k_0s_0})\right]\tilde{\bf e}_{s_0}\\
&\ \qquad+\left[(\alpha_{TK}^{[i]S^n})^{-1}u_n(p_{k_0s_0})-\alpha_{ST}^{[i]S^n}u_{n-1}(p_{k_0s_0})\right]\tilde{\bf e}_{k_0}.
\end{aligned}
\label{eq: gbar  vector at the root of maximal branches}
\\
\tildelim_{m \to \infty} \tilde{\mathbf{c}}_{S}^{[i]S^nTS^m} &\sim \alpha_{ST}^{[i]S^n}\tilde{\mathbf{e}}_{t_0}+u_n(p_{k_0s_0})\tilde{\mathbf{e}}_{s_0}+u_{n+1}(p_{k_0s_0})\tilde{\mathbf{e}}_{k_0},
\label{eq: motivation of ci+n}
\\
\tildelim_{m \to \infty} \tilde{\mathbf{c}}_{S}^{[i]S^nT^2S^m} &\sim -(\alpha_{TK}^{[i]S^n})^{-1}\tilde{\mathbf{e}}_{t_0}-u_{n-1}(p_{k_0s_0})\tilde{\mathbf{e}}_{s_0}-u_n(p_{k_0s_0})\tilde{\mathbf{e}}_{k_0}.
\label{eq: motivation of ci-n}
\end{align}
\end{lemma}
\begin{proof}
They can be shown by \Cref{lem: infinite S mutation in trunks}, \Cref{lem: T-mutations of modified c- g-vectores}, and \Cref{lem: limit of S mutations}.
\end{proof}
Motivated by the equalities \eqref{eq: motivation of ci+n} and \eqref{eq: motivation of ci-n}, we define
\begin{align}
\mathfrak{c}_{i}^{+}(n)&=
\alpha_{ST}^{[i]S^n}\tilde{\mathbf{e}}_{t_0}+u_n(p_{k_0s_0})\tilde{\mathbf{e}}_{s_0}+u_{n+1}(p_{k_0s_0})\tilde{\mathbf{e}}_{k_0},\label{eq: definition of c+ vector}\\
\mathfrak{c}_{i}^{-}(n)&=-(\alpha_{TK}^{[i]S^n})^{-1}\tilde{\mathbf{e}}_{t_0}-u_{n-1}(p_{k_0s_0})\tilde{\mathbf{e}}_{s_0}-u_n(p_{k_0s_0})\tilde{\mathbf{e}}_{k_0}.\label{eq: definition of c- vector}
\end{align}
By \eqref{eq: hyperplane expression of local upper bound} and $\tilde{\mathbf{c}}_{K}^{[i]S^nT}=-\tilde{\mathbf{e}}_{t_0}$, we have
\begin{equation}\label{eq: first upper bound expression via the plane}
\mathscr{V}_{\circ}^{[i]S^nT}=\positiveregion{\mathfrak{c}_{i}^{+}(n)} \cap \positiveregion{\mathfrak{c}_{i}^{-}(n)} \cap \negativeregion{\tilde{\mathbf{e}}_{t_0}}. 
\end{equation}
Then, the relation among maximal branches can be given as follows:
\begin{lemma}\label{lem: hyperplane lemma}
Fix an initial mutation direction $i=1,2,3$. Then, for any $m,n \in \mathbb{Z}_{\geq 0}$, the following statements hold.
\\
\textup{($a$)} If $m<n$, we have
\begin{equation}
\mathscr{V}_{\circ}^{[i]S^mT} \subset \positiveregion{\mathfrak{c}_{i}^{+}(n)},
\quad
\mathscr{V}_{\circ}^{[i]S^mT} \subset \negativeregion{\mathfrak{c}_{i}^{-}(n)}.
\end{equation}
\textup{($b$)} If $m>n$, we have
\begin{equation}
\mathscr{V}_{\circ}^{[i]S^mT} \subset \negativeregion{\mathfrak{c}_{i}^{+}(n)},
\quad
\mathscr{V}_{\circ}^{[i]S^mT} \subset \positiveregion{\mathfrak{c}_{i}^{-}(n)}.
\end{equation}
\textup{($c$)} If $m=n$, we have
\begin{equation}
\mathscr{V}_{\circ}^{[i]S^nT} \subset \positiveregion{\mathfrak{c}_{i}^{+}(n)},
\quad
\mathscr{V}_{\circ}^{[i]S^nT} \subset \positiveregion{\mathfrak{c}_{i}^{-}(n)}.
\end{equation}
\end{lemma}
In this proof, we often use the following fact.
\begin{lemma}\label{lem: monotonicity of alpha product}
Fix one $\mathbf{w} \in \mathcal{T} \setminus \{\emptyset\}$, and set $p=p_{SK}^{\mathbf{w}}$ and $\alpha=\alpha_{SK}^{\mathbf{w}}$. For each $l\in \mbZ_{\geq 0}$, let $F(l)=\alpha_{TK}^{\mathbf{w}S^l}\alpha^l$ and $G(l)=\alpha_{TK}^{\mathbf{w}S^l}\alpha^{-l}$. Then, $F(l)$ is non-decreasing and $G(l)$ is non-increasing.
\end{lemma}
\begin{proof}
Note that $\alpha=\alpha_{SK}^{\mathbf{w}S^{l}}$ and $\alpha_{TK}^{\mathbf{w}S^{l+1}}=\alpha_{ST}^{\mathbf{w}S^{l}}$ hold by \eqref{eq: S-mutations of p}. Thus, we have
\begin{equation}
\begin{aligned}
F(l+1)-F(l)&=\alpha^l\left(\alpha_{ST}^{\mathbf{w}S^{l}}\alpha_{SK}^{\mathbf{w}S^{l}}-\alpha_{TK}^{\mathbf{w}S^{l}}\right) \geq 0,
\\
G(l+1)-G(l)&=\alpha^{-l-1}\left(\alpha_{ST}^{\mathbf{w}S^{l}}-\alpha_{TK}^{\mathbf{w}S^{l}}\alpha_{SK}^{\mathbf{w}S^{l}}\right) \leq 0,
\end{aligned}
\end{equation}
where the last inequalities are obtained by \eqref{eq: 3rd form of cluster-cyclic}.
\end{proof}

\begin{proof}[Proof of \Cref{lem: hyperplane lemma}]
For simplicity, we write $p=p_{k_0s_0}$ and $u_n=u_n(p)$ for any $n \in \mathbb{Z}_{\geq -2}$. 
\par
Firstly, we prove the claim for $\mathfrak{c}_{i}^{+}(n)$. By \Cref{lem: explicit formulas for the root of maximal branches} and \eqref{eq: definition of c+ vector}, we have
\begin{equation}\label{eq: g-frac c inner product}
\begin{aligned}
\langle \tilde{\bf g}_{S}^{[i]S^mT}, \mathfrak{c}_{i}^{+}(n) \rangle_D&=u_{m}u_{n}-u_{m-1}u_{n+1},
\quad
\langle \tilde{\bf g}_{T}^{[i]S^mT}, \mathfrak{c}_{i}^{+}(n) \rangle_D=u_{m+1}u_{n}-u_{m}u_{n+1},\\
\langle \bar{\bf g}^{[i]S^mT}, \mathfrak{c}_{i}^{+}(n) \rangle_D&=-\alpha_{ST}^{[i]S^n}-(\alpha_{TK}^{[i]S^m})^{-1}(u_{m+1}u_{n}-u_{m}u_{n+1})+\alpha_{ST}^{[i]S^m}(u_{m}u_{n}-u_{m-1}u_{n+1}).
\end{aligned}
\end{equation}
Set $f(m,n)=u_{m+1}u_{n}-u_{m}u_{n+1}$. Then, the above equalities can be expressed as
\begin{equation}
\begin{aligned}
\langle \tilde{\bf g}_{S}^{[i]S^mT}, \mathfrak{c}_{i}^{+}(n) \rangle_D&=f(m-1,n),
\quad
\langle \tilde{\bf g}_{T}^{[i]S^mT}, \mathfrak{c}_{i}^{+}(n) \rangle_D=f(m,n),\\
\langle \bar{\bf g}^{[i]S^mT}, \mathfrak{c}_{i}^{+}(n) \rangle_D&=-\alpha_{ST}^{[i]S^n}-(\alpha_{TK}^{[i]S^m})^{-1}f(m,n)+\alpha_{ST}^{[i]S^m}f(m-1,n).
\end{aligned}
\end{equation}
If $m,n \geq -1$, we have
\begin{equation}\label{eq: preservation of f}
f(m,n)=(pu_m-u_{m-1})u_n-u_m(pu_n-u_{n-1})=u_mu_{n-1}-u_{m-1}u_n=f(m-1,n-1).
\end{equation}
Let $m>n$. Then, by applying the above equality repeatedly, we have
\begin{equation}\label{eq: f expression m>n}
f(m,n)=f(m-n-2,-2)=-u_{m-n-1},
\quad
f(m-1,n)=f(m-n-3,-2)=-u_{m-n-2}.
\end{equation}
Since $m-n \geq 1$, we obtain
\begin{equation}
\langle \tilde{\bf g}_{S}^{[i]S^mT}, \mathfrak{c}_{i}^{+}(n) \rangle_D=-u_{m-n-2} \leq 0,
\quad
\langle \tilde{\bf g}_{T}^{[i]S^mT}, \mathfrak{c}_{i}^{+}(n) \rangle_D=-u_{m-n-1} < 0.
\end{equation}
Moreover, we obtain
\begin{equation}
\begin{aligned}
\langle \bar{\bf g}^{[i]S^mT}, \mathfrak{c}_{i}^{+}(n) \rangle_D
&=-\alpha_{ST}^{[i]S^n}+(\alpha_{TK}^{[i]S^m})^{-1}u_{m-n-1}-\alpha_{ST}^{[i]S^m}u_{m-n-2}
\\
&=-\alpha_{ST}^{[i]S^n}+(\alpha_{TK}^{[i]S^m})^{-1}\left(u_{m-n-1}-\alpha_{TK}^{[i]S^m}\alpha_{ST}^{[i]S^m}u_{m-n-2}\right).
\end{aligned}
\end{equation}
By \eqref{eq: 3rd form of cluster-cyclic}, we have $\alpha_{TK}^{[i]S^m}\alpha_{ST}^{[i]S^m} \geq \alpha_{SK}^{[i]S^m}$. Moreover, by \eqref{eq: infinite S mutations for p}, $p_{SK}^{[i]S^m}=p_{SK}^{[i]S^n}=p$ holds. So, let us write $\alpha=\alpha(p)=\alpha_{SK}^{[i]S^m}=\alpha_{SK}^{[i]S^n}$. Then, we have
\begin{equation}
\langle \bar{\bf g}^{[i]S^mT}, \mathfrak{c}_{i}^{+}(n) \rangle_D \leq -\alpha_{ST}^{[i]S^n}+(\alpha_{TK}^{[i]S^m})^{-1}\left(u_{m-n-1}-\alpha u_{m-n-2}\right)
\end{equation}
Set $l=m-n \geq 1$. Then, we have
\begin{equation}\label{eq: u-au equality1}
u_{l-1}-\alpha u_{l-2}= (p-\alpha)u_{l-2}-u_{l-3}=\alpha^{-1}u_{l-2}-u_{l-3}=\alpha^{-1}(u_{l-2}-\alpha u_{l-3}).
\end{equation}
By applying this equality repeatedly, we obtain 
\begin{equation}\label{eq: u-au equality2}
u_{l-1}-\alpha u_{l-2}=\alpha^{1-l}
\end{equation}
and
\begin{equation}
\langle \bar{\bf g}^{[i]S^mT}, \mathfrak{c}_{i}^{+}(n) \rangle_D \leq -\alpha_{ST}^{\mathbf{w}}+(\alpha_{TK}^{\mathbf{w}S^l})^{-1}\alpha^{1-l},
\end{equation}
where we set $\mathbf{w}=[i]S^n$. If $l=1$, the right hand side is non-positive by $\alpha_{SK}^{\mathbf{w}},\alpha_{TK}^{\mathbf{w}} \geq 1$. Moreover, since $(\alpha_{TK}^{\mathbf{w}S^l})^{-1}\alpha^{1-l}=\alpha(\alpha_{TK}^{\mathbf{w}S^{l}}\alpha^{l})^{-1}$, the right hand side is non-increasing by \Cref{lem: monotonicity of alpha product}. Thus, $\langle \bar{\bf g}^{[i]S^mT}, \mathfrak{c}_{i}^{+}(n) \rangle_D \leq 0$ holds. Therefore, we obtain $\mathscr{V}_{\circ}^{[i]S^mT} \subset \negativeregion{\mathfrak{c}_{i}^{+}(n)}$ when $m>n$.
\par
Let $m \leq n$. Then, by applying \eqref{eq: preservation of f} repeatedly, we have
\begin{equation}\label{eq: f expression m<n}
f(m,n)=f(-2,n-m-2)=u_{n-m-1},
\quad
f(m-1,n)=f(-2,n-m-1)=u_{n-m}.
\end{equation}
Thus, by \eqref{eq: g-frac c inner product}, we have
\begin{equation}
\langle \tilde{\bf g}_{S}^{[i]S^mT}, \mathfrak{c}_{i}^{+}(n) \rangle_D=u_{n-m} > 0,
\quad
\langle \tilde{\bf g}_{T}^{[i]S^mT}, \mathfrak{c}_{i}^{+}(n) \rangle_D=u_{n-m-1} > 0,
\end{equation}
and
\begin{equation}
\begin{aligned}
\langle \bar{\bf g}^{[i]S^mT}, \mathfrak{c}_{i}^{+}(n) \rangle_D&=-\alpha_{ST}^{[i]S^n}-(\alpha_{TK}^{[i]S^m})^{-1}u_{n-m-1}+\alpha_{ST}^{[i]S^m}u_{n-m}
\\
&=-\alpha_{ST}^{[i]S^n}+ \alpha_{ST}^{[i]S^m}\left(u_{n-m}-(\alpha_{ST}^{[i]S^m}\alpha_{TK}^{[i]S^m})^{-1}u_{n-m-1}\right)
\\
\overset{\eqref{eq: 3rd form of cluster-cyclic}}&{\geq} -\alpha_{ST}^{[i]S^n}+\alpha_{ST}^{[i]S^m}\left(u_{n-m} - \alpha^{-1}u_{n-m-1}\right),
\end{aligned}
\end{equation}
where $\alpha=\alpha_{SK}^{[i]S^m}$. Set $\mathbf{w}=[i]S^m$ and $l=n-m \geq 0$. Then, by the same argument in \eqref{eq: u-au equality1} and \eqref{eq: u-au equality2}, we have 
\begin{equation}\label{eq: u-ainvu equality}
u_l-\alpha^{-1}u_{l-1}=\alpha^{l}.
\end{equation}
Thus, we obtain
\begin{equation}
\langle \bar{\bf g}^{[i]S^mT}, \mathfrak{c}_{i}^{+}(n) \rangle_D \geq -\alpha_{ST}^{\mathbf{w}S^{l}} + \alpha_{ST}^{\mathbf{w}}\alpha^{l}=\alpha^{l}\left(-\alpha_{ST}^{\mathbf{w}S^l}\alpha^{-l}+\alpha_{ST}^{\mathbf{w}}\right).
\end{equation}
When $l=0$, $-\alpha_{ST}^{\mathbf{w}S^l}\alpha^{-l}+\alpha_{ST}^{\mathbf{w}}=0$ holds. Moreover, by \Cref{lem: monotonicity of alpha product}, $\alpha_{ST}^{\mathbf{w}S^l}\alpha^{-l}=\alpha \alpha_{TK}^{\mathbf{w}S^{l+1}}\alpha^{-(l+1)}$ is non-increasing. Thus, $-\alpha_{ST}^{\mathbf{w}S^l}\alpha^{-l}+\alpha_{ST}^{\mathbf{w}}$ is non-decreasing, and we have $\langle \bar{\bf g}^{[i]S^mT}, \mathfrak{c}_{i}^{+}(n) \rangle_D \geq 0$. Therefore, we obtain $\mathscr{V}_{\circ}^{[i]S^mT} \subset \positiveregion{\mathfrak{c}_{i}^{+}(n)}$ when $m \leq n$.
\par
For $\mathfrak{c}_{i}^{-}(n)$, we have
\begin{equation}
\begin{aligned}
\langle \tilde{\bf g}_{S}^{[i]S^mT}, \mathfrak{c}_{i}^{-}(n) \rangle_D&=-f(m-1,n-1),
\quad
\langle \tilde{\bf g}_{T}^{[i]S^mT}, \mathfrak{c}_{i}^{-}(n) \rangle_D=-f(m,n-1),\\
\langle \bar{\bf g}^{[i]S^mT}, \mathfrak{c}_{i}^{-}(n) \rangle_D&=(\alpha_{TK}^{[i]S^n})^{-1}+(\alpha_{TK}^{[i]S^m})^{-1}f(m,n-1)-\alpha_{ST}^{[i]S^m}f(m-1,n-1).
\end{aligned}
\end{equation}
If $m<n$, by \eqref{eq: f expression m<n}, we have
\begin{equation}
\langle \tilde{\bf g}_{S}^{[i]S^mT}, \mathfrak{c}_{i}^{-}(n) \rangle_D=-u_{n-m-1} < 0,
\quad
\langle \tilde{\bf g}_{T}^{[i]S^mT}, \mathfrak{c}_{i}^{-}(n) \rangle_D= -u_{n-m-2} \leq 0,
\end{equation}
and
\begin{equation}
\begin{aligned}
\langle \bar{\bf g}^{[i]S^mT},\mathfrak{c}_{i}^{-}(n) \rangle_D&=(\alpha_{TK}^{[i]S^n})^{-1}+(\alpha_{TK}^{[i]S^m})^{-1}u_{n-m-2}-\alpha_{ST}^{[i]S^m}u_{n-m-1}
\\
&=
(\alpha_{TK}^{[i]S^n})^{-1}-\alpha_{ST}^{[i]S^m}\left(u_{n-m-1}-(\alpha_{TK}^{[i]S^m}\alpha_{ST}^{[i]S^m})^{-1}u_{n-m-2}\right)
\\
\overset{\eqref{eq: 3rd form of cluster-cyclic}}&{\leq}
(\alpha_{TK}^{[i]S^n})^{-1}-\alpha_{ST}^{[i]S^m}\left(u_{n-m-1}-\alpha^{-1}u_{n-m-2}\right).
\end{aligned}
\end{equation}
Set $l=n-m \geq 1$ and $\mathbf{w}=[i]S^m$. By \eqref{eq: u-ainvu equality}, we have
\begin{equation}
\langle \bar{\bf g}^{[i]S^mT},\mathfrak{c}_i(n) \rangle_D \leq (\alpha_{TK}^{\mathbf{w}S^l})^{-1}-\alpha_{ST}^{\mathbf{w}}\alpha^{l-1}=(\alpha_{TK}^{\mathbf{w}S^l})^{-1}\left(1-\alpha_{ST}^{\mathbf{w}}\alpha_{TK}^{\mathbf{w}S^l}\alpha^{l-1}\right).
\end{equation}
If $l=1$, $1-\alpha_{ST}^{\mathbf{w}}\alpha_{TK}^{\mathbf{w}S^l}\alpha^{l-1}$ is non-positive because $\alpha_{ST}^{\mathbf{w}},\alpha_{TK}^{\mathbf{w}S^l} \geq 1$. Moreover, by \Cref{lem: monotonicity of alpha product}, the factor $1-\alpha_{ST}^{\mathbf{w}}\alpha_{TK}^{\mathbf{w}S^l}\alpha^{l-1}$ is non-increasing. Thus, $\langle \bar{\bf g}^{[i]S^mT},\mathfrak{c}_{i}^{-}(n) \rangle_D \leq 0$ holds. Therefore, we obtain $\mathscr{V}_{\circ}^{[i]S^mT} \subset \negativeregion{\mathfrak{c}_{i}^{-}(n)}$ when $m<n$.
\par
If $m \geq n$, by \eqref{eq: f expression m>n}, we have
\begin{equation}
\langle \tilde{\bf g}_{S}^{[i]S^mT}, \mathfrak{c}_{i}^{-}(n) \rangle_D=u_{m-n-1}>0,
\quad
\langle \tilde{\bf g}_{T}^{[i]S^mT}, \mathfrak{c}_{i}^{-}(n) \rangle_D=u_{m-n}>0,
\end{equation}
and
\begin{equation}
\begin{aligned}
\langle \bar{\bf g}^{[i]S^mT}, \mathfrak{c}_{i}^{-}(n) \rangle_D&=(\alpha_{TK}^{[i]S^n})^{-1}-(\alpha_{TK}^{[i]S^m})^{-1}u_{m-n}+\alpha_{ST}^{[i]S^m}u_{m-n-1}
\\
&=
(\alpha_{TK}^{[i]S^n})^{-1}-(\alpha_{TK}^{[i]S^m})^{-1}\left(u_{m-n}-\alpha_{TK}^{[i]S^m}\alpha_{ST}^{[i]S^m}u_{m-n-1}\right)
\\
\overset{\eqref{eq: 3rd form of cluster-cyclic}}&{\geq}
(\alpha_{TK}^{[i]S^n})^{-1}-(\alpha_{TK}^{[i]S^m})^{-1}\left(u_{m-n}-\alpha u_{m-n-1}\right).
\end{aligned}
\end{equation}
Set $l=m-n \geq 0$ and $\mathbf{w}=[i]S^n$. Then, by \eqref{eq: u-au equality2}, we have
\begin{equation}
\langle \bar{\bf g}^{[i]S^mT}, \mathfrak{c}_{i}^{-}(n) \rangle_D \geq (\alpha_{TK}^{\mathbf{w}})^{-1}-(\alpha_{TK}^{\mathbf{w}S^l})^{-1}\alpha^{-l} \geq (\alpha_{TK}^{\mathbf{w}})^{-1}-(\alpha_{TK}^{\mathbf{w}S^0})^{-1}\alpha^{0}=0
\end{equation}
by \Cref{lem: monotonicity of alpha product}. Therefore, we obtain $\mathscr{V}_{\circ}^{[i]S^mT} \subset \positiveregion{\mathfrak{c}_{i}^{-}(n)}$ when $m \geq n$.
\end{proof}
Thanks to \Cref{lem: hyperplane lemma}, every local upper bound $\mathscr{V}_{\circ}^{[i]S^mT}$ is eventually included in $\positiveregion{\mathfrak{c}_{i}^{+}(n)}$ for enough large $n$. 
The limit of $\mathfrak{c}_{i}^{+}(n)$ is given as follows.
\begin{lemma}
For each $i=1,2,3$, we have
\begin{equation}
\tildelim_{n \to \infty} \mathfrak{c}_{i}^{+}(n) \sim (-p_{s_0t_0}+\alpha_{s_0k_0}p_{k_0t_0})\tilde{\mathbf{e}}_{t_0}+\tilde{\mathbf{e}}_{s_0}+\alpha_{s_0k_0}\tilde{\mathbf{e}}_{k_0}
\end{equation}
\end{lemma}

\begin{proof}
Set $p=p_{s_0k_0}$ and $\alpha=\alpha_{s_0k_0}$. Suppose $p>2$.
By \eqref{eq: limit of Chebychev polynomials} and \eqref{eq: p in trunk}, we have
\begin{equation}
\lim_{n \to \infty} \frac{\sqrt{p^2-4}}{\alpha^{n+1}}p_{ST}^{[i]S^{n}}=- p_{s_0t_0} + \alpha p_{k_0t_0}.
\end{equation}
Since $\alpha_{ST}^{[i]S^n}=\frac{1}{2}\left(p_{ST}^{[i]S^n}+\sqrt{(p_{ST}^{[i]S^n})^2-4}\right)$, we also have
\begin{equation}
\lim_{n \to \infty} \frac{\sqrt{p^2-4}}{\alpha^{n+1}}\alpha_{ST}^{[i]S^n}=- p_{s_0t_0} + \alpha p_{k_0t_0}
\end{equation}
Therefore, we obtain
\begin{equation}
\lim\limits_{n \to \infty} \frac{\sqrt{p^2-4}}{\alpha^{n+1}}\mathfrak{c}_{i}^{+}(n)=(- p_{s_0t_0} + \alpha p_{k_0t_0})\tilde{\mathbf{e}}_{t_0}+\tilde{\mathbf{e}}_{s_0}+\alpha\tilde{\mathbf{e}}_{k_0}.
\end{equation}
When $p=2$, the claim is shown by replacing $\frac{\sqrt{p^2-4}}{\alpha^{n+1}}$ with $\frac{1}{n}$ in the above argument.
\end{proof}
Based on the observations in this subsection, let us define
\begin{equation}
\mathfrak{c}_{i}^{+}=(-p_{s_0t_0}+\alpha_{s_0k_0}p_{k_0t_0})\tilde{\mathbf{e}}_{t_0}+\tilde{\mathbf{e}}_{s_0}+\alpha_{s_0k_0}\tilde{\mathbf{e}}_{k_0},
\quad
\mathfrak{c}_{i}^{-}=\mathfrak{c}_{i}^{-}(0)=-\tilde{\mathbf{e}}_{k_0}-\alpha_{t_0k_0}^{-1}\tilde{\mathbf{e}}_{t_0}.
\end{equation}
Then, we obtain the following upper bound.
\begin{lemma}\label{lem: rough upper bound for branches}
For each $i=1,2,3$, we have
\begin{equation}\label{eq: rough upper bound for branches}
\bigcup_{m \in \mathbb{Z}_{\geq 0}} \mathscr{V}_{\circ}^{[i]S^mT} \subset \positiveregion{\mathfrak{c}_{i}^{+}} \cap \positiveregion{\mathfrak{c}_{i}^{-}} \cap \negativeregion{\tilde{\mathbf{e}}_{t_0}}.
\end{equation}
\end{lemma}

\begin{proof}
The inclusions $\mathscr{V}_{\circ}^{[i]S^mT} \subset \negativeregion{\tilde{\mathbf{e}}_{t_0}}$ and $\mathscr{V}_{\circ}^{[i]S^mT} \subset \positiveregion{\mathfrak{c}_{i}^{-}}$ are shown by \eqref{eq: first upper bound expression via the plane} and \Cref{lem: hyperplane lemma}. We show $\mathscr{V}_{\circ}^{[i]S^mT} \subset \positiveregion{\mathfrak{c}_{i}^{+}}$. Let $\mathbf{x} \in \mathscr{V}_{\circ}^{[i]S^mT}$. Then, by \Cref{lem: hyperplane lemma}, we have $\langle \mathbf{x}, \mathfrak{c}_{i}^{+}(n) \rangle_D > 0$ for any $n > m$. Since $\mathfrak{c}_{i}^{+}\sim\tildelim\limits_{n \to \infty} \mathfrak{c}_{i}^{+}(n)$, by taking its limit with some normalization, we have
\begin{equation}
\langle \mathbf{x}, \mathfrak{c}_{i}^{+} \rangle_D \geq 0.
\end{equation}
Thus, we have $\mathscr{V}_{\circ}^{[i]S^mT} \subset \positiveclosure{\mathfrak{c}_{i}^{+}}$. Since $\mathscr{V}_{\circ}^{[i]S^mT}$ is an open set, this implies $\mathscr{V}_{\circ}^{[i]S^mT} \subset \left(\positiveclosure{\mathfrak{c}_{i}^{+}}\right)^{\circ}=\positiveregion{\mathfrak{c}_{i}^{+}}$ as desired.
\end{proof}
From \Cref{lem: rough upper bound for branches}, we may give other rough but simple upper bounds via some planes. For example, since
\begin{equation}
-\tilde{\mathbf{e}}_{k_0}-\tilde{\mathbf{e}}_{t_0}=\mathfrak{c}_{i}^{-}-(1-\alpha_{t_0k_0}^{-1})\tilde{\mathbf{e}}_{t_0}
\end{equation}
and $1-\alpha_{t_0k_0}^{-1} \geq 0$, we have
\begin{equation}\label{eq: 2nd rough upper bound}
\bigcup_{m \in \mathbb{Z}_{\geq 0}}\mathscr{V}_{\circ}^{[i]S^mT} \subset \positiveregion{-\tilde{\mathbf{e}}_{k_0}-\tilde{\mathbf{e}}_{t_0}}.
\end{equation}
This is because, if $\mathbf{x}$ belongs to the set in the right hand size of \eqref{eq: rough upper bound for branches}, then we have
\begin{equation}
\langle \mathbf{x}, -\tilde{\mathbf{e}}_{t_0}-\tilde{\mathbf{e}}_{k_0} \rangle_D = \langle \mathbf{x}, \mathfrak{c}_{i}^{-} \rangle_D-(1-\alpha_{s_0k_0}^{-1})\langle \mathbf{x}, \tilde{\mathbf{e}}_{t_0} \rangle_D > 0.
\end{equation}
Similarly, since
\begin{equation}
\tilde{\mathbf{e}}_{s_0}+\alpha_{s_0k_0}\tilde{\mathbf{e}}_{k_0}=\mathfrak{c}_{i}^{+}-(-p_{s_0t_0}+\alpha_{s_0k_0}p_{k_0t_0})\tilde{\mathbf{e}}_{t_0}
\end{equation}
and $-p_{s_0t_0}+\alpha_{s_0k_0}p_{k_0t_0} \geq 0$ by \eqref{eq: 2nd form of cluster-cyclic}, we have
\begin{equation}\label{eq: 1st rough upper bound}
\bigcup_{m \in \mathbb{Z}_{\geq 0}}\mathscr{V}_{\circ}^{[i]S^mT} \subset \positiveregion{\tilde{\mathbf{e}}_{s_0}+\alpha_{s_0k_0}\tilde{\mathbf{e}}_{k_0}}.
\end{equation}
Moreover, by \eqref{eq: support of the trunk}, this region $\positiveregion{\tilde{\mathbf{e}}_{s_0}+\alpha_{s_0k_0}\tilde{\mathbf{e}}_{k_0}}$ also includes the trunk $\Delta^{<[i]S^{\infty}}(B)$ except for one line $\mathcal{C}(\tilde{\mathbf{e}}_{t_0})$. Thus, we obtain the following proposition.
\begin{proposition}
For each $i=1,2,3$, we have
\begin{equation}
|\Delta^{\geq [i]}(B)| \subset \positiveregion{\tilde{\mathbf e}_{s_0}+\alpha_{s_0k_0}\tilde{\mathbf{e}}_{k_0}} \cup \mathcal{C}(\tilde{\mathbf{e}}_{t_0}).
\end{equation}
\end{proposition}
\begin{proof}
This follows from \eqref{eq: support of the trunk} and \eqref{eq: 1st rough upper bound}. 
\end{proof}

\subsection{Proof of \Cref{prop: separateness of upper bounds}} In this subsection, we focus on proving \Cref{prop: separateness of upper bounds}.
\begin{proof}[Proof of \Cref{prop: separateness of upper bounds}]
We prove the claim by case-by-case calculation in the following list:
\begin{enumerate}
\item $W=\mathscr{T}_i$ and $U=\mathscr{T}_j$ with $i \neq j$.\label{item: case 1 for the proof of separateness}
\item $W=\mathscr{V}_{\circ}^{[i]S^nT}$ and $U=\mathscr{V}_{\circ}^{[i]S^mT}$ with $n \neq m$.\label{item: case 2 for the proof of separateness}
\item $W=\mathscr{V}_{\circ}^{[i]S^nT}$ and $U=\mathscr{T}_j$.\label{item: case 3 for the proof of separateness}
\item $W=\mathscr{V}_{\circ}^{[i]S^nT}$ and $U=\mathscr{V}_{\circ}^{[j]S^mT}$ with $i \neq j$.\label{item: case 4 for the proof of separateness}
\end{enumerate}
In this proof, we set $k_0=K([i])$, $s_0=S([i])$, and $t_0=T([i])$.
\\
\eqref{item: case 1 for the proof of separateness} This follows from $\mathscr{T}_i \subset \negativeregion{\tilde{\mathbf{e}}_{i}} \cap \positiveregion{\tilde{\mathbf{e}}_{j}}$ and $\mathscr{T}_j \subset \positiveregion{\tilde{\mathbf{e}}_{i}} \cap \negativeregion{\tilde{\mathbf{e}}_{j}}$ by \Cref{fig: trunks}.
\\
\eqref{item: case 2 for the proof of separateness} This follows from \Cref{lem: hyperplane lemma}.
\\
\eqref{item: case 3 for the proof of separateness} If $j=i$, this follows from $\mathscr{V}_{\circ}^{[i]S^nT} \subset \negativeregion{\tilde{\mathbf{e}}_{t_0}}$ (\Cref{lem: rough upper bound for branches}) and $\mathscr{T}_i \subset \positiveregion{\tilde{\mathbf{e}}_{t_0}}$. Suppose $i \neq j$. If $j=s_0$, we have $K([j])=s_0$, $S([j])=t_0$, and $T([j])=k_0$ by \eqref{eq: i to s0 indices}. By \eqref{eq: definition of T}, we have
\begin{equation}
\mathscr{T}_{j}=\negativeregion{\tilde{\mathbf{e}}_{s_0}} \cap \positiveregion{\tilde{\mathbf{e}}_{k_0}} \cap \positiveregion{\alpha_{t_0s_0}\tilde{\mathbf{e}}_{s_0}+\tilde{\mathbf{e}}_{t_0}}.
\end{equation}
Note that $\negativeregion{\tilde{\mathbf{e}}_{s_0}} \cap \positiveregion{\alpha_{t_0s_0}\tilde{\mathbf{e}}_{s_0}+\tilde{\mathbf{e}}_{t_0}} \subset \positiveregion{\tilde{\mathbf{e}}_{t_0}}$ since for any ${\bf x}$ in the left hand side,
\begin{equation}
\langle \mathbf{x}, \tilde{\mathbf{e}}_{t_0}\rangle_D = \langle \mathbf{x}, \alpha_{t_0s_0}\tilde{\mathbf{e}}_{s_0}+\tilde{\mathbf{e}}_{t_0} \rangle_D - \alpha_{t_0s_0} \langle \mathbf{x}, \tilde{\mathbf{e}}_{s_0} \rangle_D > 0. 
\end{equation}
Hence, $\mathscr{T}_{j} \subset \positiveregion{\tilde{\mathbf{e}}_{t_0}}$ holds. On the other hand, by \eqref{eq: rough upper bound for branches}, we have $\mathscr{V}_{\circ}^{[i]S^nT} \subset \negativeregion{\tilde{\mathbf{e}}_{t_0}}$. Thus, they are disjoint. If $j=t_0$, we have $K([j])=t_0$, $S([j])=k_0$, and $T([j])=s_0$ by \eqref{eq: i to t0 indices}. Then, we have
\begin{equation}
\mathscr{T}_{j}=\negativeregion{\tilde{\mathbf{e}}_{t_0}} \cap \positiveregion{\tilde{\mathbf{e}}_{s_0}} \cap \positiveregion{\alpha_{k_0t_0}\tilde{\mathbf{e}}_{t_0}+\tilde{\mathbf{e}}_{k_0}}.
\end{equation}
Note that $\negativeregion{\tilde{\mathbf{e}}_{t_0}} \cap \positiveregion{\alpha_{k_0t_0}\tilde{\mathbf{e}}_{t_0}+\tilde{\mathbf{e}}_{k_0}} \subset \positiveregion{\tilde{\mathbf{e}}_{t_0}+\tilde{\mathbf{e}}_{k_0}}$ holds because, for any $\mathbf{x}$ in the left hand side,
\begin{equation}
\langle \mathbf{x}, \tilde{\mathbf{e}}_{t_0}+\tilde{\mathbf{e}}_{k_0} \rangle_D=\langle \mathbf{x}, \alpha_{k_0t_0}\tilde{\mathbf{e}}_{t_0}+\tilde{\mathbf{e}}_{k_0} \rangle_D -(\alpha_{k_0t_0}-1) \langle \mathbf{x}, \tilde{\mathbf{e}}_{t_0} \rangle_D > 0,
\end{equation}
where we use $\alpha_{k_0t_0} \geq 1$. On the other hand, by \eqref{eq: 2nd rough upper bound}, we have $\mathscr{V}_{\circ}^{[i]S^nT} \subset \negativeregion{\tilde{\mathbf{e}}_{t_0}+\tilde{\mathbf{e}}_{k_0}}$ holds. Thus, the claim holds.
\\
\eqref{item: case 4 for the proof of separateness} In this case, we may additionally assume $j=t_0$ without loss of generality. (If $j=s_0$, then $i=T([j])$ holds.) Then, we have $K([j])=t_0$, $S([j])=k_0$, and $T([j])=s_0$ by \eqref{eq: i to t0 indices}. By \eqref{eq: 1st rough upper bound}, we have $\mathscr{V}_{\circ}^{[j]S^mT} \subset \positiveregion{\tilde{\mathbf{e}}_{k_0}+\alpha_{k_0t_0}\tilde{\mathbf{e}}_{t_0}}$.
On the other hand, $\negativeregion{\tilde{\mathbf{e}}_{t_0}+\tilde{\mathbf{e}}_{k_0}} \cap \negativeregion{\mathbf{e}_{t_0}} \subset \negativeregion{\tilde{\mathbf{e}}_{k_0}+\alpha_{k_0t_0}\tilde{\mathbf{e}}_{t_0}}$ holds since for any $\mathbf{x}$ in the left hand side,
\begin{equation}
\langle \mathbf{x}, \tilde{\mathbf{e}}_{k_0}+\alpha_{k_0t_0}\tilde{\mathbf{e}}_{t_0} \rangle_D
=\langle \mathbf{x}, \tilde{\mathbf{e}}_{k_0}+\tilde{\mathbf{e}}_{t_0} \rangle_D + (\alpha_{s_0k_0}-1)\langle \mathbf{x}, \tilde{\mathbf{e}}_{t_0} \rangle_D<0.
\end{equation} 
Thus, by \eqref{eq: rough upper bound for branches} and \eqref{eq: 2nd rough upper bound}, we have $\mathscr{V}_{\circ}^{[i]S^nT} \subset \negativeregion{\tilde{\mathbf{e}}_{k_0}+\alpha_{k_0t_0}\tilde{\mathbf{e}}_{t_0}}$, which completes the proof.
\end{proof}

\section{Applications to $g$-vectors}\label{sec: application}
In this section, we present an application of \Cref{thm: global upper bound theorem}, \Cref{thm: local upper bound theorem}, and \Cref{prop: separateness of upper bounds} to the $g$-vectors, especially about the non-periodicity and the sign structure.
\subsection{Non-periodicity of $g$-vectors}
In \cite[Thm.~5.2]{AC25b}, we showed that there is no periodicity among $G$-matrices, nor among $G$-cones. Here, we prove a stronger result on $g$-vectors.
\par
We focus on an equality of $g$-vectors $\mathbf{g}_l^{\mathbf{w}}=\mathbf{g}_{m}^{\mathbf{u}}$ for $\mathbf{w},\mathbf{u} \in \mathcal{T}$ and $l,m=1,2,3$. An equality $\mathbf{g}_l^{\mathbf{w}}=\mathbf{g}_m^{\mathbf{u}}$ is said to be {\em trivial} if $l=m$ and ${\bf u}={\bf w}[k_1,\dots,k_r]$ with $k_i \neq l$. (In this case, this equality always holds by the mutation rules \eqref{eq: mutation of g-vectors} of $g$-vectors.)
\begin{theorem}\label{thm: non periodicity}
Suppose that an initial exchange matrix is cluster-cyclic of rank $3$. Then, all equalities ${\bf g}^{\bf w}_{l}={\bf g}^{\bf u}_{m}$ ($l,m=1,2,3$, $\mathbf{w},\mathbf{u} \in \mathcal{T}$) of $g$-vectors are trivial.
\end{theorem}
\begin{proof}
More strongly, we claim that there are no nontrivial cone equalities $\mathcal{C}(\mathbf{g}_{l}^{\mathbf{w}})=\mathcal{C}(\mathbf{g}_{m}^{\mathbf{u}})$. For this purpose, we can use the modified $g$-vectors instead of ordinary ones because $\mathcal{C}(\mathbf{g}_{l}^{\mathbf{w}})=\mathcal{C}(\tilde{\mathbf{g}}_{l}^{\mathbf{w}})$ holds.
Note that by \Cref{lem: S-mutations of modified c- g-vectors} and \Cref{lem: T-mutations of modified c- g-vectores}, for any $M=S,T$, $\tilde{\mathbf{g}}_{S}^{\mathbf{w}M}$ and $\tilde{\mathbf{g}}_{T}^{\mathbf{w}M}$ respectively coincide with one of $\{\tilde{\mathbf{g}}_{K}^{\mathbf{w}}, \tilde{\mathbf{g}}_{S}^{\mathbf{w}}, \tilde{\mathbf{g}}_{T}^{\mathbf{w}}\}$, and the induced equalities are trivial. Moreover, all modified $g$-vectors except the initial ones $\tilde{\mathbf{e}}_1$, $\tilde{\mathbf{e}}_2$, $\tilde{\mathbf{e}}_3$ can be expressed as the form $\tilde{\mathbf{g}}_{K}^{\mathbf{w}}$ with $\mfw \neq \emptyset$. Thus, if a nontrivial equality exists, one of them should be expressed as $\mathcal{C}(\tilde{\mathbf{g}}_{K}^{\mathbf{w}})=\mathcal{C}(\tilde{\mathbf{g}}_{K}^{\mathbf{u}})$ with $\mathbf{w} \neq \mathbf{u}$ or $\mathcal{C}(\tilde{\mathbf{g}}^{\mathbf{w}}_{K})=\mathcal{C}(\tilde{\mathbf{e}}_m)$, where $m\in \{1,2,3\}$ and $\mfu,\mfw \neq \emptyset$. We will prove that this phenomenon never occurs.
\par
If both $\mathbf{w}$ and $\mathbf{u}$ are in trunks, this fact can be shown by \Cref{lem: infinite S mutations}. Suppose that $\mathbf{w}$ is in a branch, that is, ${\bf w}=[i]S^nTX$ for some $i=1,2,3$, $n \in \mathbb{Z}_{\geq 0}$, and $X \in \mathcal{M}$. Then, by \eqref{eq: initail inclusion lemma} and \Cref{thm: local upper bound theorem}, we have $\tilde{\mathbf{g}}_K^{\mathbf{w}} \in \mathscr{V}_{\circ}^{\mathbf{w}} \subset \mathscr{V}_{\circ}^{[i]S^nT}$. If $\mathbf{u}$ is in a trunk, then $\tilde{\mathbf{g}}_{K}^{\mathbf{u}}$ is on the boundary of some $\overline{\mathscr{T}}_j$ ($j=1,2,3$), which is the closure of $\mathscr{T}_j$. Since $\mathscr{V}_{\circ}^{\bf w}$ is an open set, $\mathscr{V}_{\circ}^{[i]S^nT} \cap \overline{\mathscr{T}}_j=\emptyset$ holds by \Cref{prop: separateness of upper bounds}. Thus, $\mathcal{C}(\tilde{\mathbf{g}}_K^{\mathbf{w}}) \neq \mathcal{C}(\tilde{\mathbf{g}}_K^{\mathbf{u}})$ holds. The same argument works if we replace $\tilde{\mathbf{g}}_K^{\mathbf{u}}$ with $\tilde{\mathbf{e}}_m$ because $\tilde{\mathbf{e}}_m$ also belongs to $\overline{\mathscr{T}}_j$. Suppose that $\mathbf{u}$ is also in a branch, that is $\mathbf{u}=[j]S^mTY$ for some $j=1,2,3$, $m \in \mathbb{Z}_{\geq 0}$, and $Y \in \mathcal{M}$. If $(i,n) \neq (j,m)$, then the claim follows from \Cref{prop: separateness of upper bounds}. (Note that $\tilde{\mathbf{g}}_{K}^{\mathbf{w}} \in \mathscr{V}_{\circ}^{[i]S^nT}$ and $\tilde{\mathbf{g}}_{K}^{\mathbf{u}} \in \mathscr{V}_{\circ}^{[j]S^mT}$.) We now assume that $(i,n)=(j,m)$.
Let $\mathbf{w}_0$ be the maximum common prefix of $\mathbf{w}$ and $\mathbf{u}$. Since $[i]S^nT \leq \mathbf{w}_0$, this prefix $\mathbf{w}_0$ also belongs to a branch. Then, by \Cref{thm: local upper bound theorem}, each of $\tilde{\mathbf{g}}_{K}^{\mathbf{w}}$ and $\tilde{\mathbf{g}}_{K}^{\mathbf{u}}$ belongs to one of $\mathcal{C}(G^{\mathbf{w}_0})$, $\mathscr{V}_{\circ}^{\mathbf{w}_0S}$, or $\mathscr{V}_{\circ}^{\mathbf{w}_0T}$. Furthermore, since $\mathbf{w}_0$ is the maximum common prefix of $\mathbf{w}$ and $\mathbf{u}$, these two vectors cannot be contained in the same set. Note that these three sets are pairwise disjoint. Hence, we obtain $\mathcal{C}(\tilde{\mathbf{g}}_{K}^{\mathbf{w}}) \neq \mathcal{C}(\tilde{\mathbf{g}}_K^{\mathbf{u}})$.
\end{proof}

Thanks to the results in \cite{CKLP13, GHKK18, CL20, Nak23}, we can restate this theorem through cluster variables.
Let us consider an {\em integer} cluster-cyclic initial exchange matrix. In this case, cluster variables $x_{i}^{\mathbf{w}}$ ($i=1,2,3$, $\mathbf{w} \in \mathcal{T}$) are defined as in \cite{FZ02}. By \cite[Thm.~II.7.2]{Nak23}, the equality $\mathbf{g}_{l}^{\mathbf{w}}=\mathbf{g}_{m}^{\mathbf{u}}$ of $g$-vectors holds if and only if the equality $x_{l}^{\mathbf{w}}=x_{m}^{\mathbf{u}}$ of cluster variables holds for any $l,m=1,2,3$ and $\mathbf{w},\mathbf{u} \in \mathcal{T}$. Thus, \Cref{thm: non periodicity} implies the following corollary. (Here, we define the triviality of equalities of cluster variables in the same way as $g$-vectors.)
\begin{corollary}\label{cor: no periodicity among cluster variables}
Suppose that an initial exchange matrix is {\em integer} cluster-cyclic of rank $3$. Then, all equalities of cluster variables $x_{l}^{\mathbf{w}}=x_{m}^{\mathbf{u}}$ are trivial. 
\end{corollary}

\subsection{Sign structure of $g$-vectors}
It is known that every row vector of each $G$-matrix is sign-coherent \cite{GHKK18}. Let $\tau_{i}^{\bf w} \in \{\pm 1\}$ be the sign of the $i$th row vector of $G^{\bf w}$. As expected from \Cref{fig: main graph} and \Cref{fig: orthants}, these signs are given as follows.
\begin{theorem}\label{thm: sign of G matrix}
Consider the $G$-matrices associated with a cluster-cyclic initial exchange matrix $B$ of rank $3$. Fix an initial mutation direction $i=1,2,3$, and set $k_0=K([i])$, $s_0=S([i])$, $t_0=T([i])$ as in \Cref{tab: List of initial indices}. Then, for each $n \in \mathbb{Z}_{\geq 0}$ and $X \in \mathcal{M}$, the signs of $G$-matrices are given in \Cref{tab: list of G-signs}.
\begin{table}[htbp]
\centering
\[
\begin{array}{c|cccc}
{\bf w} & [i]S^n & [i]S^{n+1}TX & [i]TS^n & [i]TS^nTX \\ \hline
\rule{0pt}{25pt}
\left(\begin{matrix}
\tau_{k_0}^{\bf w}\\ \tau_{s_0}^{\bf w}\\ \tau_{t_0}^{\bf w}
\end{matrix}\right)
&
\left(\begin{matrix}
-\\ + \\ +
\end{matrix}\right)
&
\left(\begin{matrix}
-\\ + \\ -
\end{matrix}\right)
&
\left(\begin{matrix}
-\\ + \\ -
\end{matrix}\right)
&
\left(\begin{matrix}
+\\ + \\ -
\end{matrix}\right)
\end{array}
\]
\caption{The list of signs of $G$-matrices.}
\label{tab: list of G-signs}
\end{table}
\end{theorem}

\begin{example}
For a given $B$ and $i$, we can easily obtain the characterization of $\mathbf{w}$ in \Cref{tab: list of G-signs}. For example, consider the case of \eqref{eq: initial conditions for example of tropical signs}. Then, by \Cref{tab: List of initial indices}, we have $k_0=1$, $s_0=3$, and $t_0=2$. By following the rule in \Cref{lem: K S T mutation formula}, each $\mathbf{w} \geq [i]$ can be characterized by the following list. See also  \eqref{eq: index S-mutations}.
\begin{itemize}
\item $\mathbf{w}=[i]S^n$: A finite sequence of the form $\mathbf{w}=[k_0,s_0,k_0,s_0,\dots]$.
\item $\mathbf{w}=[i]S^{n+1}TX$: A finite sequence starting with $[k_0,s_0]$ and containing $t_0$ at least once. (Namely, $\mathbf{w}=[k_0,s_0,\dots,t_0,\dots]$.)
\item $\mathbf{w}=[i]TS^{n}$: A finite sequence of the form $\mathbf{w}=[k_0,t_0,s_0,t_0,s_0,t_0,s_0,\dots]$.
\item $\mathbf{w}=[i]TS^{n}TX$: A finite sequence starting with $[k_0,t_0]$ and containing a subsequent occurrence of $k_0$. (Namely, $\mathbf{w}=[k_0,t_0,\dots,k_0,\dots]$.)
\end{itemize}
\end{example}
\begin{remark} \label{rem: one excptional sign}
For the later purpose, we slightly refine \Cref{thm: sign of G matrix}. Consider a $g$-vector $\mathbf{g}_j^{\mathbf{w}}=(x_1,x_2,x_3)^{\top}$ with $\mathbf{w} \geq [i]$. Then, by \Cref{thm: sign of G matrix}, we directly obtain
\begin{equation}\label{eq: sign s0}
x_{s_0} \geq 0.
\end{equation}
Consider the sign of $x_{t_0}$. Then, by \Cref{thm: sign of G matrix}, if $\mathbf{w} \notin \mathcal{T}^{< [i]S^{\infty}}$, we have $x_{t_0} \leq 0$. In fact, by \eqref{eq: g vectors in trunk}, the unique positive vector is $\mathbf{e}_{t_0}$. Thus, we can conclude that, except for one vector $\mathbf{g}_{t_0}^{[i]S^n}=\mathbf{e}_{t_0}$, we have
\begin{equation}\label{eq: sign t0}
x_{t_0} \leq 0.
\end{equation}
\end{remark}

The claim for the first case $\mfw=[i]S^n$ in \Cref{thm: sign of G matrix} can be obtained directly by \Cref{lem: infinite S mutation in trunks}. For the third case $\mfw=[i]TS^n$, this follows from the following formulas.
\begin{lemma}
Set $p'=p_{SK}^{[i]T}$.
For any $n \in \mathbb{Z}_{\geq 0}$, we have
\begin{equation}\label{eq: iTS^n expressions}
\begin{aligned}
\tilde{\bf g}_{K}^{[i]TS^n}&=u_{n+1}(p')\tilde{\bf e}_{s_0}-u_n(p')\tilde{\bf e}_{t_0},
\\
\tilde{\bf g}_{S}^{[i]TS^n}&=u_n(p')\tilde{\bf e}_{s_0}-u_{n-1}(p')\tilde{\bf e}_{t_0},\\
\tilde{\bf g}_{T}^{[i]TS^n}&=-\tilde{\bf e}_{k_0}+p_{k_0s_0}\tilde{\bf e}_{s_0}.
\end{aligned}
\end{equation}
\end{lemma}
\begin{proof}
When $n=0$, the claim holds by a direct calculation based on \Cref{lem: kst} and \Cref{lem: T-mutations of modified c- g-vectores}. Moreover, by taking $\mfw=[i]T$ in \eqref{eq: infinite S mutations for g vectors} , we obtain the formula.
\end{proof}
In the following proof, we always assume that $B$ is skew-symmetric. By \Cref{prop: skew-symmetrizing method}, it is enough to show this claim. In this case, by taking $D=\mathrm{diag}(1,1,1)$, each modified $g$-vector coincides with the corresponding ordinary $g$-vector, and the inner product $\langle\,,\,\rangle_D$ coincides with the usual Euclidean inner product $\langle\,,\,\rangle$.
\subsubsection{Key lemma}
For the proof of \Cref{thm: sign of G matrix}, we focus on a skew-symmetric initial exchange matrix $B$, and fix $i$, $k_0$, $s_0$, and $t_0$ as stated before. As obtained in \Cref{lem: eigenvalues lemma for cluster-cyclic matrix}, let $\lambda$ be the unique positive eigenvalue of the symmetric pseudo Cartan companion $A^{[i]}$ of $B^{[i]}$, and let $\mathbf{v} = (v_1,v_2,v_3)^{\top}$ be the corresponding eigenvector given by \eqref{eq: positive eigenvector of lambda}. To simplify the notation, we write 
\begin{equation}\label{eq: notation in the proof of sign structure}
q_{lm}=p_{lm}^{[i]}, \quad \beta_{lm}=\alpha_{lm}^{[i]}.
\end{equation}
Note that $q_{lm}$ and $\beta_{lm}$ do not correspond to the $(l,m)$th entry of the initial exchange matrix $B$ but to that of $B^{[i]}$ (although they coincide if $l=i$ or $m=i$).
\par
The following inequality plays an essential role in the forthcoming proof.
\begin{lemma}
For any $l,m \in \{1,2,3\}$ with $l \neq m$, we have
\begin{equation}\label{eq: alpha v inequality}
\beta_{lm}v_l-v_m \geq 0.
\end{equation}
\end{lemma}
\begin{proof}
Let $n \in \{1,2,3\} \setminus \{l,m\}$.
To simplify the notation, we write $p=q_{lm}$, $q=q_{ln}$, $r=q_{mn}$, and $\alpha_p=\beta_{lm}$. Then, by \eqref{eq: eigenvector of DA}, we have
\begin{equation}
\begin{aligned}
v_{l}&=(\lambda-2)^2+(p+q)(\lambda-2)+pr+qr-r^2,
\\
v_{m}&=(\lambda-2)^2+(p+r)(\lambda-2)+pq+rq-q^2.
\end{aligned}
\end{equation}
By a direct calculation, we have
\begin{equation}
\begin{aligned}
\beta_{lm} v_l - v_m 
&=(\alpha_p-1)(\lambda-2)^2 + \alpha_p qr -\alpha_p r^2 - qr + q^2
\\
&\ \qquad +[(\alpha_p -1)p+(\alpha_p q -r)](\lambda-2)+p(\alpha_p r-q)
\\
&\geq(\alpha_p-1)(\lambda-2)^2 + \alpha_p qr -\alpha_p r^2 - qr + q^2,
\end{aligned}
\end{equation}
where the last inequality follows from \eqref{eq: 2nd form of cluster-cyclic} and $\alpha_p-1 \geq 0$.
Since $(\lambda - 2)^2 > p^2+q^2+r^2$ by \eqref{eq: lambda inequality}, we have
\begin{equation}
\begin{aligned}
\beta_{lm} v_l - v_m &\geq (\alpha_p-1)(p^2+q^2+r^2)+\alpha_p qr - qr - \alpha_p r^2 + q^2
\\
&= (\alpha_p - 1)p^2+q(\alpha_p q - r)+r(\alpha_p q -r) \geq 0.
\end{aligned}
\end{equation}
\end{proof}
\subsubsection{Case of $[i]TS^nTX$}
Let $\mathcal{D}$ be the union of $\mathcal{C}(G^{[i]TS^nTX})$ with $n \in \mathbb{Z}_{\geq 0}$ and $X \in \mathcal{M}$. Let $\mathcal{D}^{\circ}$ be its interior.
We give the following upper bound.
\begin{lemma}
Suppose that $B$ is skew-symmetric. Then, we have
\begin{equation}\label{eq: upper bound for [i]TSTX}
\mathcal{D}^{\circ} \subset \Epositiveregion{-{\bf e}_{t_0}-\beta_{k_0t_0}{\bf e}_{k_0}} \cap \Epositiveregion{{\bf e}_{k_0}} \cap \Epositiveregion{\mathbf{v}}.
\end{equation}
\end{lemma}
\begin{proof}
The inclusion $\mathcal{D}^{\circ} \subset \Epositiveregion{\mathbf{v}}$ follows from \Cref{thm: global upper bound theorem}.
By \Cref{thm: local upper bound theorem}, we have $\mathcal{D}^{\circ} \subset \mathscr{V}_{\circ}^{[i]T} = \Epositiveregion{\tildelim\limits_{n \to \infty} \mathbf{c}_S^{[i]TS^n}} \cap \Epositiveregion{\tildelim\limits_{n \to \infty} \mathbf{c}_S^{[i]T^2S^n}} \cap \Enegativeregion{\mathbf{e}_{t_0}}$.
By \eqref{eq: S mutation of c- g-vectors}, \eqref{eq: T mutation of c- g-vectors in trunks}, and \eqref{eq: T mutation of c- g-vectors in branches}, for any $m \in \mathbb{Z}_{\geq 0}$, we have
\begin{equation}
\mathbf{c}_{K}^{[i]TS^mT}=-\mathbf{c}_{T}^{[i]TS^m}=-\mathbf{c}_{T}^{[i]T}=-\mathbf{c}_{K}^{[i]}=\mathbf{e}_{k_0}.
\end{equation}
In particular, by \eqref{eq: hyperplane expression of local upper bound}, we have $\mathscr{V}_{\circ}^{[i]TS^m} \subset \Epositiveregion{\mathbf{e}_{k_0}}$.
Since $\mathcal{D}^{\circ}$ is a subset of the union of $\mathscr{V}_{\circ}^{[i]TS^mT}$, we have $\mathcal{D}^{\circ} \subset \Epositiveregion{\mathbf{e}_{k_0}}$. 
Moreover, by \Cref{lem: T-mutations of modified c- g-vectores} and \Cref{lem: limit of S mutations}, we have $\tildelim\limits_{n \to \infty} {\bf c}_{S}^{[i]T^2S^n} \sim -{\bf e}_{t_0}-\beta_{k_0t_0}{\bf e}_{k_0}$. Thus, we also obtain $\mathcal{D}^{\circ} \subset \Epositiveregion{-{\bf e}_{t_0}-\beta_{k_0t_0}{\bf e}_{k_0}}$.
\end{proof}
In what follows, let us prove this case in \Cref{thm: sign of G matrix}.
\begin{proof}
Let $\mathbf{x}=(x_1,x_2,x_3)^{\top} \in 
\Epositiveregion{-{\bf e}_{t_0}-\beta_{k_0t_0}{\bf e}_{k_0}} \cap \Epositiveregion{{\bf e}_{k_0}} \cap \Epositiveregion{\mathbf{v}}$.
By $\mathbf{x} \in \Epositiveregion{\mathbf{e}_{k_0}}$, $x_{k_0} > 0$ holds. By $\mathbf{x} \in \Epositiveregion{-\mathbf{e}_{t_0}-\beta_{k_0t_0}\mathbf{e}_{k_0}}$, we have $-x_{t_0}-\beta_{k_0t_0} x_{k_0} > 0$. This implies
\begin{equation}
x_{t_0}<-\beta_{k_0t_0} x_{k_0} <0.
\end{equation}
Since $\mathbf{x} \in \Epositiveregion{\mathbf{v}}$, we have $x_{k_0}v_{k_0}+x_{s_0}v_{s_0}+x_{t_0}v_{t_0} > 0$. Since $x_{t_0}<-\beta_{k_0t_0} x_{k_0}$, we have
\begin{equation}
x_{s_0}v_{s_0}>-x_{k_0}v_{k_0}-x_{t_0}v_{t_0}>x_{k_0}(\beta_{k_0t_0} v_{t_0} - v_{k_0}) \geq 0,
\end{equation}
where the last inequality follows from \eqref{eq: alpha v inequality}. This completes the proof.
\end{proof}

\subsubsection{Case of $[i]S^{n+1}TX$}
In parallel to the proof of the case of $[i]S^nTX$, we give the following upper bound.
\begin{lemma}
Suppose that $B$ is skew-symmetric. Then, we have
\begin{equation}
\bigcup_{n \in \mathbb{Z}_{\geq 0}} |\Delta^{\geq [i]S^{n+1}T}(B)|^{\circ} \subset \Epositiveregion{\mathfrak{c}_{i}^{-}(1)} \cap \Enegativeregion{\mathbf{e}_{t_0}} \cap \Epositiveregion{\mathbf{v}},
\end{equation}
where $A^{\circ}$ denotes the interior of $A \subset \mathbb{R}^{3}$.
\end{lemma}
\begin{proof}
This immediately follows from \Cref{thm: global upper bound theorem}, \Cref{lem: hyperplane lemma}, and \eqref{eq: rough upper bound for branches}.
\end{proof}
As a consequence, we can completely finish the proof of \Cref{thm: sign of G matrix}.
\begin{proof}
Let $\mathbf{x}=(x_1,x_2,x_3)^{\top} \in \Epositiveregion{\mathfrak{c}_{i}^{-}(1)} \cap \Enegativeregion{\mathbf{e}_{t_0}} \cap \Epositiveregion{\mathbf{v}_i}$. By $\mathbf{x} \in \Enegativeregion{\mathbf{e}_{t_0}}$, we have $x_{t_0} < 0$. By $\mathbf{x} \in \Epositiveregion{\mathfrak{c}_{i}^{-}(1)}$ and $p_{k_0s_0}=p_{k_0s_0}^{[i]}=q_{k_0s_0}$, we have
\begin{equation}
\langle \mathbf{x}, \mathfrak{c}_{i}^{-}(1) \rangle = -(\alpha_{TK}^{[i]S})^{-1}x_{t_0}-x_{s_0}-q_{k_0s_0}x_{k_0} > 0.
\end{equation}
Note that, by \eqref{eq: S-mutations of p}, we have $\alpha_{TK}^{[i]S}=\alpha_{ST}^{[i]}=\beta_{s_0t_0}$. Thus, we obtain
\begin{equation}\label{eq: lemma for the case of iSTX 1}
-\beta_{s_0t_0}^{-1}x_{t_0} - x_{s_0} - q_{k_0s_0}x_{k_0} > 0.
\end{equation}
By $\mathbf{x} \in \Epositiveregion{\mathbf{v}}$, we have
\begin{equation}\label{eq: lemma for the case of iSTX 2}
v_{t_0}x_{t_0}+v_{s_0}x_{s_0}+v_{k_0}x_{k_0}>0.
\end{equation}
By \eqref{eq: lemma for the case of iSTX 1} and \eqref{eq: lemma for the case of iSTX 2}, we obtain
\begin{equation}
(v_{t_0}-\beta_{s_0t_0}^{-1}v_{s_0})x_{t_0}+(v_{k_0}-q_{k_0s_0}v_{s_0})x_{k_0}>0,
\end{equation}
and it implies
\begin{equation}\label{eq: case of iSTX k0<0}
(q_{k_0s_0}v_{s_0}-v_{k_0})x_{k_0}<\beta_{s_0t_0}^{-1}(\beta_{s_0t_0}v_{t_0}-v_{s_0})x_{t_0}.
\end{equation}
Then, by $x_{t_0}<0$ and \eqref{eq: alpha v inequality}, the right hand side is non-positive. On the other hand, since $q_{k_0s_0} > \beta_{k_0s_0}$, we have
\begin{equation}
q_{k_0s_0}v_{s_0}-v_{k_0} > \beta_{k_0s_0}v_{s_0}-v_{k_0} \geq 0.
\end{equation}
Thus, the inequality \eqref{eq: case of iSTX k0<0} implies $x_{k_0}<0$. By $x_{t_0},x_{k_0}<0$ and \eqref{eq: lemma for the case of iSTX 2}, we obtain $x_{s_0}>0$.
\end{proof}

\section{Under the minimum assumption}\label{sec: under the minimum assumption}
For most rank $3$ cluster-cyclic matrices, it is known that their mutation equivalence classes contain the minimum element. By choosing the initial exchange matrix to be such a minimum one, several additional structural properties emerge. In this section, we establish the monotonicity of $g$-vectors and introduce a simple global upper bound under this condition.
\subsection{Minimum assumption}
To state the minimum assumption, we define the preorder $\leq$ on $\mathrm{M}_3(\mbR)$ as
\begin{equation}
A=(a_{ij}) \leq A'=(a'_{ij}) \Longleftrightarrow |a_{ij}| \leq |a'_{ij}| \quad (\forall i,j=1,2,3).
\end{equation}
Note that by \eqref{eq: rank 3 skew-symmetrizable condition} and \eqref{eq: mutation of b}, for any $j\in \{1,2,3\}$, either $\mu_{j}(B) \geq B$ or $B \geq \mu_{j}(B)$ (or both) holds. Moreover, the following fact is known.
\begin{lemma}[{\cite[Lem.~2.1~(a)]{BBH11}},\ {\cite[Lem.~7.5]{Aka24}}]\label{lem: monotonicity of B}
For any cluster-cyclic skew-symmetrizable matrix $B$, there are at least two indices $j\in \{1,2,3\}$ satisfying $\mu_{j}(B) \geq B$. Moreover, we set $\{j,l,m\} = \{1,2,3\}$. If $\mu_j(B) \geq B$, the largest entry in $\{p_{lm}^{[j]},p_{jl}^{[j]},p_{jm}^{[j]}\}$ is $p_{lm}^{[j]}$.
\end{lemma}

\begin{definition}
For any cluster-cyclic skew-symmetrizable matrix $B \in \mathrm{M}_{3}(\mathbb{R})$, we call the sequence (admitting empty, finite, and infinite) $\bm{\delta} (B)=[k_1,k_2,\dots]$ a {\em decreasing sequence} for $B$ if the following statements hold.
\begin{itemize}
\item For each $l=1,2,\dots$, the inequality $\mu_{k_{l}}(B) \geq B$ does {\em not} hold. (Note that the statements exclude the possibility that $\mu_{k_l}(B)=-B$.)
\item If $\bm{\delta}(B)$ is a finite sequence, then it holds that $\mu_{j}(B^{\bm{\delta}(B)}) \geq B^{\bm{\delta}(B)}$ for any $j=1,2,3$.
\end{itemize}
If the decreasing sequence for $B$ is empty ($\bm{\delta}(B)=\emptyset$), we say that $B$ is {\em minimum} in this $B$-pattern ${\bf B}(B)$.
\end{definition}
By \Cref{lem: monotonicity of B}, this decreasing sequence $\bm{\delta}(B)$ is uniquely determined by $B$. Moreover, since $\mu_j$ is an involution, $\bm{\delta}(B)$ should be reduced.
\par
By the following property, most $B$-patterns corresponding to cluster-cyclic matrices have the minimum element.
\begin{lemma}[{\cite[Lem.~7.5, Thm.~10.3]{Aka24}}]\label{lem: the condition for minimum assumption}
Let $B \in \mathrm{M}_{3}(\mathbb{R})$ be a cluster-cyclic skew-symmetrizable matrix.
\\
\textup{($a$)} The decreasing sequence $\bm{\delta}(B)$ is finite if and only if either of the following holds.
\begin{itemize}
\item $C(B) < 4$.
\item $\mathrm{Sk}(B)$ is mutation equivalent to the transposition or permutation of the following matrix.
\begin{equation}
\left(\begin{matrix}
0 & -p & p\\
p & 0 & -2\\
-p & 2 & 0
\end{matrix}\right) \quad (p \in \mathbb{R}_{\geq 2}).
\end{equation}
\end{itemize}
\textup{($b$)} Suppose that $\bm{\delta}(B)$ is finite. Let $B'$ be mutation equivalent to $B$. Then, $\bm{\delta}(B')$ is also finite, and $B^{\bm{\delta}(B)}=\pm(B')^{\bm{\delta}(B')}$ holds. In other words, if the minimum element $B^{\bm{\delta}(B)}$ exists, it is unique up to the difference of the sign. 
\end{lemma}
For any {\em integer} cluster-cyclic exchange matrix, its decreasing sequence is known to be finite. Thus, we can take the minimum element in its mutation equivalence class.
\begin{example}
We give three types of examples as follows.
\\
($a$) Consider the case of
\begin{equation}\label{eq: non minimum example}
B=\left(\begin{matrix}
0 & -228 & 1795\\
228 & 0 & -409252\\
-1795 & 409252 & 0
\end{matrix}\right).
\end{equation}
The Markov constant is $C(B)=-7$.
Then, we obtain $\bm{\delta}(B)=[1,2,3,2,1]$ as follows.
\begin{equation}
\begin{tikzpicture}[baseline={(0,-1.25)}]
\node (A0) at (0,0) {$\left(\begin{smallmatrix}
0 & -228 & 1795\\
228 & 0 & -409252\\
-1795 & 409252 & 0
\end{smallmatrix}\right)$};
\node (A1) at (4,0) {$\left(\begin{smallmatrix}
0 & 228 & -1795\\
-228 & 0 & 8\\
1795 & -8 & 0
\end{smallmatrix}\right)$};
\node (A2) at (7.5,0) {$\left(\begin{smallmatrix}
0 & -228 & 29\\
228 & 0 & -8\\
-29 & 8 & 0
\end{smallmatrix}\right)$};
\node (A3) at (10.5,0) {$\left(\begin{smallmatrix}
0 & 4 & -29\\
-4 & 0 & 8\\
29 & -8 & 0
\end{smallmatrix}\right)$};
\node (A4) at (0,-2) {$\left(\begin{smallmatrix}
0 & -4 & 3\\
4 & 0 & -8\\
-3 & 8 & 0
\end{smallmatrix}\right)$};
\node (A5) at (4,-2) [draw, rectangle] {$\left(\begin{smallmatrix}
0 & 4 & -3\\
-4 & 0 & 4\\
3 & -4 & 0
\end{smallmatrix}\right)$};
\node (A6) at (7.5,-1.5) {$\left(\begin{smallmatrix}
0 & -4 & 13\\
4 & 0 & -4\\
-13 & 4 & 0
\end{smallmatrix}\right)$};
\node (A7) at (7.5,-2.5) {$\left(\begin{smallmatrix}
0 & -8 & 3\\
8 & 0 & -4\\
-3 & 4 & 0
\end{smallmatrix}\right)$};
\draw[->] (A3.south west)->(A4.north east) node [pos=0.7, auto=left, sloped] {$\leq$} node [pos=0.8, auto=right] {$2$};
\draw[->] (A0.east)->(A1.west) node [pos=0.5, above] {$1$} node [pos=0.5, below] {$\geq$};
\draw[->] (A1.east)->(A2.west) node [pos=0.5, above] {$2$} node [pos=0.5, below] {$\geq$};
\draw[->] (A2.east)->(A3.west) node [pos=0.5, above] {$3$} node [pos=0.5, below] {$\geq$};
\draw[->] (A4.east)->(A5.west) node [pos=0.5, above] {$1$} node [pos=0.5, below] {$\geq$};
\draw[->] (A5.10)->(A6.west) node [pos=0.7, above] {$2$} node [pos=0.8, below, sloped] {$\leq$};
\draw[->] (A5.-10)->(A7.west) node [pos=0.55, above] {$3$} node [pos=0.5, below, sloped] {$\leq$};
\end{tikzpicture}
\end{equation}
In this $B$-pattern, the minimum element is uniquely determined.
\\
($b$) Consider that 
\begin{equation}
B=\left(\begin{matrix}
0 & -3 & 18\\
3 & 0 & -7\\
-18 & 7 & 0
\end{matrix}\right),
\quad
B'=\left(\begin{matrix}
0 & 3 & -18\\
-3 & 0 & 7\\
18 & -7 & 0
\end{matrix}\right).
\end{equation}
The Markov constant is $C(B)=4$. Then, we obtain the following decreasing sequence:
\begin{equation}
\begin{tikzpicture}
\node (A0) at (0,0) {$\left(\begin{smallmatrix}
0 & -3 & 18\\
3 & 0 & -7\\
-18 & 7 & 0
\end{smallmatrix}\right)$};
\node (A1) at (3,0) {$\left(\begin{smallmatrix}
0 & 3 & -3\\
-3 & 0 & 7\\
3 & -7 & 0
\end{smallmatrix}\right)$};
\node (A2) at (6,0) [draw, rectangle] {$\left(\begin{smallmatrix}
0 & -3 & 3\\
3 & 0 & -2\\
-3 & 2 & 0
\end{smallmatrix}\right)$};
\node (B0) at (0,-2) {$\left(\begin{smallmatrix}
0 & 3 & -18\\
-3 & 0 & 7\\
18 & -7 & 0
\end{smallmatrix}\right)$};
\node (B1) at (3,-2) {$\left(\begin{smallmatrix}
0 & -3 & 3\\
3 & 0 & -7\\
-3 & 7 & 0
\end{smallmatrix}\right)$};
\node (B2) at (6,-2) [draw, rectangle] {$\left(\begin{smallmatrix}
0 & 3 & -3\\
-3 & 0 & 2\\
3 & -2 & 0
\end{smallmatrix}\right)$};
\draw[->] (A0.east)->(A1.west) node [pos=0.5, above] {$2$} node [pos=0.5, below] {$\geq$};
\draw[->] (A1.east)->(A2.west) node [pos=0.5, above] {$1$} node [pos=0.5, below] {$\geq$};
\draw[->] (B0.east)->(B1.west) node [pos=0.5, above] {$2$} node [pos=0.5, below] {$\geq$};
\draw[->] (B1.east)->(B2.west) node [pos=0.5, above] {$1$} node [pos=0.5, below] {$\geq$};
\draw[<->] (A2.south)--(B2.north) node [pos=0.5, left] {$2,3$} node [pos=0.5, right] {\rotatebox{90}{$\leq$} \rotatebox{90}{$\geq$}};
\end{tikzpicture}
\end{equation}
In this $B$-pattern, there are two possible minimum elements (up to sign) as follows
\begin{equation}
B^{\bm{\delta}(B)}=\left(\begin{matrix}
0 & -3 & 3\\
3 & 0 & -2\\
-3 & 2 & 0
\end{matrix}\right),
\quad
(B')^{\bm{\delta}(B')}=\left(\begin{matrix}
0 & 3 & -3\\
-3 & 0 & 2\\
3 & -2 & 0
\end{matrix}\right).
\end{equation}
($c$) Lastly, we consider the exchange matrix 
\begin{equation}
B=\left(\begin{matrix}
0 & -14.5 & 4.75\\
14.5 & 0 & -3.5\\
-4.75 & 3.5 & 0
\end{matrix}\right).
\end{equation}
The Markov constant is $C(B)=4$. This is an example where $\bm{\delta}(B)$ is infinite. (The classification of this case is given by \cite[Rem.~4.6]{FT19}.) For example, we obtain the following decreasing sequence $\bm{\delta}(B)=[1,2,3,\dots]$.
\begin{equation}
\begin{tikzpicture}
\node (A0) at (0,0) {$\left(\begin{smallmatrix}
0 & -14.5 & 4.75\\
14.5 & 0 & -3.5\\
-4.75 & 3.5 & 0
\end{smallmatrix}\right)$};
\node (A1) at (4,0) {$\left(\begin{smallmatrix}
0 & 2.125 & -4.75\\
-2.125 & 0 & 3.5\\
4.75 & -3.5 & 0
\end{smallmatrix}\right)$};
\node (A2) at (8.5,0) {$\left(\begin{smallmatrix}
0 & -2.125 & 2.6875\\
2.125 & 0 & -3.5\\
-2.6875 & 3.5 & 0
\end{smallmatrix}\right)$};
\draw[->] (A0.east)->(A1.west) node [pos=0.5, above] {$1$} node [pos=0.5, below] {$\geq$};
\draw[->] (A1.east)->(A2.west) node [pos=0.5, above] {$2$} node [pos=0.5, below] {$\geq$};
\draw[->] (A2.east)->(11,0) node [pos=0.5, above] {$3$} node [pos=0.5, below] {$\geq$} node [right] {$\cdots$};
\end{tikzpicture}
\end{equation}
\end{example}
As a conclusion of this subsection, we introduce the following assumption.
\begin{definition}
For a $B$-pattern $\mathbf{B}$ associated with a cluster-cyclic exchange matrix, if we take an initial exchange matrix $B$ as the minimum element in $\mathbf{B}=\mathbf{B}(B)$, we call this choice the {\em minimum assumption}.
\end{definition}

\begin{lemma}[{\cite[Lem.~7.5]{Aka24}}]\label{lem: minimum assumption for B}
Let $B \in \mathrm{M}_3(\mathbb{R})$ be a cluster-cyclic initial exchange matrix and $i\in \{1,2,3\}$ satisfying $\mu_i(B) \geq B$. For any ${\bf w},{\bf u} \geq [i]$, if ${\bf u} \leq {\bf w}$, we have $B^{\bf u} \leq B^{\bf w}$. In particular, under the minimum assumption, we have $B^{\bf u} \leq B^{\bf w}$ for any $\mathbf{u} \leq \mathbf{w}$.
\end{lemma}

\subsection{Monotonicity of $g$-vectors}
For any vector ${\bf x}=(x_1,x_2,x_3)^{\top} \in \mathbb{R}^3$, we define its absolute value vector as 
\begin{align}
\mathrm{abs}({\bf x})=(|x_1|,|x_2|,|x_3|)^{\top}.
\end{align}
Under the minimum assumption, we can show the following monotonicity of $g$-vectors.
\begin{theorem}\label{thm: monotonicity of g-vectors}
Let $B \in \mathrm{M}_3(\mathbb{R})$ be a cluster-cyclic initial exchange matrix, and let $i\in \{1,2,3\}$ satisfy $\mu_i(B) \geq B$.
\\
\textup{($a$)} For any ${\bf w},{\bf u} \geq [i]$, if ${\bf w} \leq {\bf u}$, we have
\begin{equation}\label{eq: K monotonicity}
\mathrm{abs}(\tilde{\bf g}_{K}^{\bf w}) \leq \mathrm{abs}(\tilde{\bf g}_{K}^{\bf u}). 
\end{equation}
\textup{($b$)} For any ${\bf w},{\bf u} \geq [i]$ and $j=1,2,3$, if ${\bf w} \leq {\bf u}$, we have
\begin{equation}\label{eq: j monotonicity}
\mathrm{abs}({\bf g}_{j}^{\bf w}) \leq \mathrm{abs}({\bf g}_{j}^{\bf u}),
\end{equation}
where ${\bf g}_{j}^{\bf w}$ and ${\bf g}_{j}^{\bf u}$ are the ordinary $g$-vectors, not the modified ones.
\par
In particular, under the minimum assumption, the above inequalities \eqref{eq: K monotonicity} and \eqref{eq: j monotonicity} always hold for any ${\bf w} \leq {\bf u}$ and $j=1,2,3$.
\end{theorem}
Unfortunately, the assumption $\mu_i(B) \geq B$ cannot be eliminated, see \Cref{ex: conter example of the monotonicity}.

Note that the claim ($a$) can directly imply ($b$) due to \eqref{eq: definition of modified c-, g-vectors}. The key point for the proof of \Cref{thm: monotonicity of g-vectors} is the following lemma.
\begin{lemma}\label{lem: key lemma for the monotonicity}
Let $B \in \mathrm{M}_3(\mathbb{R})$ be a cluster-cyclic initial exchange matrix. Let ${\bf w} \in \mathcal{T}$ be in a branch. Suppose that
\begin{equation}\label{eq: key lemma for the monotonicity}
\mathrm{abs}(\tilde{\bf g}_{K}^{\bf w}) \geq \mathrm{abs}(\tilde{\bf g}_{S}^{\bf w}),\mathrm{abs}(\tilde{\bf g}_{T}^{\bf w}).
\end{equation}
Then, for each $M=S,T$, we have
\begin{equation}
\mathrm{abs}(\tilde{\bf g}_{K}^{{\bf w}M}) \geq \mathrm{abs}(\tilde{\bf g}_{S}^{{\bf w}M}),\mathrm{abs}(\tilde{\bf g}_{T}^{{\bf w}M}), \mathrm{abs}(\tilde{\bf g}_{K}^{\bf w}).
\end{equation}
In particular, if (\ref{eq: key lemma for the monotonicity}) holds, then we have $\tilde{\bf g}_{K}^{\mathbf{u}'} \geq \tilde{\bf g}_{K}^{\mathbf{u}}$ for any $\mathbf{u}' \geq \mathbf{u} \geq {\bf w}$.
\end{lemma}
\begin{proof}
We show the case of $M=S$. (The case of $M=T$ is similar.) Since $\tilde{\bf g}_{S}^{{\bf w}S}=\tilde{\bf g}_{K}^{\bf w}$ and $\tilde{\bf g}_{T}^{{\bf w}S}=\tilde{\bf g}_{T}^{\bf w}$, it suffices to show $\mathrm{abs}(\tilde{\bf g}_{K}^{{\bf w}S})-\mathrm{abs}(\tilde{\bf g}_{M'}^{\bf w}) \geq {\bf 0}$ for any $M'=K,S,T$. First, we obtain
\begin{equation}\label{eq: expansion of abs g}
\mathrm{abs}(\tilde{\bf g}_{K}^{{\bf  w}S})=p_{SK}^{\bf w}\mathrm{abs}(\tilde{\bf g}_{K}^{\bf w})-\mathrm{abs}(\tilde{\bf g}_{S}^{\bf w}).
\end{equation}
Let ${\bf g}_{K}^{\bf w}=(x_1,x_2,x_3)^{\top}$ and ${\bf g}_{S}^{\bf w}=(y_1,y_2,y_3)^{\top}$. Then, by the sign-coherence of $G$-matrices, the sign of $x_i$ and $y_i$ is the same (admitting to combine $0$ and another sign). Let $\tau_{i} \in \{\pm1\}$ be the sign of $x_i$ and $y_i$. Then, we may express $\mathrm{abs}(\tilde{\bf g}_{K}^{\bf w})=(\tau_1x_1,\tau_2x_2,\tau_3x_3)^{\top}$ and $\mathrm{abs}({\bf g}_{S}^{\bf w})=(\tau_1y_1,\tau_2y_2,\tau_3y_3)^{\top}$. Thus, we have
\begin{equation}
p_{SK}^{\bf w}\mathrm{abs}(\tilde{\bf g}_{K}^{\bf w})-\mathrm{abs}(\tilde{\bf g}_{S}^{\bf w})=\left(\begin{matrix}
\tau_1(p_{SK}^{\bf w}x_1-y_1)\\
\tau_2(p_{SK}^{\bf w}x_2-y_2)\\
\tau_3(p_{SK}^{\bf w}x_3-y_3)
\end{matrix}\right).
\end{equation}
Note that the sign of $p_{SK}^{\bf w}x_i-y_i$ should be the same as $\tau_i$ because $p_{SK}^{\bf w} \geq 2$ and $|x_i| \geq |y_i|$. Thus, we obtain (\ref{eq: expansion of abs g}). By using this equality, we decompose $\mathrm{abs}(\tilde{\bf g}_{K}^{{\bf w}S})-\mathrm{abs}(\tilde{\bf g}_{M'}^{\bf w})$ into the sum of the following three terms.
\begin{equation}
(p_{SK}^{\bf w}-2)\mathrm{abs}(\tilde{\bf g}_{K}^{\bf w})+(\mathrm{abs}(\tilde{\bf g}_{K}^{\bf w})-\mathrm{abs}(\tilde{\bf g}_{S}^{\bf w}))+(\mathrm{abs}(\tilde{\bf g}_{K}^{\bf w})-\mathrm{abs}(\tilde{\bf g}_{M'}^{\bf w})).
\end{equation}
We can show that each term is larger or equal to ${\bf 0}$ by using $p_{SK}^{\bf w} \geq 2$ and $\mathrm{abs}(\tilde{\bf g}_{K}^{\bf w}) \geq \mathrm{abs}(\tilde{\bf g}_{S}^{\bf w}), \mathrm{abs}(\tilde{\bf g}_{T}^{\bf w})$. Thus, we obtain the claim.
\end{proof}
Thanks to this claim, once the inequality \eqref{eq: key lemma for the monotonicity} holds, the monotonicity \eqref{eq: K monotonicity} always holds after that. Unfortunately, this condition does not always hold, so we need to consider this more carefully to prove \Cref{thm: monotonicity of g-vectors}. The strategy of the following proof is to prove \eqref{eq: key lemma for the monotonicity} when ${\bf w}=[i]S^nT$ ($n \geq 1$) and ${\bf w}=[i]TS^nT$ ($n \geq 0$), and to check the monotonicity for the other cases, ${\bf w}=[i]S^n$ and ${\bf w}=[i]TS^n$.

\begin{proof}[Proof of Theorem~\ref{thm: monotonicity of g-vectors}]
By \Cref{lem: key lemma for the monotonicity}, it suffices to show the following four conditions. (In the proof of \eqref{item: second case for the monotonicity} and \eqref{item: fourth case for the monotonicity}, the assumption $\mu_i(B) \geq B$ is needed.)
\begin{enumerate}
\item $\mathrm{abs}(\tilde{\bf g}_{K}^{[i]S^n}) \leq \mathrm{abs}(\tilde{\bf g}_{K}^{[i]S^m})$ for $0 \leq n \leq m$.
\label{item: first case for the monotonicity}
\item Let ${\bf w}=[i]S^nT$ with $n=1,2,\dots$. Then, \eqref{eq: key lemma for the monotonicity} holds.
\label{item: second case for the monotonicity}
\item $\mathrm{abs}(\tilde{\bf g}_{K}^{[i]TS^n}) \leq \mathrm{abs}(\tilde{\bf g}_{K}^{[i]TS^m})$ for $0 \leq n \leq m$.
\label{item: third case for the monotonicity}
\item Let ${\bf w}=[i]TS^nT$ with $n=0,1,2,\dots$. Then, \eqref{eq: key lemma for the monotonicity} holds.
\label{item: fourth case for the monotonicity}
\end{enumerate}
Due to \Cref{lem: key lemma for the monotonicity}, for each $\mathbf{w}=[i]S^nT$ and $\mathbf{w}=[i]TS^nT$, \eqref{item: second case for the monotonicity} and \eqref{item: fourth case for the monotonicity} imply the monotonicity after that. Moreover, since $\tilde{\bf g}_{K}^{[i]S^n}=\tilde{\bf g}_{T}^{[i]S^nT}$ and $\tilde{\bf g}_{K}^{[i]TS^n}=\tilde{\bf g}_{S}^{[i]TS^nT}$ by \eqref{eq: T mutation of c- g-vectors in trunks} and \eqref{eq: T mutation of c- g-vectors in branches}, the inequalities $\mathrm{abs}(\tilde{\bf g}_{K}^{[i]S^n}) \leq \mathrm{abs}(\tilde{\bf g}_{K}^{[i]S^nT})$ and $\mathrm{abs}(\tilde{\bf g}_{K}^{[i]TS^n}) \leq \mathrm{abs}(\tilde{\bf g}_{K}^{[i]TS^nT})$ can be shown by \eqref{item: second case for the monotonicity} and \eqref{item: fourth case for the monotonicity}.
\\
\eqref{item: first case for the monotonicity}
We can check it by Lemma~\ref{lem: infinite S mutations}. (Note that $u_{n+1}(p) \geq u_{n}(p)$.)
\\
\eqref{item: second case for the monotonicity}
By \eqref{eq: g vectors at the root of maximal branches}, it suffices to show
\begin{equation}
p_{ST}^{[i]S^n}u_{n-1}(p_{k_0s_0}) \geq u_{n}(p_{k_0s_0})
\end{equation}
for any $n \in \mathbb{Z}_{\geq 1}$. Since $u_{n-2}(p_{k_0s_{0}}) \geq u_{-1}(p_{k_0s_0})=0$, we have $u_{n}(p_{k_0s_0})=p_{k_0s_0}u_{n-1}(p_{k_0s_0})-u_{n-2}(p_{k_0s_0}) \leq p_{k_0s_0}u_{n-1}(p_{k_0s_0})$. Moreover, according to Lemma~\ref{lem: monotonicity of B}, the largest element in $\{p_{12}^{[i]S^n},p_{23}^{[i]S^n},p_{31}^{[i]S^n}\}$ should be expressed as $p_{ST}^{[i]S^n}$. In particular, we obtain $p_{ST}^{[i]S^n} \geq p_{k_0s_0}^{[i]S^n}$. By \Cref{lem: minimum assumption for B}, we have $p_{k_0s_0}^{[i]S^n} \geq p_{k_0s_0}$. Thus, we have
$u_{n}(p_{k_0s_0}) \leq p_{k_0s_0}u_{n-1}(p_{k_0s_0}) \leq p_{ST}^{[i]S^n}u_{n-1}(p_{k_0s_0})$ as we desired.
\\
\eqref{item: third case for the monotonicity} This can be verified by \eqref{eq: iTS^n expressions}.
\\
\eqref{item: fourth case for the monotonicity}
Set $p'=p_{SK}^{[i]T}$. By \Cref{lem: T-mutations of modified c- g-vectores} and \eqref{eq: iTS^n expressions}, 
we have
\begin{equation}
\begin{aligned}\label{eq: root of TSnT}
\tilde{\bf g}_{S}^{[i]TS^nT}&=\tilde{\bf g}_{K}^{[i]TS^n}=u_{n+1}(p')\tilde{\bf e}_{s_0}-u_n(p')\tilde{\bf e}_{t_0},
\\
\tilde{\bf g}_{T}^{[i]TS^nT}&=\tilde{\bf g}_{S}^{[i]TS^n}=u_n(p')\tilde{\bf e}_{s_0}-u_{n-1}(p')\tilde{\bf e}_{t_0},\\
\tilde{\bf g}_{K}^{[i]TS^nT}&=\tilde{\bf e}_{k_0}+(p_{KT}^{[i]TS^n}u_{n+1}(p')-p_{k_0s_0})\tilde{\bf e}_{s_0}-p_{KT}^{[i]TS^n}u_{n}(p')\tilde{\bf e}_{t_0}.
\end{aligned}
\end{equation}
Thus, it suffices to show the following inequality for any $n \in \mathbb{Z}_{\geq 0}$:
\begin{equation}
p_{KT}^{[i]TS^n}u_{n+1}(p') -p_{k_0s_0} \geq u_{n+1}(p'). 
\end{equation}
We aim to show
$(p_{KT}^{[i]TS^n}-1)u_{n+1}(p') \geq p_{k_0s_0}$.
By $p_{KT}^{[i]TS^n} \geq 2$, we have
\begin{equation}
(p_{KT}^{[i]TS^n}-1)u_{n+1}(p') \geq u_{n+1}(p') \geq u_{1}(p')= p'.
\end{equation}
By \eqref{eq: T mutation of B in trunks}, we have
\begin{equation}
p'=p_{SK}^{[i]T}=p_{ST}^{[i]}=p_{s_0t_0}^{[i]}.
\end{equation}
Note that $i \neq s_0$ and $i \neq t_0$. Thus, by \Cref{lem: monotonicity of B}, the largest number in $\{p_{k_0s_0}^{[i]}, p_{s_0t_0}^{[i]}, p_{t_0k_0}^{[i]}\}$ is $p'=p_{s_0t_0}^{[i]}$. Namely, we have
\begin{equation}
p' \geq p_{k_0s_0}^{[i]}\overset{\eqref{eq: initial mutations}}{=}p_{k_0s_0}.
\end{equation}
Thus, we obtain $(p_{KT}^{[i]TS^n}-1)u_{n+1}(p')  \geq p_{k_0s_0}$ as desired.
This completes the proof.
\end{proof}
\begin{remark}
Due to the above proof, we can replace the assumption with a more restrictive one. Namely, if the conditions \eqref{item: second case for the monotonicity} and \eqref{item: fourth case for the monotonicity} are satisfied, that is, if
\begin{equation}\label{eq: sufficient condition for the monotonicity}
p_{ST}^{[i]S^{n+1}}u_{n}(p_{k_0s_0}) \geq u_{n+1}(p_{k_0s_0}),
\quad
p_{KT}^{[i]TS^n}u_{n+1}(p_{SK}^{[i]T}) -p_{k_0s_0} \geq u_{n+1}(p_{SK}^{[i]T})
\end{equation}
for any $n \in \mathbb{Z}_{\geq 0}$, then the monotonicity holds in $\mathcal{T}^{\geq [i]}$.
\end{remark}
\begin{example}\label{ex: conter example of the monotonicity}
We give one counterexample when we lose the assumption $\mu_{i}(B) \geq B$. Let an initial exchange matrix be the one in \eqref{eq: non minimum example}.
By a direct calculation, we may obtain the following $G$-matrices associated with the mutation sequence $\mfw=[1,2,1]$.
\begin{equation}
\left(\begin{smallmatrix}
1 & 0 & 0\\
0 & 1 & 0\\
0 & 0 & 1
\end{smallmatrix}\right)
\overset{1}{\mapsto}
\left(\begin{smallmatrix}
-1 & 0 & 0\\
0 & 1 & 0\\
1795 & 0 & 1
\end{smallmatrix}\right)
\overset{2}{\mapsto}
\left(\begin{smallmatrix}
-1 & 0 & 0\\
0 & -1 & 0\\
1795 & 8 & 1
\end{smallmatrix}\right)
\overset{1}{\mapsto}
\left(\begin{smallmatrix}
1 & 0 & 0\\
-228 & -1 & 0\\
29 & 8 & 1
\end{smallmatrix}\right).
\end{equation}
In particular, the $(3,1)$th entry decreases by the last mutation. Thus, in this case, the monotonicity does not hold in general. (If we restrict the mutations in $\mathcal{T}^{\geq [2]}$ and $\mathcal{T}^{\geq [3]}$, the monotonicity holds.)
\par
On the other hand, even if the condition $\mu_i(B) \geq B$ does not hold, the monotonicity sometimes still holds. For example, let an initial exchange matrix be
\begin{equation}
B=\left(\begin{matrix}
0 & -4 & 3\\
4 & 0 & -8\\
-3 & 8 & 0
\end{matrix}\right).
\end{equation}
It holds that 
\begin{equation}
\mu_1(B)=\left(\begin{smallmatrix}
0 & 4 & -3\\
-4 & 0 & 4\\
3 & -4 & 0
\end{smallmatrix}\right) \not\geq B.
\end{equation}
However, by a direct calculation, we may check that the inequalities in \eqref{eq: sufficient condition for the monotonicity} hold, which implies that the monotonicity holds in this case.
\end{example}

\subsection{Simplification of global upper bounds}
In \Cref{sec: global}, we introduced the global upper bound, which is an upper bound depending only on the initial mutation $i=1,2,3$. In this subsection, we simplify this upper bound under the minimum assumption.
\par
Let $Q_{\mathrm{initial}}$, $Q^{+}_{\mathrm{initial}}$ and $Q^{-}_{\mathrm{initial}}$ be the sets defined by $Q$, $Q^{+}$ and $Q^{-}$ in \eqref{eq: set Q} with respect to the initial exchange matrix $B$ together with $\tilde{A}$.
\begin{theorem}\label{thm: simplified global upper bound}
Let $B \in \mathrm{M}_3(\mathbb{R})$ be a cluster-cyclic initial exchange matrix, and let $i\in \{1,2,3\}$ satisfy $\mu_i(B) \geq B$. Then, we have
\begin{equation}
|\Delta^{\geq [i]}(B)| \subset Q^{+}_{\mathrm{initial}}.
\end{equation}
In particular, under the minimum assumption, we have
\begin{equation}
|\Delta(B)| \subset Q^{+}_{\mathrm{initial}}.
\end{equation}
\end{theorem}
Fix an initial mutation $i=1,2,3$, and set $k_0$, $s_0$, and $t_0$ as in \Cref{tab: List of initial indices}.
Note that by \Cref{thm: global upper bound theorem}, \eqref{eq: sign s0}, and \eqref{eq: sign t0}, all $g$-vectors except $\mathbf{e}_{t_0}$ belong to $Q^{+}_{i} \cap \positiveclosure{\mathbf{e}_{s_0}} \cap \negativeclosure{\mathbf{e}_{t_0}}$; see \Cref{rem: one excptional sign}.
Before proceeding with the proof, we establish the following lemma.
\begin{lemma}\label{lem: lemma for the simplification of global upper bounds}
Fix $i=1,2,3$ satisfying $\mu_i(B) \geq B$. Then, we have
\begin{equation}\label{eq: lemma for the simplification of global upper bounds}
Q^{+}_{i} \cap \positiveclosure{\mathbf{e}_{s_0}} \cap \negativeclosure{\mathbf{e}_{t_0}} \subset Q^{+}_{\mathrm{initial}},
\quad
Q^{-}_{i} \cap \positiveclosure{\mathbf{e}_{s_0}} \cap \negativeclosure{\mathbf{e}_{t_0}} \subset Q^{-}_{\mathrm{initial}}.
\end{equation}
\end{lemma}
\begin{proof}
Firstly, we prove $Q_{i} \cap \positiveclosure{\mathbf{e}_{s_0}} \cap \negativeclosure{\mathbf{e}_{t_0}} \subset Q_{\mathrm{initial}}$. Let $\mathbf{x}=(x_1,x_2,x_3) \in Q_{i} \cap \positiveclosure{\mathbf{e}_{s_0}} \cap \negativeclosure{\mathbf{e}_{t_0}}$. If $\mathbf{x}=\mathbf{0}$, we have $\mathbf{0} \in Q_{\mathrm{initial}}$. Assume that $\mathbf{x}\neq \mathbf{0}$. Then, we have $x_{s_0} \geq 0$, $x_{t_0} \leq 0$, and
\begin{equation}
\begin{aligned}
&\ d_{k_0}x_{k_0}^2+d_{s_0}x_{s_0}^2+d_{t_0}x_{t_0}^2+\sqrt{d_{k_0}d_{s_0}}p_{k_0s_0}^{[i]}x_{k_0}x_{s_0}+\sqrt{d_{k_0}d_{t_0}}p_{k_0t_0}^{[i]}x_{k_0}x_{t_0}+\sqrt{d_{s_0}d_{t_0}}p_{s_0t_0}^{[i]}x_{s_0}x_{t_0}
\\
>&\ 0.
\end{aligned}
\end{equation}
By \eqref{eq: initial mutations}, we have $p_{k_0s_0}^{[i]}=p_{k_0s_0}$ and $p_{k_0t_0}^{[i]}=p_{k_0s_0}$. By the assumption $\mu_i(B) \geq B$, we have $p_{s_0t_0}^{[i]} \geq p_{s_0t_0}$. Since $x_{s_0} \geq 0$ and $x_{t_0} \leq 0$, we have $p_{s_0t_0}^{[i]}x_{s_0}x_{t_0} \leq p_{s_0t_0}x_{s_0}x_{t_0}$, and it implies
\begin{equation}
\begin{aligned}
&\ d_{k_0}x_{k_0}^2+d_{s_0}x_{s_0}^2+d_{t_0}x_{t_0}^2+\sqrt{d_{k_0}d_{s_0}}p_{k_0s_0}x_{k_0}x_{s_0}+\sqrt{d_{k_0}d_{t_0}}p_{k_0t_0}x_{k_0}x_{t_0}+\sqrt{d_{s_0}d_{t_0}}p_{s_0t_0}x_{s_0}x_{t_0}
\\
\geq&\  d_{k_0}x_{k_0}^2+d_{s_0}x_{s_0}^2+d_{t_0}x_{t_0}^2+\sqrt{d_{k_0}d_{s_0}}p_{k_0s_0}^{[i]}x_{k_0}x_{s_0}+\sqrt{d_{k_0}d_{t_0}}p_{k_0t_0}^{[i]}x_{k_0}x_{t_0}+\sqrt{d_{s_0}d_{t_0}}p_{s_0t_0}^{[i]}x_{s_0}x_{t_0}
\\
>&\ 0.
\end{aligned}
\end{equation}
This implies $\mathbf{x} \in Q_{\mathrm{initial}}$.
\par
Now, we assume that the inclusions \eqref{eq: lemma for the simplification of global upper bounds} do not hold. By the symmetry, we might assume that $Q_{i}^{+} \cap \positiveclosure{\mathbf{e}_{s_0}} \cap \negativeclosure{\mathbf{e}_{t_0}} \nsubset Q_{\mathrm{initial}}^{+}$ without loss of generality. Since $Q_{\mathrm{initial}}$ is decomposed into $Q_{\mathrm{initial}}^{+}$ and $Q_{\mathrm{initial}}^{-}\setminus\{\mathbf{0}\}$, the set $Q_{i}^{+} \cap \positiveclosure{\mathbf{e}_{s_0}} \cap \negativeclosure{\mathbf{e}_{t_0}}$ intersects with $Q_{\mathrm{initial}}^{-}\setminus\{\mathbf{0}\}$. Note that $Q_{\mathrm{initial}}^{+} \setminus \{\mathbf{0}\}$ also intersects with $Q_{i}^{+} \cap \positiveclosure{\mathbf{e}_{s_0}} \cap \negativeclosure{\mathbf{e}_{t_0}}$. For example, $\mathbf{e}_{s_0}$ is an intersection of these two sets. Take an element $\mathbf{y} \in Q_{i}^{+} \cap \positiveclosure{\mathbf{e}_{s_0}} \cap \negativeclosure{\mathbf{e}_{t_0}}$ so that $\mathbf{y} \in Q^{-}_{\mathrm{initial}} \setminus\{\mathbf{0}\}$.
Then, since $Q_{\mathrm{initial}}^{+} \setminus \{\mathbf{0}\}$ and $Q_{\mathrm{initial}}^{-} \setminus\{\mathbf{0}\}$ are disconnected by \eqref{eq: separateness for Q}, there exist positive numbers $a,b \in \mathbb{R}_{>0}$ such that $a\mathbf{e}_{s_0}+b\mathbf{y} \notin Q_{\mathrm{initial}}$.
(Note that $a\mathbf{e}_{s_0}+b\mathbf{y} \neq \mathbf{0}$ holds because $\mathbf{y} \in \positiveclosure{\mathbf{e}_{s_0}} \setminus \{\mathbf{0}\}$.) On the other hand, since $Q_{i}^{+} \cap \positiveclosure{\mathbf{e}_{s_0}} \cap \negativeclosure{\mathbf{e}_{t_0}}$ is a convex cone by \Cref{lem: convexity of Q}, we have $a\mathbf{e}_{s_0}+b\mathbf{y} \in Q_{i}^{+} \cap \positiveclosure{\mathbf{e}_{s_0}} \cap \negativeclosure{\mathbf{e}_{t_0}}$. This contradicts with the inclusion $Q_{i} \cap \positiveclosure{\mathbf{e}_{s_0}} \cap \negativeclosure{\mathbf{e}_{t_0}} \subset Q_{\mathrm{initial}}$. Hence, we complete the proof.
\end{proof}
\begin{proof}[Proof of \Cref{thm: simplified global upper bound}]
Consider a $g$-vector $\mathbf{g}_{j}^{\mathbf{w}}$ with $\mathbf{w} \geq [i]$. If this $g$-vector is $\mathbf{e}_{t_0}$, the claim is shown directly. For another $g$-vector, by \Cref{thm: global upper bound theorem}, \eqref{eq: sign s0} and \eqref{eq: sign t0}, it belongs to $Q^{+}_{i} \cap \positiveclosure{\mathbf{e}_{s_0}} \cap \negativeclosure{\mathbf{e}_{t_0}}$; see \Cref{rem: one excptional sign}. By \Cref{lem: lemma for the simplification of global upper bounds}, we obtain the claim.
\end{proof}
\begin{remark}
Note that for the exchange matrix corresponding to the Markov quiver, it is minimum. Hence, we can give the upper bound described in \Cref{thm: simplified global upper bound}, which is also the same as the one given in \Cref{thm: global upper bound theorem}. 
\end{remark}

\subsection*{Acknowledgements} 
 The authors would like to express their sincere gratitude to Tomoki Nakanishi for his thoughtful guidance. The authors also wish to thank Peigen Cao, Yasuaki Gyoda, Salvatore Stella and Toshiya Yurikusa for their valuable discussions and insightful suggestions. In addition, the second author wants to thank Xiaowu Chen, Zhe Sun and Yu Ye for their help and support. The first author is supported by JSPS KAKENHI Grant Number JP25KJ1438 and Chubei Itoh Foundation.
 The second author is supported by National Natural Science Foundation of China (Grant No. 124B2003) and China Scholarship Council (Grant No. 202406340022).

\newpage
\bibliography{Akagi_Chen_25_paper3}

@article{AC25a,
  title={Real {$C$}-, {$G$}-structures and sign-coherence of cluster algebras},
  author={Akagi, R. and Chen, Z.},
  journal={arXiv preprint arXiv:2509.06486, ver.~3},
  year={2025}
}

@article{AC25b,
    AUTHOR = {Akagi, R. and Chen, Z.},
     TITLE = {Sign-coherence and tropical sign pattern for rank 3 real
              cluster-cyclic exchange matrices},
   JOURNAL = {Int. Math. Res. Not. IMRN},
  FJOURNAL = {International Mathematics Research Notices. IMRN},
      YEAR = {2026},
    NUMBER = {22},
     PAGES = {Paper No. rnag262},
      ISSN = {1073-7928,1687-0247},
   MRCLASS = {99-06},
  MRNUMBER = {5155602},
       DOI = {10.1093/imrn/rnag262},
       URL = {https://doi.org/10.1093/imrn/rnag262},
}

@article{Aka24,
  title={Cluster-cyclic condition of skew-symmetrizable matrices of rank 3 via the Markov constant},
  author={R. Akagi},
  journal={arXiv preprint arXiv:2411.07083, ver.~2},
  year={2024}
}

@article {BBH11,
    AUTHOR = {Beineke, A. and Br\"ustle, T. and Hille, L.},
     TITLE = {Cluster-cyclic quivers with three vertices and the {M}arkov
              equation},
      NOTE = {With an appendix by Otto Kerner},
   JOURNAL = {Algebr. Represent. Theory},
  FJOURNAL = {Algebras and Representation Theory},
    VOLUME = {14},
      YEAR = {2011},
    NUMBER = {1},
     PAGES = {97--112},
      ISSN = {1386-923X,1572-9079},
   MRCLASS = {16G20 (13F60)},
  MRNUMBER = {2763295},
MRREVIEWER = {Gregoire\ Dupont},
       DOI = {10.1007/s10468-009-9179-9},
       URL = {https://doi.org/10.1007/s10468-009-9179-9},
}

@article {Mul16,
    AUTHOR = {Muller, G.},
     TITLE = {The existence of a maximal green sequence is not invariant
              under quiver mutation},
   JOURNAL = {Electron. J. Combin.},
  FJOURNAL = {Electronic Journal of Combinatorics},
    VOLUME = {23},
      YEAR = {2016},
    NUMBER = {2},
     PAGES = {Paper 2.47, 23},
      ISSN = {1077-8926},
   MRCLASS = {13F60},
  MRNUMBER = {3512669},
MRREVIEWER = {Fan\ Qin},
       DOI = {10.37236/5412},
       URL = {https://doi.org/10.37236/5412},
}

@article {GHKK18,
    AUTHOR = {Gross, M. and Hacking, P. and Keel, S. and Kontsevich, M.},
     TITLE = {Canonical bases for cluster algebras},
   JOURNAL = {J. Amer. Math. Soc.},
  FJOURNAL = {Journal of the American Mathematical Society},
    VOLUME = {31},
      YEAR = {2018},
    NUMBER = {2},
     PAGES = {497--608},
      ISSN = {0894-0347,1088-6834},
   MRCLASS = {13F60 (14J33)},
  MRNUMBER = {3758151},
MRREVIEWER = {Ralf\ Schiffler},
       DOI = {10.1090/jams/890},
       URL = {https://doi.org/10.1090/jams/890},
}

@article {FZ03,
    AUTHOR = {Fomin, S. and Zelevinsky, A.},
     TITLE = {Cluster algebras. {II}. {F}inite type classification},
   JOURNAL = {Invent. Math.},
  FJOURNAL = {Inventiones Mathematicae},
    VOLUME = {154},
      YEAR = {2003},
    NUMBER = {1},
     PAGES = {63--121},
      ISSN = {0020-9910,1432-1297},
   MRCLASS = {17B20 (05E15 16S99 52B12)},
  MRNUMBER = {2004457},
MRREVIEWER = {Eric\ N.\ Sommers},
       DOI = {10.1007/s00222-003-0302-y},
       URL = {https://doi.org/10.1007/s00222-003-0302-y},
}

@article {FZ07,
    AUTHOR = {Fomin, S. and Zelevinsky, A.},
     TITLE = {Cluster algebras. {IV}. {C}oefficients},
   JOURNAL = {Compos. Math.},
  FJOURNAL = {Compositio Mathematica},
    VOLUME = {143},
      YEAR = {2007},
    NUMBER = {1},
     PAGES = {112--164},
      ISSN = {0010-437X,1570-5846},
   MRCLASS = {16S99 (05E15 14M17 22E46)},
  MRNUMBER = {2295199},
MRREVIEWER = {Christof\ Gei\ss},
       DOI = {10.1112/S0010437X06002521},
       URL = {https://doi.org/10.1112/S0010437X06002521},
}

@article{LL24,
  title={An unexpected property of ${\bf g}$-vectors for rank 3 mutation-cyclic quivers},
  author={Lee, J. and Lee, K.},
  journal={arXiv preprint arXiv:2409.00599, ver.~1},
  year={2024}
}

@book {Nak23,
    AUTHOR = {Nakanishi, T.},
     TITLE = {Cluster algebras and scattering diagrams},
    SERIES = {MSJ Memoirs},
    VOLUME = {41},
 PUBLISHER = {Mathematical Society of Japan, Tokyo},
      YEAR = {2023},
     PAGES = {xiv+279},
      ISBN = {978-4-86497-105-8},
   MRCLASS = {13F60},
  MRNUMBER = {4563311},
MRREVIEWER = {Ibrahim\ Saleh},
       DOI = {10.1142/e073},
       URL = {https://doi.org/10.1142/e073},
}

@article{Nak26,
    author = {T. Nakanishi},
    title = {Local and Global Patterns of Rank 3 {$G$}-Fans of Totally-Infinite Type},
    journal = {Experimental Mathematics},
    volume = {0},
    number = {0},
    pages = {1--27},
    year = {2026},
    publisher = {Taylor \& Francis},
    doi = {10.1080/10586458.2026.2736665},
    URL = {
            https://doi.org/10.1080/10586458.2026.2736665
    },
    eprint = {
            https://doi.org/10.1080/10586458.2026.2736665
    }
}

@incollection{NZ12,
    AUTHOR = {Nakanishi, T. and Zelevinsky, A.},
     TITLE = {On tropical dualities in cluster algebras},
 BOOKTITLE = {Algebraic groups and quantum groups},
    SERIES = {Contemp. Math.},
    VOLUME = {565},
     PAGES = {217--226},
 PUBLISHER = {Amer. Math. Soc., Providence, RI},
      YEAR = {2012},
      ISBN = {978-0-8218-5317-7},
   MRCLASS = {13F60},
  MRNUMBER = {2932428},
MRREVIEWER = {Yu\ Zhou},
       DOI = {10.1090/conm/565/11159},
       URL = {https://doi.org/10.1090/conm/565/11159},
}

@article {Rea14,
    AUTHOR = {Reading, N.},
     TITLE = {Universal geometric cluster algebras},
   JOURNAL = {Math. Z.},
  FJOURNAL = {Mathematische Zeitschrift},
    VOLUME = {277},
      YEAR = {2014},
    NUMBER = {1-2},
     PAGES = {499--547},
      ISSN = {0025-5874,1432-1823},
   MRCLASS = {13F60 (05E15 20F55 52B12)},
  MRNUMBER = {3205782},
MRREVIEWER = {Xueqing\ Chen},
       DOI = {10.1007/s00209-013-1264-4},
       URL = {https://doi.org/10.1007/s00209-013-1264-4},
}

@incollection {Sev12,
    AUTHOR = {Seven, A. I.},
     TITLE = {Mutation classes of {$3\times3$} generalized {C}artan matrices},
 BOOKTITLE = {Highlights in {L}ie algebraic methods},
    SERIES = {Progr. Math.},
    VOLUME = {295},
     PAGES = {205--211},
 PUBLISHER = {Birkh\"auser/Springer, New York},
      YEAR = {2012},
      ISBN = {978-0-8176-8273-6},
   MRCLASS = {05E15 (05C50 05E10 13F60 17B67)},
  MRNUMBER = {2866853},
       DOI = {10.1007/978-0-8176-8274-3\_9},
       URL = {https://doi.org/10.1007/978-0-8176-8274-3_9},
}

@article {FG16,
    AUTHOR = {Fock, V. V. and Goncharov, A. B.},
     TITLE = {Cluster {P}oisson varieties at infinity},
   JOURNAL = {Selecta Math. (N.S.)},
  FJOURNAL = {Selecta Mathematica. New Series},
    VOLUME = {22},
      YEAR = {2016},
    NUMBER = {4},
     PAGES = {2569--2589},
      ISSN = {1022-1824,1420-9020},
   MRCLASS = {14T05 (13F60 30F60)},
  MRNUMBER = {3573965},
MRREVIEWER = {Ralf\ Schiffler},
       DOI = {10.1007/s00029-016-0282-6},
       URL = {https://doi.org/10.1007/s00029-016-0282-6},
}

@article{Cha12,
     AUTHOR = {Ch{\'a}vez, A. N.},
     TITLE = {On the c-vectors and g-vectors of the {M}arkov cluster
              algebra},
   JOURNAL = {S\'em. Lothar. Combin.},
  FJOURNAL = {S\'eminaire Lotharingien de Combinatoire},
    VOLUME = {69},
      YEAR = {2012},
     PAGES = {Art. B69d, 12},
      ISSN = {1286-4889},
   MRCLASS = {13F60},
  MRNUMBER = {3118908},
MRREVIEWER = {Li\ Li},
       DOI = {10.1029/jz069i001p00001},
       URL = {https://doi.org/10.1029/jz069i001p00001},
}

@article {FZ02,
    AUTHOR = {Fomin, S. and Zelevinsky, A.},
     TITLE = {Cluster algebras. {I}. {F}oundations},
   JOURNAL = {J. Amer. Math. Soc.},
  FJOURNAL = {Journal of the American Mathematical Society},
    VOLUME = {15},
      YEAR = {2002},
    NUMBER = {2},
     PAGES = {497--529},
      ISSN = {0894-0347,1088-6834},
   MRCLASS = {16S99 (14M99 17B99)},
  MRNUMBER = {1887642},
MRREVIEWER = {Eric\ N.\ Sommers},
       DOI = {10.1090/S0894-0347-01-00385-X},
       URL = {https://doi.org/10.1090/S0894-0347-01-00385-X},
}

@book{Zwi95,
  title={CRC Standard Mathematical Tables and Formulae},
  author={Zwillinger, D.},
  year={1995},
  publisher={CRC press},
  edition={30th},
}

@article {FT19,
    AUTHOR = {Felikson, A. and Tumarkin, P.},
     TITLE = {Geometry of mutation classes of rank 3 quivers},
   JOURNAL = {Arnold Math. J.},
  FJOURNAL = {Arnold Mathematical Journal},
    VOLUME = {5},
      YEAR = {2019},
    NUMBER = {1},
     PAGES = {37--55},
      ISSN = {2199-6792,2199-6806},
   MRCLASS = {13F60 (20H15 51F15)},
  MRNUMBER = {3981452},
MRREVIEWER = {Calin\ Chindris},
       DOI = {10.1007/s40598-019-00101-2},
       URL = {https://doi.org/10.1007/s40598-019-00101-2},
}

@article {Rea20b,
    AUTHOR = {Reading, N.},
     TITLE = {Scattering fans},
   JOURNAL = {Int. Math. Res. Not. IMRN},
  FJOURNAL = {International Mathematics Research Notices. IMRN},
      YEAR = {2020},
    NUMBER = {23},
     PAGES = {9640--9673},
      ISSN = {1073-7928,1687-0247},
   MRCLASS = {13F60 (05E14 14J33)},
  MRNUMBER = {4182806},
MRREVIEWER = {Fan\ Qin},
       DOI = {10.1093/imrn/rny260},
       URL = {https://doi.org/10.1093/imrn/rny260},
}

@article {RS16,
    AUTHOR = {Reading, N. and Speyer, D. E.},
     TITLE = {Combinatorial frameworks for cluster algebras},
   JOURNAL = {Int. Math. Res. Not. IMRN},
  FJOURNAL = {International Mathematics Research Notices. IMRN},
      YEAR = {2016},
    NUMBER = {1},
     PAGES = {109--173},
      ISSN = {1073-7928,1687-0247},
   MRCLASS = {13F60 (05C25)},
  MRNUMBER = {3514060},
MRREVIEWER = {Gustavo\ Jasso},
       DOI = {10.1093/imrn/rnv101},
       URL = {https://doi.org/10.1093/imrn/rnv101},
}

@article {RS18,
    AUTHOR = {Reading, N. and Speyer, D. E.},
     TITLE = {Cambrian frameworks for cluster algebras of affine type},
   JOURNAL = {Trans. Amer. Math. Soc.},
  FJOURNAL = {Transactions of the American Mathematical Society},
    VOLUME = {370},
      YEAR = {2018},
    NUMBER = {2},
     PAGES = {1429--1468},
      ISSN = {0002-9947,1088-6850},
   MRCLASS = {13F60 (20F55)},
  MRNUMBER = {3729507},
MRREVIEWER = {Ralf\ Schiffler},
       DOI = {10.1090/tran/7193},
       URL = {https://doi.org/10.1090/tran/7193},
}

@article {Rea20a,
    AUTHOR = {Reading, N.},
     TITLE = {A combinatorial approach to scattering diagrams},
   JOURNAL = {Algebr. Comb.},
  FJOURNAL = {Algebraic Combinatorics},
    VOLUME = {3},
      YEAR = {2020},
    NUMBER = {3},
     PAGES = {603--636},
      ISSN = {2589-5486},
   MRCLASS = {13F60},
  MRNUMBER = {4113600},
MRREVIEWER = {Ibrahim\ Saleh},
       DOI = {10.5802/alco.107},
       URL = {https://doi.org/10.5802/alco.107},
}

@article {Yur20,
    AUTHOR = {Yurikusa, T.},
     TITLE = {Density of {$g$}-vector cones from triangulated surfaces},
   JOURNAL = {Int. Math. Res. Not. IMRN},
  FJOURNAL = {International Mathematics Research Notices. IMRN},
      YEAR = {2020},
    NUMBER = {21},
     PAGES = {8081--8119},
      ISSN = {1073-7928,1687-0247},
   MRCLASS = {16G20 (05E16 13F60 57K20)},
  MRNUMBER = {4176846},
MRREVIEWER = {Christof\ Gei\ss},
       DOI = {10.1093/imrn/rnaa008},
       URL = {https://doi.org/10.1093/imrn/rnaa008},
}

@article{Yur25,
  title={Finite type and completeness of $g$-fans},
  author={Yurikusa, T.},
  journal={arXiv preprint arXiv:2512.21603, ver.~1},
  year={2025}
}

@article {RS22,
    AUTHOR = {Reading, Nathan and Stella, Salvatore},
     TITLE = {Cluster scattering diagrams of acyclic affine type},
   JOURNAL = {C. R. Math. Acad. Sci. Paris},
  FJOURNAL = {Comptes Rendus Math\'ematique. Acad\'emie des Sciences. Paris},
    VOLUME = {364},
      YEAR = {2026},
     PAGES = {701--742},
      ISSN = {1631-073X,1778-3569},
   MRCLASS = {99-06},
  MRNUMBER = {5151308},
       DOI = {10.5802/crmath.850},
       URL = {https://doi.org/10.5802/crmath.850},
}

@article {Qin24,
    AUTHOR = {Qin, F.},
     TITLE = {Bases for upper cluster algebras and tropical points},
   JOURNAL = {J. Eur. Math. Soc. (JEMS)},
  FJOURNAL = {Journal of the European Mathematical Society (JEMS)},
    VOLUME = {26},
      YEAR = {2024},
    NUMBER = {4},
     PAGES = {1255--1312},
      ISSN = {1435-9855,1435-9863},
   MRCLASS = {13F60},
  MRNUMBER = {4721032},
MRREVIEWER = {Juan\ Bosco\ Fr\'ias-Medina},
       DOI = {10.4171/jems/1308},
       URL = {https://doi.org/10.4171/jems/1308},
}

@article {RS23,
    AUTHOR = {Rupel, D. and Stella, S.},
     TITLE = {Dominance regions for rank two cluster algebras},
   JOURNAL = {Ann. Comb.},
  FJOURNAL = {Annals of Combinatorics},
    VOLUME = {27},
      YEAR = {2023},
    NUMBER = {4},
     PAGES = {873--894},
      ISSN = {0218-0006,0219-3094},
   MRCLASS = {13F60},
  MRNUMBER = {4657334},
       DOI = {10.1007/s00026-023-00636-4},
       URL = {https://doi.org/10.1007/s00026-023-00636-4},
}

@article{RRS25,
  title={Dominance regions for affine cluster algebras},
  author={Reading, N. and Rupel, D. and Stella, S.},
  journal={arXiv preprint arXiv:2512.02218, ver.~2},
  year={2025}
}

@article {Rea23,
    AUTHOR = {Reading, N.},
     TITLE = {Dominance phenomena: mutation, scattering and cluster algebras},
   JOURNAL = {Trans. Amer. Math. Soc.},
  FJOURNAL = {Transactions of the American Mathematical Society},
    VOLUME = {376},
      YEAR = {2023},
    NUMBER = {2},
     PAGES = {773--835},
      ISSN = {0002-9947,1088-6850},
   MRCLASS = {13F60 (05E16 20F55 57Q15)},
  MRNUMBER = {4531662},
MRREVIEWER = {Fayadh\ Kadhem},
       DOI = {10.1090/tran/7888},
       URL = {https://doi.org/10.1090/tran/7888},
}

@article {CL20,
    AUTHOR = {Cao, P. and Li, F.},
     TITLE = {The enough {$g$}-pairs property and denominator vectors of cluster algebras},
   JOURNAL = {Math. Ann.},
  FJOURNAL = {Mathematische Annalen},
    VOLUME = {377},
      YEAR = {2020},
    NUMBER = {3-4},
     PAGES = {1547--1572},
      ISSN = {0025-5831,1432-1807},
   MRCLASS = {13F60 (05E40)},
  MRNUMBER = {4126901},
MRREVIEWER = {Francesco\ Esposito},
       DOI = {10.1007/s00208-020-02033-1},
       URL = {https://doi.org/10.1007/s00208-020-02033-1},
}

@article {CKLP13,
    AUTHOR = {Cerulli Irelli, G. and Keller, B. and Labardini-Fragoso, D. and Plamondon, P.-G.},
     TITLE = {Linear independence of cluster monomials for skew-symmetric cluster algebras},
   JOURNAL = {Compos. Math.},
  FJOURNAL = {Compositio Mathematica},
    VOLUME = {149},
      YEAR = {2013},
    NUMBER = {10},
     PAGES = {1753--1764},
      ISSN = {0010-437X,1570-5846},
   MRCLASS = {13F60 (18E30)},
  MRNUMBER = {3123308},
MRREVIEWER = {Olga\ Kravchenko},
       DOI = {10.1112/S0010437X1300732X},
       URL = {https://doi.org/10.1112/S0010437X1300732X},
}

@incollection {Yur23,
    AUTHOR = {Yurikusa, T.},
     TITLE = {Acyclic cluster algebras with dense {$g$}-vector fans},
 BOOKTITLE = {Mc{K}ay correspondence, mutation and related topics},
    SERIES = {Adv. Stud. Pure Math.},
    VOLUME = {88},
     PAGES = {437--459},
 PUBLISHER = {Math. Soc. Japan, Tokyo},
      YEAR = {[2023] \copyright 2023},
      ISBN = {978-4-86497-098-3},
   MRCLASS = {13F60},
  MRNUMBER = {4603594},
MRREVIEWER = {Truong\ Le\ Hoang},
}
\bibliographystyle{alpha}
\end{document}